\documentclass[11pt]{article}

\usepackage[utf8]{inputenc}

\usepackage{
amsmath,amssymb,amsfonts,amsthm} 
\usepackage{bm}
\usepackage{colonequals}
\usepackage{comment}
\usepackage{epsfig} \usepackage{latexsym,nicefrac,bbm}
\usepackage
{xspace}
\usepackage{color,fancybox,graphicx,subfigure,fullpage}
\usepackage[top=1.25in, bottom=1.25in, left=1in, right=1in]{geometry}
\usepackage{tabularx} 
\usepackage{hyperref} 
\usepackage{cleveref}
\usepackage{pdfsync}
\usepackage[boxruled,linesnumbered]{algorithm2e}
\usepackage{siunitx}

\usepackage{multicol}
\usepackage{enumitem}
\usepackage{float}
\usepackage{tikz}
\usepackage{tabularray}
\usepackage{framed}

\usepackage{authblk}

\makeatletter
\newcommand{\vast}{\bBigg@{4}}
\newcommand{\Vast}{\bBigg@{5}}
\makeatother

\newtheorem{theorem}{Theorem}[section]
\newtheorem{fact}{Fact}[section]
\newtheorem{conjecture}[theorem]{Conjecture}
\newtheorem{definition}[theorem]{Definition}
\newtheorem{lemma}[theorem]{Lemma}
\newtheorem{remark}[theorem]{Remark}
\newtheorem{proposition}[theorem]{Proposition}
\newtheorem{corollary}[theorem]{Corollary}

\newcommand{\EE}{\mathbb{E}}

\newcommand{\NN}{\mathbb{N}}

\newcommand{\RR}{\mathbb{R}}

\DeclareSymbolFont{bbold}{U}{bbold}{m}{n}
\DeclareSymbolFontAlphabet{\mathbbold}{bbold}
\newcommand{\One}{\mathbbold{1}}
\newcommand{\Var}{\mathrm{Var}}

\renewcommand{\emptyset}{\varnothing}

\newcommand{\sym}{\mathrm{sym}}

\newcommand{\eps}{\varepsilon}
\renewcommand{\epsilon}{\varepsilon}

\DeclareMathOperator{\Aut}{Aut}

\DeclareMathOperator{\tr}{tr}

\renewcommand{\epsilon}{\varepsilon}
\renewcommand{\Tilde}{\widetilde}
\renewcommand{\hat}{\widehat}
\newcommand{\Unif}{\mathrm{Unif}}

\newcommand{\Adv}{\mathrm{Adv}}

\newcommand{\sA}{\mathcal{A}}

\newcommand{\sQ}{\mathcal{Q}}
\newcommand{\sP}{\mathcal{P}}
\newcommand{\sG}{\mathcal{G}}
\newcommand{\sS}{\mathcal{S}}
\newcommand{\sC}{\mathcal{C}}
\newcommand{\sI}{\mathcal{I}}
\newcommand{\sL}{\mathcal{L}}
\newcommand{\sK}{\mathcal{K}}
\newcommand{\sO}{\mathcal{O}}
\newcommand{\sR}{\mathcal{R}}
\newcommand{\sW}{\mathcal{W}}

\newcommand{\Sym}{\mathrm{Sym}}

\newcommand{\TV}{\mathrm{TV}}
\newcommand{\KL}{\mathrm{KL}}

\newcommand{\Pois}{\mathrm{Pois}}
\newcommand{\Bin}{\mathrm{Bin}}
\newcommand{\Emb}{\mathrm{Emb}}

\newcommand{\stub}{\mathsf{stub}}
\newcommand{\originalretained}{\mathsf{original{-}retained}}
\newcommand{\originaldeleted}{\mathsf{original{-}deleted}}
\newcommand{\rematched}{\mathsf{rematched}}
\newcommand{\del}{\mathsf{del}}
\newcommand{\ret}{\mathsf{ret}}
\newcommand{\support}{\mathsf{support}}
\newcommand{\cyc}{\mathsf{cyc}}
\newcommand{\head}{\mathsf{head}}
\newcommand{\tail}{\mathsf{tail}}

\newcommand\numberthis{\addtocounter{equation}{1}\tag{\theequation}}

\newif\ifnotes
\notestrue

\title{\Large \begin{tabular}{c}
Analysis of Polynomial Threshold Functions on Random Regular \\
Graphs: Computational Complexity of Detecting Noisy Random Lifts
\end{tabular}}
\author{Xifan Yu\thanks{Email: \texttt{xifan.yu@yale.edu}. X.Y.~is partially supported by 
	ONR Award N00014-24-12611.}}

\affil{Department of Computer Science, Yale University}

\begin{document}

\pagenumbering{roman}

\maketitle

\thispagestyle{empty}

\begin{abstract}
    In this work, we present the first analysis of low-degree polynomial threshold functions for the natural hypothesis testing problem of detecting the noisy random lift of a base $d$-regular graph from a uniformly random $d$-regular graph. Along the way, we obtain a new result for the distribution of short cycle counts in noisy random lift up to logarithmic lengths, which generalizes results by McKay, Wormald, and Wysocka and by Johnson in the case of random regular graphs, and results by Greenhill, Janson, and Ruci\'{n}ski and by Fortin and Rudinsky in the case of random lifts.
\end{abstract}

\clearpage

\pagestyle{empty}

\tableofcontents

\clearpage

\pagestyle{plain}
\pagenumbering{arabic}

\allowdisplaybreaks{

\section{Introduction}

We study the natural hypothesis testing problem of distinguish a uniformly random $d$-regular graph on $n$ vertices, from a noisy random lift on $n$ vertices of some fixed base $d$-regular graph $H$.

This problem has received interest due to its connections to computational aspects of certifying properties of random regular graphs \cite{kunisky2024computational, nagda2025reinforced}, and falls outside of the usual settings where most techniques in average-case analysis apply. In particular, most existing methods for proving computational lower bounds rely on the existence of explicit orthonormal basis for product distributions and stop working on the distribution of random regular graphs due to correlation between the edges. Despite a few works that make progress towards this direction, past literature that studies computational aspects of random regular graphs is limited to settings of local algorithms for ``unplanted'' random optimization problem \cite{gamarnik2014limits, rahman2017local}, basic SDP relaxations \cite{banks2019lovasz, deshpande2019threshold} and constant degree local statistics hierarchy \cite{banks2021local, bandeira2021spectral, kunisky2024computational} for detection and refutation tasks, and constant degree sum-of-squares hierarchy \cite{mohanty2020lifting} for certification tasks. More recently, \cite{xu2024switching} analyzed non-constant degree polynomial relaxations on random regular graphs, but in a regime where the degree of the regular graph is non-constant. The corresponding question for constant degree random regular graphs remains open, and providing analysis of non-constant degree polynomials for constant degree random regular graph distributions is explicitly raised as an open question in both \cite{kunisky2024computational} and \cite{wein2025computational}.

In this work, we address this open problem and present a new method for analyzing low-degree polynomials and polynomial threshold functions on random regular graphs.

\subsection{Our Contribution}

Our motivation for analyzing the problem of detecting random lifts comes from the following conjecture, which has implications for the computational hardness for certifying expansion property, cut sizes, and independence number of random regular graphs \cite{kunisky2024computational}. We formalize the precise definitions of random regular graphs and noisy random lifts in Section~\ref{sec:prelim}.

\begin{conjecture}[Informal {\cite[Conjecture 1.6]{kunisky2024computational}}] \label{conj:informal}
    Suppose $H$ is a base $d$-regular graph on $k$ vertices. The following hold:
    \begin{itemize}
        \item If $H$ is Ramanujan, then there is no polynomial-time algorithm that achieves strong detection between $\sQ^S$ the distribution of random regular simple graphs, and $\sP^S$ the distribution of simple random lifts of $H$ subject to a small linear amount of noise.
        \item If $H$ is bipartite Ramanujan, then there is no polynomial-time algorithm that achieves strong detection between $\sQ^{S,b}$ the distribution of random regular bipartite simple graphs, and $\sP^{S,b}$ the distribution of simple bipartite random lifts of $H$ subject to a small linear amount of noise.
    \end{itemize}
\end{conjecture}

\begin{remark}
    As pointed out in \cite{kunisky2024computational}, the small amount of noise is crucial for the conjecture. If a noiseless random lift of a fixed base graph $H$ is drawn, then the noiseless lift inherits all the eigenvalues of $H$, which appear in the spectrum of a random regular graph with probability $0$. However, this test that looks for specific eigenvalues of $H$ is extremely brittle to noise, and the conjecture postulates that any efficient test fails to strongly distinguish a random lift of a Ramanujan $H$ from a random regular graph if the observed graph is perturbed by a small amount of noise.
\end{remark}

In this work, we provide further evidence for Conjecture~\ref{conj:informal} by establishing the following hardness result for polynomial threshold functions (Definition~\ref{def:PTF}).

\begin{theorem}[PTF Hardness; Informal]
     Suppose $H$ is a base $d$-regular graph on $k$ vertices. The following hold:
    \begin{itemize}
        \item If $H$ is Ramanujan, then for some $D = \Omega_H(\log (n))$, no degree-$D$ symmetric polynomial threshold function achieves strong detection between $\sQ^S$ the distribution of random regular simple graphs, and $\sP^S$ the distribution of simple random lifts of $H$ subject to a small linear amount of noise.
        \item If $H$ is bipartite Ramanujan, then for some $D = \Omega_H(\log (n))$, no degree-$D$ symmetric polynomial threshold function achieves strong detection between $\sQ^{S,b}$ the distribution of random regular bipartite simple graphs, and $\sP^{S,b}$ the distribution of simple bipartite random lifts of $H$ subject to a small linear amount of noise.
    \end{itemize}
\end{theorem}

We also establish the following result on the distribution of short cycles in noisy random lifts, which is a main ingredient for establishing the PTF hardness above.
\begin{theorem}[Cycle Distribution; Informal]
    Suppose $H$ is a base $d$-regular graph on $k$ vertices, and let $B$ be its non-backtracking matrix (see Definition~\ref{def:nb-matrix}). Then, for some $D = \Omega_H(\log (n))$, the following hold:
    \begin{itemize}
        \item The distribution of cycle counts up to length $D$ in the distribution $\sP$ of noisy random lift of $H$ converges in total variation distance to a product of independent Poisson distributions $\bigotimes_{t=1}^D \Pois(\mu_t)$, where $\mu_t \colonequals (1-\delta)^t \left(\nu_t - \frac{(d-1)^t}{2t}\right) + \frac{(d-1)^t}{2t}$, $\delta \colonequals \frac{2r}{nd}$, and $\nu_t \colonequals \frac{\tr(B^t)}{2t}$.
        \item If $H$ is bipartite, then the distribution of cycle counts up to length $2D$ in the distribution $\sP^b$ of noisy bipartite random lift of $H$ converges in total variation distance to a product of independent Poisson distributions $\bigotimes_{t=1}^D \Pois(\mu_{2t}')$, where $\mu_{2t}' \colonequals (1-\delta)^{2t} \left(\nu_{2t} - \frac{(d-1)^{2t}}{2t}\right) + \frac{(d-1)^{2t}}{2t}$, $\delta \colonequals \frac{r}{nd}$, and $\nu_{2t} \colonequals \frac{\tr(B^{2t})}{4t}$.
    \end{itemize}
\end{theorem}

\subsection{Related Works}

\paragraph{Low-Degree Polynomials and Polynomial Threshold Functions}

Since the introduction of Low-Degree Conjecture by Hopkins \cite{hopkins2018statistical}, there has been a substantial line of work analyzing low-degree polynomials as an attempt to establish average-case computational hardness, and it has been empirically observed that the low-degree method is very successful at predicting the computational phase transition for a great number of average-case problems. See \cite{kunisky2019notes} and \cite{wein2025computational} for surveys on low-degree methods. At a high level, the Low-Degree Conjecture posits the class of low-degree polynomials as being as powerful as any polynomial-time algorithm for solving ``natural'' hypothesis testing problems.

Recently, \cite{diakonikolas2025ptf} initiated the study of polynomial threshold functions, a more powerful class of algorithms compared to low-degree polynomial algorithms, in the setting of average-case complexity. Following this work, we study the class of polynomial threshold functions in the substantially different setting of random regular graph distributions, and rule out the class of low-degree symmetric polynomial threshold functions as evidence for Conjecture~\ref{conj:informal}, though our techniques differ from those in \cite{diakonikolas2025ptf}.

\paragraph{Cycle Distributions in Random Regular Graphs and Random Lifts}

There is a rich history of work in random graph theory community studying the statistical properties of random regular graphs and random lifts. One of the most well-studied problems in this area is the distribution of short cycles in random regular graphs. The classical work of \cite{wormald1981asymptotic} established the asymptotic convergence of fixed-length cycle counts to independent Poisson distributions in random regular graphs. \cite{mckay2004short} extended this result and established the convergence of short cycle counts up to logarithmic lengths in random regular graphs to a product of independent Poisson distributions using a switching argument, and it was later improved in \cite{Johnson_2015} that the convergence continues to hold for a larger constant in front of the logarithm. On the side of (noiseless) random lifts, \cite{greenhill2010number} and \cite{fortin2013asymptotic} showed the convergence of short cycle counts of constant lengths to a product of independent Poisson distributions by a method of factorial moments.

Our work studies a different new model of noisy random lift (Definition~\ref{def:random-lift} and Definition~\ref{def:bipartite-random-lift}), which covers the settings of both random regular graphs and noiseless random lifts by varying the noise parameter. We generalize the previous literature and establish the convergence of short cycle counts up to logarithmic lengths to a product of independent Poisson variables, though our proof using the Chen-Stein method for Poisson approximation is significantly more involved.

\paragraph{Spectra of Random Regular Graphs and Random Lifts}

Another related line of work that has been extensively studied is the spectra of random regular graphs and random lifts. Friedman \cite{friedman2008proof} showed that random regular graphs are near Ramanujan with high probability, resolving Alon's second eigenvalue conjecture \cite{alon1986eigenvalues}. A recent breakthrough \cite{huang2024ramanujan} by Huang, McKenzie, and Yau established the convergence of the second eigenvalue and the last eigenvalue of a random regular graph to the Tracy-Widom distribution, showing that asymptotically $69\%$ of random regular graphs are Ramanujan.

For random lifts, it is well-known that their spectra consist of two parts, the collection of ``old'' eigenvalues coming from the base graph and the collection of ``new'' eigenvalues. \cite{charles2020new} proved a statement analogous to second eigenvalue conjecture for the ``new'' eigenvalues of random lifts, and \cite{bordenave2019eigenvalues} showed that the collection of the ``new'' eigenvalues converges in Hausdorff distance to the spectrum of the universal cover of $H$, which is $[-2\sqrt{d-1}, 2\sqrt{d-1}]$ when $H$ is $d$-regular.

Our work draws a interesting connection between the spectra of the base graph $H$ in the noisy random lift, the distribution of short cycle counts, and the computational complexity of detecting noisy random lifts. We show that there is a sharp transition for the distributions of the short cycle counts and the symmetric PTF degree needed to strongly detect the noisy random lift between the case of Ramanujan base graph $H$ and the case of non-Ramanujan $H$.

\subsection{Open Questions}

We believe our work opens a few interesting directions.

\begin{enumerate}
    \item Can one establish computational hardness for the problem of detecting noisy random lifts (or other natural hypothesis testing problems involving random regular graphs) against higher degree polynomials and PTFs? Can the degree be improved to $\omega(\log(n))$ or $n^{\Omega(1)}$? Any potential argument must necessarily move away from assuming symmetric low-degree polynomials are functions of the cycle counts with high probability, as bicycles (see Lemma~\ref{lem:rare-event}, Remark~\ref{rem:bicycles}, and Lemma~\ref{lem:function-of-cycles}) start to emerge at super-logarithmic scale.
    \item Can one establish higher degree sum-of-squares lower bound or local-statistical lower bound for problems involving constant degree random regular graphs?
    \item Is there an algorithm that achieves strong detection between random regular graphs and noisy random lifts, which runs faster than in time $\exp(\Omega(n))$?
\end{enumerate}

\section*{Acknowledgements}

The author would like to thank Dmitriy Kunisky for introducing this problem and many helpful discussions on it.

GPT-5.4 Pro was used to identify the Chen-Stein method for proving Poisson approximation for Theorem~\ref{thm:cycle-counts}. The detailed analysis and the underlying coupling construction for the analysis of $b_3$ in Section~\ref{sec:b3-bound} is obtained in multiple interactive sessions with GPT-5.4 Pro, GPT-5.5 Pro, and GPT-5.6 Sol Pro, and the counting argument in Section~\ref{sec:statistical} is assisted by GPT-5.6 Sol Extra High. 

\section{Preliminaries}\label{sec:prelim}

We denote $[n] = \{1, 2, \dots, n\}$. We will use standard asymptotic notations $O(\cdot ), o(\cdot ), \Omega(\cdot ), \omega(\cdot )$, where the limits are taken as $n \to \infty$.

Throughout, graphs are undirected. A multigraph $G$ consists of a vertex set $V$ and an edge set $E$ that is a multiset of unordered pairs of vertices. An edge $\{v,v\} \in E(G)$ with the same endpoints is called a loop. Multiple edges in $E(G)$ that share the same endpoints $u$ and $v$ are called parallel edges. A multigraph without loops or parallel edges is called simple. Unless specified otherwise, graphs always refer to multigraphs. The adjacency matrix of a graph $G = (V, E)$ is a matrix indexed by $V$, whose diagonal entries at $(v,v)$ equals twice the number of loops at vertex $v$, and off-diagonal entries at $(u,v)$ equals the number of edges between vertices $u$ and $v$.

The degree of a vertex $v \in V(G)$ is equal to the sum of the contribution coming from its incident edges. Every loop attached to $v$ contributes degree $2$ and every edge with exactly one endpoint $v$ contributes degree $1$. We note that with this convention, the sum of the degrees of all the vertices in $G$ equals twice the number of edges in $G$. A graph $G$ is $d$-regular if all the vertices have degree $d$. A graph $G$ is bipartite if there exists a bipartition of the vertex set, so that all the edges of $G$ cross this partition.

For a subset of vertices $S \subseteq V(G)$, the induced subgraph of $G$ on $S$ is the graph with vertex set $S$ and edge set containing all the edges of $G$ with both endpoints in $S$. For two disjoint $S, T \subseteq V(G)$, the induced bipartite subgraph of $G$ on $S \sqcup T$ is the bipartite graph with bipartition of the vertex set $S \sqcup T$, and edge set containing all the edges of $G$ with one endpoint in $S$ and the other endpoint in $T$.

In the definition of graph lift below, we also allow a special type of edge called half-loop. Having identical endpoints, a half-loop is similar to a loop, but only contributes degree $1$ unlike a loop. Moreover, a half-loop at $v$ only contributes $1$ to the diagonal entry at $(v,v)$ of the adjacency matrix. In this work, half-loops will only appear in the base graphs of graph lifts. For a base graph $H$, we will use $E_{1}(H)$ to denote its multiset of half-loops, and $E_2(H)$ to denote its multiset of edges, so that $E(H) = E_1(H) \sqcup E_2(H)$. 

\begin{definition}[Graph Lift]
    Let $H$ be a graph on vertex set $[k]$ with adjacency matrix $M$, with $M_{i,i}$ half-loops attached to vertex $i \in [k]$, and $M_{i,j}$ parallel edges between distinct vertices $i,j \in [k]$. A graph $G = (V, E)$ is an $m$-lift of $H$ if $|V| = km$ and there exists a balanced partition $\sigma: V \to [k]$ (i.e., having $|\sigma^{-1}(i)| = m$ for every $i \in [k]$) such that
    \begin{itemize}
        \item For every $i \in [k]$, $\sigma^{-1}(i)$ induces a $M_{i,i}$-regular graph on $G$.
        \item For every pair $\{i,j\} \in \binom{[k]}{2}$, $\sigma^{-1}(i) \sqcup \sigma^{-1}(j)$ induces a $M_{i,j}$-regular bipartite graph on $G$, with bipartition $\sigma^{-1}(i) \sqcup \sigma^{-1}(j)$.
    \end{itemize}
\end{definition}

\begin{definition}[Base Graph and Fiber]
    In the context of graph lift as above, we refer to $H$ as the base graph, and $\sigma^{-1}(i) \subseteq V(G)$ as the fiber of $i \in V(H)$. 
\end{definition}

\subsection{Configuration Model and Noisy Random Lift}

Now, we define the random regular graph distributions that we study in this paper.

\begin{definition}[Configuration Model]
    Let $n,m \in \NN$ and $d_1, d_2, \dots, d_n \in \NN$ be a degree sequence such that $\sum_{i=1}^n d_i = 2m$. The configuration model draws a random graph on $n$ vertices and $m$ edges with the given degree sequence as follows:
    \begin{itemize}
        \item Create $d_i$ stubs (or half edges) for every $i \in [n]$.
        \item Draw a perfect matching of the $2m$ stubs uniformly at random.
        \item Connect the stubs paired by the matching to form edges.
    \end{itemize}
\end{definition}

\begin{definition}[Bipartite Configuration Model]
    Let $n_1, n_2,m \in \NN$ and $\ell_1, \ell_2, \dots, \ell_{n_1}, r_1, r_2, \dots, r_{n_2} \in \NN$ be a degree sequence such that $\sum_{i=1}^{n_1} \ell_i = \sum_{i=1}^{n_2} r_i  = m$. The bipartite configuration model draws a random bipartite graph on $n_1 + n_2$ vertices and $m$ edges with the degrees of one part given by $\ell_i$'s and the degrees of the other part given by $r_i$'s as follows:
    \begin{itemize}
        \item Create $\ell_i$ L-stubs for every $i \in [n_1]$ and $r_i$ R-stubs for every $i \in [n_2]$.
        \item Draw a perfect matching between the $m$ L-stubs and the $m$ R-stubs uniformly at random.
        \item Connect the stubs paired by the matching to form edges.
    \end{itemize}
\end{definition}

\begin{definition}[Random Regular Graph]
    A random $d$-regular graph is generated from the configuration model defined above. A random $d$-regular simple graph is generated from the configuration model conditioned on the sampled graph is simple.
\end{definition}

\begin{definition}[Random Bipartite Regular Graph]
    A random bipartite $d$-regular graph is generated from the bipartite configuration model defined above. A random bipartite $d$-regular simple graph is generated from the bipartite configuration model conditioned on the sampled graph is simple.
\end{definition}

\begin{definition}[Noisy Random Lift]\label{def:random-lift}
    Let $H$ be a base non-bipartite $d$-regular graph on vertex set $[k]$ potentially with half-loops. Let $M$ be the adjacency matrix of $H$. Let $m \in \NN$ and $r \in \left\{0, 1, \dots, \frac{nd}{2}\right\}$ where $n \colonequals mk$. Assume that $m$ is even\footnote{This condition is assumed so that perfect matchings exist on the fiber $V_i$ for every $i \in [k]$, and the noisy random lift model is well-defined.} for every $i \in [k]$. An $r$-noisy random $m$-lift of $H$ is generated as follows:
    \begin{itemize}
        \item For every $i \in [k]$, create a fiber $V_i$ of $m$ vertices.
        \item For every edge $e \in E(H)$ connecting two distinct vertices $i \ne j$, independently draw a bijection $\sigma_e: V_i \to V_j$ uniformly at random, and add $m$ edges between $V_i$ and $V_j$ according to $\sigma_e$.
        \item For every half-loop $e \in E(H)$ attached to vertex $i$, independently draw a perfect matching\footnote{By perfect matching, we mean a bijection $\sigma_e: V_i \to V_i$ such that $\sigma_e(v) \ne v$ and $\sigma_c^2(v) = v$ for every $v \in V_i$.} $\sigma_e: V_i \to V_i$ uniformly at random, and add $\frac{m}{2}$ edges on $V_i$ according to $\sigma_e$.
    \end{itemize}
        So far, a random $d$-regular graph on $V \colonequals \bigsqcup_{i=1}^k V_i$ is created, which is an $m$-lift of $H$. Next, a random subset of $r$ edges gets deleted and rematched to create the final $r$-noisy random $m$-lift:
    \begin{itemize}
        \item Select an $r$-subset of the $\frac{nd}{2}$ formed edges uniformly at random and delete these edges, exposing a total of $2r$ stubs $\mathcal{R}$.
        \item Draw a perfect matching $\sigma_{\sR}$ of the stub collection $\mathcal{R}$ uniformly at random, and add $r$ edges back according to $\sigma_{\sR}$.
    \end{itemize}
\end{definition}

\begin{definition}[Noisy Bipartite Random Lift]\label{def:bipartite-random-lift}
    Let $H$ be a base bipartite $d$-regular graph on vertex set $[2k]$, whose bipartition is $\{1, \dots, k\} \sqcup \{k + 1, \dots, 2k\}$. Let $M$ be the adjacency matrix of $H$. Let $m \in \NN$ and $r \in \{0, 1, \dots, nd\}$ where $n \colonequals mk$. An $r$-noisy bipartite random $m$-lift of $H$ is generated as follows:
    \begin{itemize}
        \item For every $i \in [2k]$, create a fiber $V_i$ of $m$ vertices.
        \item For every edge $e \in E(H)$ connecting two distinct vertices $i \ne j$, independently draw a bijection $\sigma_e: V_i \to V_j$ uniformly at random, and add $m$ edges between $V_i$ and $V_j$ according to $\sigma_e$.
    \end{itemize}
        So far, a random bipartite $d$-regular graph on $V \colonequals \bigsqcup_{i=1}^{2k} V_i$ is created with bipartition $(\bigsqcup_{i=1}^k V_i) \sqcup (\bigsqcup_{i=k+1}^{2k} V_i)$, which is an $m$-lift of $H$. Next, a random subset of $r$ edges gets deleted and rematched to create the final $r$-noisy bipartite random $m$-lift:
    \begin{itemize}
        \item Select an $r$-subset of the $nd$ formed edges uniformly at random and delete these edges, exposing a total of $2r$ stubs partitioned into $\mathcal{R}_1 \sqcup \mathcal{R}_2$ according to the bipartition $(\bigsqcup_{i=1}^k V_i) \sqcup (\bigsqcup_{i=k+1}^{2k} V_i)$.
        \item Draw a bijection $\sigma_{\sR}: \mathcal{R}_1 \to \mathcal{R}_2$ uniformly at random, and add $r$ edges back according to $\sigma_{\sR}$.
    \end{itemize}
\end{definition}

\begin{remark}[Stub Viewpoint and Oriented Edge Set] \label{rem:directed-edges}
    Note that we adopt the following stub (half-edge) viewpoint of the noisy random lift model. For convenience, we work with the following directed version $\vec{E}(H)$ of the edge set of $H$, defined as follows.
    \begin{itemize}
        \item For every half-loop $e \in E(H)$ attached to vertex $i$, we add one directed edge $(i, e)$ to $\vec{E}(H)$, which is understood as a directed edge with tail at $i$ and head at $i$, indexed by $e$ \footnote{We keep the undirected edge $e$ in the notation because the base graph $H$ need not be simple, so between the same pair of endpoints, there could be multiple edges}.
        \item For every edge $e \in E(H)$ connecting two distinct vertices $i \ne j$, we add two directed edges $(i,e)$ and $(j,e)$ to $\vec{E}(H)$, where $(i,e)$ is understood as a directed edge with tail at $i$ and head at $j$, indexed by $e$, and $(j,e)$ is understood similarly as a directed edge with tail at $j$ and head at $i$, indexed by $e$.
    \end{itemize}
    We use $\iota: \vec{E}(H) \to \vec{E}(H)$ to denote the map that sends any directed edge to its reverse. More precisely, if $e$ is a half-loop attached to $i$, then $\iota(i,e) = (i,e)$; if $e$ is an edge connected distinct $i \ne j$, then $\iota(i,e) = (j,e)$. We further use $\tail(i,e) \colonequals i$ to denote the tail of a directed edge, and $\head(i,e) \colonequals \tail(\iota(i,e))$ to denote the head of a directed edge.
    
    For every edge $e \in E(H)$ connecting two distinct vertices $i,j \in [k]$, we attach a stub of type $(i,e) \in \vec{E}(H)$ attached to every $v \in V_i$ and a stub of type $(j,e) \in \vec{E}(H)$ attached to every $u \in V_j$. Then, adding $m$ edges between $V_i$ and $V_j$ according to $\sigma_e$ is equivalently obtained by pairing the $m$ stubs of type $(i,e)$ at $V_i$ and the $m$ stubs of type $(j,e)$ at $V_j$ uniformly at random.

    Similarly, for every half-loop $e \in E(H)$ at vertex $i \in [k]$, we attach a stub of type $(i,e) \in \vec{E}(H)$ to every $v \in V_i$. Adding $\frac{m}{2}$ edges on $V_i$ according to $\sigma_e$ is equivalently obtained by drawing a perfect matching of the $m$ stubs of type $(i,e)$ at $V_i$ uniformly at random.

    Using this stub viewpoint, every edge in the noisy random lift is a stub-pair, for the stubs described above.
\end{remark}

\begin{remark}
    In the models of noisy random lifts defined above, the underlying fiber partition of the vertex set $V$ is treated as a uniformly random equitable partition, i.e., we apply a uniform random relabelling to hide the vertex labels.
\end{remark}

\begin{remark}
    If $r = 0$, then the $r$-noisy random lift model defined above corresponds to the classical random lift model \cite{amit2001random}. If $r = \frac{nd}{2}$ (or $nd$ in the bipartite case), then the $r$-noisy random lift model corresponds to random (bipartite) regular graph distribution.
\end{remark}

\subsection{Computational Complexity of Hypothesis Testing}

Let $\sQ = \sQ_n$ and $\sP = \sP_n$ be two sequences of distributions supported on $\RR^N$ where $N = N(n)$ is a polynomial in $n$. The hypothesis testing problem between $\sQ$ and $\sP$ asks for a function $f: \RR^N \to \{0, 1\}$ such that $f$ distinguishes $\sQ$ and $\sP$ in the following sense of strong detection.

\begin{definition}[Strong Detection]
    $f: \RR^N \to \{0, 1\}$ is said to achieve strong detection between $\sQ$ and $\sP$ if as $n \to \infty$,
    \begin{align*}
        \Pr_{Y \sim \sQ}(f(Y) = 1) + \Pr_{Y \sim \sP}(f(Y) = 0) = o(1).
    \end{align*}
\end{definition}

If one does not restrict the complexity of computing such an $f$, then the problem becomes purely statistical and the optimal function is given by the likelihood ratio test \cite{neyman1992problem}. However, for many high-dimensional hypothesis testing problems, computing or even approximating the likelihood ratio is extremely challenging, and it has been observed that the best-known polynomial time algorithms often do not achieve the information-theoretic limits. It is thus a natural question to study when efficient algorithms can solve a hypothesis testing problem.

Since the classical theory of NP-hardness does not apply to an average-case setting like hypothesis testing, researchers have resorted to other ways of providing evidence of computational hardness for average-case problems, with canonical examples including the Planted Clique (see \cite{jerrum1992large, alon1998finding, barak2019nearly}) and the sparse PCA (see \cite{berthet2013optimal, hopkins2017power, brennan2018reducibility}). One method that has been widely adopted is to show that concrete classes of powerful algorithms fail to solve a problem. Among a variety of classes of algorithms, low-degree polynomials have received special attention due to the Low-Degree Conjecture \cite{hopkins2018statistical}, which posits that for most natural high dimensional hypothesis testing problem, the best low-degree polynomial has the same performance as the optimal polynomial-time algorithm.

Past works that prove hardness for low-degree polynomials mostly focused on bounding the so-called low-degree advantage, defined as follows.
\begin{definition}[Low-Degree Advantage]
    Let $\sQ$ and $\sP$ be two distributions supported on $\RR^N$. The degree-$D$ advantage between $\sQ$ and $\sP$ is
    \begin{align*}
        \Adv_{\le D}(\sQ, \sP) \colonequals \sup_{\substack{f \in \RR[Y]_{\le D}:\\ \EE_{\sQ}[f(Y)^2] \ne 0}} \frac{\EE_{\sP}[f(Y)]}{\sqrt{\EE_{\sQ}[f(Y)^2]}}.
    \end{align*}
\end{definition}

\begin{conjecture}[Low-Degree Conjecture {\cite{hopkins2018statistical}}]
    For ``nice'' sequences of distributions $\sP = \sP_n$ and $\sQ = \sQ_n$ supported on $\RR^N$ for some $N = N(n)$ that is a polynomial in $n$, if for some $D \ge \log(n)^{1 + \eps}$ with $\eps > 0$ fixed, the degree-$D$ advantage between $\sP$ and $\sQ$ is bounded as $n \to \infty$, i.e., $\Adv_{\le D}(\sQ, \sP) = O(1)$, then no polynomial time algorithm achieves strong detection between $\sQ$ and $\sP$.
\end{conjecture}

The Low-Degree Conjecture has proved to be successful for predicting the computational complexity of many hypothesis testing problems, but its rigorous implications so far remain somewhat limited \cite{hsieh2026rigorous, yu2025counting, bangachev2025graph, chen2026rigorous}. Furthermore, recent works \cite{buhai2025quasi, mao2026polynomial} have revealed that more structural assumptions on the hypothesis testing problem are needed in order for the conjecture to hold. Nevertheless, we still treat Low-Degree Conjecture as a heuristic conjecture that is expected to apply to many natural problems, and consider hardness against low-degree polynomials to be convincing evidence for the computational hardness of hypothesis testing problems.

One rigorous consequence of bounding the low-degree advantage is that it rules out low-degree polynomials that achieve a notion called strong separation.

\begin{definition}[Strong Separation]
    A function $f: \RR^N \to \RR$ is said to achieve strong separation between $\sQ$ and $\sP$ if as $n \to \infty$,
    \begin{align*}
        \left|\EE_{\sP}[f(Y)] - \EE_{\sQ}[f(Y)]\right| = \omega\left(\max\left\{\sqrt{\Var_{\sP}[f(Y)]}, \sqrt{\Var_{\sP}[f(Y)]}\right\}\right).
    \end{align*}
\end{definition}

\begin{proposition}[{\cite[Lemma 7.3]{coja2022statistical}}]
    Let $\sQ$ and $\sP$ be two sequences of distributions supported on $\RR^N$. If $\Adv_{\le D}(\sQ, \sP) = O(1)$, then no sequence of polynomials of degree at most $D$ strongly separates $\sQ$ and $\sP$.
\end{proposition}

\begin{remark}\label{rem:separation-implies-detection}
    Note that the strong separation is a natural notion that implies strong detection. If some function $f$ achieves strong separation, then thresholding it at some value achieves strong detection by Chebyshev's inequality. However, the reverse direction is not true in general.
\end{remark}

In this work, we study a more general class of algorithms called the low-degree polynomial threshold functions. 

\begin{definition}[Polynomial Threshold Function] \label{def:PTF}
    A degree-$D$ polynomial threshold function (PTF) is specified by degree-$D$ polynomial $f$, whose output equals to the sign function of the polynomial output. In other words, the PTF outputs $1$ if $f \ge 0$, and $0$ if $f < 0$.
\end{definition}

We view the failure of low-degree polynomial threshold functions as stronger evidence for the corresponding hypothesis testing problem, as they are strictly more powerful than the class of low-degree polynomials.

\section{Main Results}
In this section, we present our main results. Fix $d \ge 3$. Let $\sQ = \sQ_n$ be the distribution of random $d$-regular graphs on $n$ vertices, and $\sQ^{b} = \sQ^{b}_n$ be the distribution of random bipartite $d$-regular graphs on $2n$ vertices. Let $\delta \in [0,1]$ be a constant. For $H$ a fixed base $d$-regular graph on $k$ vertices potentially with half-loops, let $\sP = \sP_n(H, \delta)$ be the distribution of $r$-noisy random $m$-lift of $H$ where $r \colonequals \lceil \delta \frac{nd}{2} \rceil$ and $n \colonequals mk$. For $H$ a fixed base bipartite $d$-regular graph on $2k$ vertices with a given bipartition, let $\sP^{b} = \sP^{b}_n(H, \delta)$ be the distribution of $r$-noisy bipartite random $m$-lift of $H$ where $r \colonequals \lceil \delta nd \rceil$ and $n \colonequals mk$. Let $S$ denote the collection of simple graphs. Let $\sQ^S, \sQ^{b,S}, \sP^S, \sP^{b,S}$ denote the corresponding distributions conditioned on the graph being simple. 

In this work, we always assume the base graph $H$ is connected, as otherwise for small enough $\delta$, the noisy random lift distribution is easy to distinguish from the random regular graph distribution, by computing the second largest eigenvalue for example.

\subsection{PTF Hardness for Detecting Noisy Random Lift of Ramanujan Graph}

Our first main result is the following PTF hardness for detecting noisy random lift of a Ramanujan base graph. Recall the definition of Ramanujan graph.

\begin{definition}[Ramanujan Graph]
    Fix $d \ge 3$. Let $H$ be a connected $d$-regular graph on $k$ vertices and $M$ be its adjacency matrix. Suppose $d = \lambda_1(M) > \lambda_2(M) \ge \dots \ge \lambda_k(M)$ are the eigenvalues of $M$.
    
    We say $H$ is Ramanujan if $\max_{\lambda_i(M) \ne d}\{|\lambda_i(M)|\} \le 2\sqrt{d-1}$.

    We say $H$ is bipartite Ramanujan if $H$ is bipartite and $\max_{|\lambda_i(M)| \ne d}\{|\lambda_i(M)|\} \le 2\sqrt{d-1}$.
\end{definition}

\begin{theorem}[PTF Hardness]\label{thm:ptf-hardness}
    Fix $d \ge 3$ and $\delta > 0$.
    
    Suppose $H$ is a base $d$-regular Ramanujan graph on $k$ vertices potentially with half-loops. Then, there exists a constant $c = c(H) > 0$ such that for any $D \le c\log(n)$, no symmetric\footnote{A symmetric PTF is specificed by a symmetric polynomial $f$ in the graph variables, in the sense that $f$ is invariant under permutation of vertex labels. See Definition~\ref{def:sym-PTF} for more details. We believe this is a natural requirement, in that it assumes the output of the PTF and the value of the underlying polynomial is a function of the isomorphism class of the sampled graph and does not depend on the specific labellings of the vertices.} PTF of degree at most $D$ achieves strong detection between $\sQ^S$ and $\sP^S$.

    Suppose $H$ is a base bipartite $d$-regular Ramanujan graph on $2k$ vertices with a given bipartition. Then, there exists a constant $c = c(H) > 0$ such that for any $D \le c\log(n)$, no symmetric PTF of degree at most $D$ achieves strong detection between $\sQ^{b,S}$ and $\sP^{b,S}$.
\end{theorem}

As a corollary, the same hardness also applies to the multigraph setting, since the event $S$ has probability bounded away from $0$ under $\sQ, \sQ^b, \sP$ and $\sP^b$ by Corollary~\ref{cor:simple-prob}.

\begin{corollary}\label{cor:multigraph-hardness}
    The PTF hardness above also applies to multigraph distributions $\sQ, \sP$ and $\sQ^b, \sP^b$, 
\end{corollary}

By Remark~\ref{rem:separation-implies-detection} and a standard Jensen's inequality argument, the failure of degree-$D$ symmetric PTFs also yields hardness for degree-$D$ polynomials.

\begin{corollary}\label{cor:low-deg-hardness}
    If no symmetric PTF of degree at most $D$ achieves strong detection between $\sQ^S$ and $\sP^S$ ($\sQ^{b,S}$ and $\sP^{b,S}$, respectively), then no polynomial of degree at most $D$ achieves strong separation between $\sQ^S$ and $\sP^S$ ($\sQ^{b,S}$ and $\sP^{b,S}$, respectively).
\end{corollary}

The proofs of Theorem~\ref{thm:ptf-hardness}, Corollary~\ref{cor:multigraph-hardness}, and Corollary~\ref{cor:low-deg-hardness} are deferred to Section~\ref{sec:proof-ptf-hardness}.

\subsection{Distribution of Short Cycle Counts in Noisy Random Lift}

One important ingredient in our proof of Theorem~\ref{thm:ptf-hardness} is the joint convergence of the distribution of short cycle counts to independent Poisson distributions under both $\sQ$ and $\sP$. Our second main result proves such a convergence statement, generalizing a classical result \cite{mckay2004short} by McKay, Wormald, and Wysocka for the random regular simple graph distribution, later improved in \cite{Johnson_2015} by Johnson, which shows the joint convergence to Poisson up to logarithmic lengths, and \cite{greenhill2010number} by Greenhill, Janson, and Ruci\'{n}ski and \cite{fortin2013asymptotic} by Fortin and Rudinsky for the case of random lift distribution, which show the joint convergence to Poisson up to constant lengths.

We will use the following definition of cycles. Let $G$ be a graph with potentially loops and parallel edges. A cycle of length $t$ in $G$ is an interleaving sequence of $t$ distinct vertices and $t$ distinct edges of $G$, $(v_1, e_1, v_2, e_2, \dots, v_{t}, e_{t})$, such that for all $i\in [t]$, the edge $e_i$ has endpoints $v_i$ and $v_{i+1}$ where we let $v_{t + 1} \colonequals v_1$. A cycle is determined up to a cyclic permutation and reversal. Note that a cycle of length $1$ is a loop, and a cycle of length $2$ is a pair of parallel edges.

To state the result, we will need the following notion of the non-backtracking matrix, sometimes called the Hashimoto operator \cite{hashimoto1989zeta, hashimoto1990zeta}, which has been widely studied in the analysis of spectra of random regular graphs.
\begin{definition}[Non-Backtracking Matrix]\label{def:nb-matrix}
    Fix $d \ge 3$. Suppose $H$ is a $d$-regular graph on $k$ vertices potentially with half-loops, and $M$ is the adjacency matrix of $H$. Then, the non-backtracking matrix $B$ of $H$ is a matrix indexed by the set of directed edges $\vec{E}(H)$ discussed in Remark~\ref{rem:directed-edges}, whose entries are given by
    \begin{align*}
        B_{\vec{e}_1,\vec{e}_2} = \One\{\tail(\vec{e}_2) = \head(\vec{e}_1), \vec{e}_2 \ne \iota(\vec{e}_1)\},
    \end{align*}
    where we recall $\iota: \vec{E}(H) \to \vec{E}(H)$ maps any directed edges to its reverse.
\end{definition}

Our second main result is the following convergence result for short cycles in noisy random lifts, the proof of which is deferred to Section~\ref{sec:proof-cycle-counts}.
\begin{theorem}[Distribution of Short Cycle Counts in Noisy Random Lift]\label{thm:cycle-counts}
    Fix $d \ge 3$.
    
    Suppose $H$ is a base $d$-regular graph on $k$ vertices potentially with half-loops. Let $B$ be the non-backtracking matrix of $H$. Let $m \in \NN$, $n \colonequals mk$, and $r \in \left\{0, 1, \dots, \frac{nd}{2}\right\}$. Let $Z_t$ denote the number of length-$t$ cycles in an $r$-noisy random lift of $H$ on $n$ vertices drawn from $\sP$. Then, there exists a constant $c = c(H) > 0$ such that uniformly for all $D \le c\log(n)$,
    \begin{align*}
        d_{\TV}\left(\sL(Z_1, Z_2, Z_3, \dots, Z_D),  \bigotimes_{t=1}^D \Pois(\mu_t)\right) = o(1), \quad \text{ as } n \to \infty,
    \end{align*}
    where $\mu_t \colonequals (1-\delta)^t \left(\nu_t - \frac{(d-1)^t}{2t}\right) + \frac{(d-1)^t}{2t}$, $\delta \colonequals \frac{2r}{nd}$, and $\nu_t \colonequals \frac{\tr(B^t)}{2t}$.

    Suppose $H$ is a base bipartite $d$-regular graph on $2k$ vertices with a given bipartition. Let $B$ be the non-backtracking matrix of $H$. Let $m \in \NN$, $n \colonequals mk$, and $r \in \{0, 1, \dots, nd\}$. Let $Z_t$ denote the number of length-$t$ cycles in an $r$-noisy bipartite random lift of $H$ on $2n$ vertices drawn from $\sP^b$. Then, there exists a constant $c = c(H) > 0$ such that uniformly for all $L \le c\log(n)$,
    \begin{align*}
        d_{\TV}\left(\sL(Z_2, Z_4, \dots, Z_{2L}),  \bigotimes_{t=1}^L \Pois(\mu_{2t}')\right) = o(1), \quad \text{ as } n \to \infty,
    \end{align*}
    where $\mu_{2t}' \colonequals (1-\delta)^{2t} \left(\nu_{2t} - \frac{(d-1)^{2t}}{2t}\right) + \frac{(d-1)^{2t}}{2t}$, $\delta \colonequals \frac{r}{nd}$, and $\nu_{2t} \colonequals \frac{\tr(B^{2t})}{4t}$.
\end{theorem}

As a corollary, both noisy random lift and random regular graph have a constant probability being simple, since under both models the number of loops and the number of pairs of parallel edges ($Z_1$ and $Z_2$ in Theorem~\ref{thm:cycle-counts} resp.) jointly converge to a product of Poisson random variables of constant means.

\begin{corollary}\label{cor:simple-prob}
    Fix $d \ge 3$ and $\delta \in [0,1]$. 
    
    Suppose $H$ is a base $d$-regular graph on $k$ vertices potentially with half-loops. Then, there exists a constant $p = p(H) > 0$ such that 
    \begin{align*}
        \liminf_{n \to \infty}\Pr_{G \sim \sP_n(H, \delta)}\left(G \text{ is simple}\right) \ge p.
    \end{align*}

    Suppose $H$ is a base bipartite $d$-regular graph on $2k$ vertices with a given bipartition. Then, there exists a constant $p = p(H) > 0$ such that 
    \begin{align*}
        \liminf_{n \to \infty}\Pr_{G \sim \sP^b_n(H, \delta)}\left(G \text{ is simple}\right) \ge p.
    \end{align*}
\end{corollary}

\subsection{Detecting Noisy Random Lift of Non-Ramanujan Graph is Easy}

To contrast the lower bound for the Ramanujan case, our next result shows that noisy random lift of non-Ramanujan base graph can be strongly detected with PTFs of barely non-constant degrees. We defer its proof to Section~\ref{sec:non-ramanujan}.
\begin{theorem}[PTF for Detecting Noisy Non-Ramanujan Lift] \label{thm:non-ramanujan}
    Fix $d \ge 3$.

    Suppose $H$ is a base $d$-regular non-Ramanujan graph on $k$ vertices potentially with half-loops. Then, there exists a constant $\Delta = \Delta(H) > 0$ such that for any $D = \omega(1)$, there exists a PTF of degree $D$ that achieves strong detection between $\sQ^S$ and $\sP^S = \sP^S_n(H, \delta)$ for any $\delta \le \Delta$.

    Suppose $H$ is a base bipartite $d$-regular non-Ramanujan graph on $2k$ vertices with a given bipartition. Then, there exists a constant $\Delta = \Delta(H) > 0$ such that for any $D = \omega(1)$, there exists a PTF of degree $D$ that achieves strong detection between $\sQ^{b,S}$ and $\sP^{b,S} = \sP^{b,S}_n(H, \delta)$ for any $\delta \le \Delta$.
\end{theorem}

We also note that any symmetric PTF of constant degree cannot achieve strong detection, as explained in Remark~\ref{rem:necessity-of-nonconstant-deg}. Thus, Theorem~\ref{thm:non-ramanujan} is optimal in terms of the degree dependency.

\begin{remark}
    We state the result above to pin down the exact PTF complexity of detecting noisy random lift of non-Ramanujan graphs.

    In terms of computational complexity, it is already known in \cite{kunisky2024computational} that the noisy random lift of non-Ramanujan graphs can be strong distinguished in polynomial time, by solving the feasibility of the semidefinite programs coming from the degree-$(2, O(1))$ local statistic hierarchy.
\end{remark}

\subsection{Statistical Distinguishability of Noisy Random Lift}

Finally, we complement our average-complexity results with a statistical result, showing that intuitively, noisy random lifts is distinguishable from random regular graphs. The proof is deferred to Section~\ref{sec:statistical}.

\begin{theorem}[Statistical Distinguishability]\label{thm:statistical}
    Fix $d \ge 3$.

    Suppose $H$ is a base $d$-regular graph that is not a single vertex with $d$ half-loops. Then, there exists a constant $\Delta = \Delta(H) > 0$ such that for any $0 \le \delta \le \Delta$ and $\sP^S = \sP^S_n(H, \delta)$,
    \begin{align*}
        d_{\TV}(\sQ^S, \sP^S) = 1 - o(1), \quad \text{ as } n \to \infty.
    \end{align*}

    Suppose $H$ is a base bipartite $d$-regular graph that is not two vertices connected by $d$ parallel edges. Then, there exists a constant $\Delta = \Delta(H) > 0$ such that for any $0 \le \delta \le \Delta$ and $\sP^{b,S} = \sP^{b,S}_n(H, \delta)$,
    \begin{align*}
        d_{\TV}(\sQ^{b,S}, \sP^{b,S}) = 1 - o(1), \quad \text{ as } n \to \infty.
    \end{align*}

\end{theorem}

\subsection{Proof Overview of the PTF Hardness Result}

To establish the PTF hardness for detecting noisy random lift of Ramanujan graphs, we first show that for any symmetric polynomial $f$ on simple regular graphs, there exists another polynomial $g$ that can be written as a linear combination of subgraph count polynomials where the subgraphs being count have minimum degree at least $2$. This is achieved by iteratively using the degree regularity constraint to peel off degree $1$ vertices from the subgraph used in the count polynomials, as explained in Lemma~\ref{lem:leaf-peeling}.

Then, we establish the standard fact that random regular graphs and random lifts are unlikely to have two short cycles that are not vertex disjoint in Lemma~\ref{lem:rare-event}. This fact and the previous argument that we can equivalently work with linear combinations of subgraph count polynomials whose minimum degree is at least $2$ allows us to deduce in Lemma~\ref{lem:function-of-cycles} that with high probability, symmetric low-degree PTF is a function of the short cycle counts.

Therefore, to show the failure of low-degree symmetric PTF, it is enough to establish that the total variation distance between the distributions of short cycle counts under both the random regular graph distribution and the noisy random lift distribution is strictly bounded away from $1$. We achieve this by showing the joint convergence of the short cycle counts to a product of independent Poisson distributions with specific means in Theorem~\ref{thm:cycle-counts} using the Chen-Stein method for proving Poisson convergence, which is the most technical proof of this work. Having established the Poisson convergence, the PTF hardness follows as a straightforward consequence.

\section{Tools for Noisy Random Lift Model}

In this section, we prove some preliminary facts about noisy random lifts.

\begin{lemma}\label{lem:matching-prob}
    Consider a set $U$ of $m$ vertices and suppose $m$ is even. Consider a collection $\sI \subseteq \binom{U}{2}$ of $e$ disjoint edges on $U$ with $e = o(m)$. Then, the probability that all the edges in $\sI$ are paired by a uniformly random perfect matching $\sigma: U \to U$  is at most
    \begin{align*}
        \Pr(\sI) \le \left(\frac{1 + o(1)}{m}\right)^e.
    \end{align*}

    Similarly, consider two disjoint sets $U$ and $V$ of equal size, each having $m$ vertices. Consider a collection $\sI \subseteq U \times V$ of $e$ disjoint edges with $e = o(m)$. Then, the probability that all the edges in $\sI$ are paired by a uniformly random bijection $\sigma: U \to V$ is at most
    \begin{align*}
        \Pr(\sI) \le \left(\frac{1 + o(1)}{m}\right)^e.
    \end{align*}
\end{lemma}

\begin{proof}
    Since $|\sI| = e = o(m)$, the probability that $\sI \subseteq \binom{U}{2}$ are paired by a perfect matching $\sigma: U \to U$ chosen uniformly at random is
    \begin{align*}
        \Pr(\sI) = \frac{(m-2e-1)!!}{(m-1)!!} \le \left(\frac{1 + o(1)}{m}\right)^e.
    \end{align*}
    Similarly, since $|\sI| = e = o(m)$, the probability that $\sI \subseteq U \times V$ are paired by a bijection $\sigma: U \to V$ chosen uniformly at random is
    \begin{align*}
        \Pr(\sI) = \frac{(m-e)!}{m!} \le \left(\frac{1 + o(1)}{m}\right)^e.
    \end{align*} 
\end{proof}

\begin{lemma}\label{lem:subgraph-containment}
    Fix $d \ge 3$ and $\delta \in [0,1]$.

    Suppose $H$ is a base $d$-regular graph on vertex set $[k]$ potentially with half-loops. Then, there exists a constant $K = K(H) > 0$ such that, for any simple graph $F \subseteq \binom{[n]}{2}$ without isolated vertex with $|F| = o(n)$, and $\sP = \sP_n(H, \delta)$,
    \begin{align*}
        \Pr_{G \sim \sP}(F \subseteq G) \le \left(\frac{K}{n}\right)^{|F|}.
    \end{align*}

    Suppose $H$ is a base bipartite $d$-regular graph on vertex set $[2k]$ with a given bipartition. Then, there exists a constant $K = K(H) > 0$ such that, for any simple graph $F \subseteq \binom{[2n]}{2}$ without isolated vertex with $|F| = o(n)$, and $\sP^b = \sP^b_{n}(H, \delta)$,
    \begin{align*}
        \Pr_{G \sim \sP^b}(F \subseteq G) \le \left(\frac{K}{n}\right)^{|F|}.
    \end{align*}
\end{lemma}

\begin{proof}
    Consider a $d$-regular graph $G \sim \sP = \sP_n(H,\delta)$, and the underlying realization of noisy random lift that produces $G$. Denote $r \colonequals \lceil \delta \frac{nd}{2} \rceil$. Under $\sP$, a total of $\frac{nd}{2}$ edges are first formed using the noiseless random lift model, which are naturally colored by the edges $E(H)$ of the base graph. Then, a uniformly random $r$-subset of the formed edges are deleted, and finally the $2r$ stubs exposed from edge deletion are paired using the configuration model to produce a $d$-regular graph. Thus, the total of $\frac{nd}{2}$ edges for $G$ are naturally split into two kinds of edges, the collection of lift-edges $\mathcal{O}$ coming from the noiseless random lift that are not deleted, and the collection of noise-edges $\mathcal{R}$ coming from the rematching process after the deletion. Let $G_{\mathcal{O}}$ denote the subgraph of $G$ realized by the lift-edges $\mathcal{O}$ and $G_{\mathcal{R}}$ denote the subgraph of $G$ realized by the noise-edges $\mathcal{R}$. Then, $G = G_{\mathcal{O}} \sqcup G_{\mathcal{R} }$ as a disjoint union of multiset. By a union bound, we have
    \begin{align*}
        \Pr_{G \sim \sP}(F \subseteq G) &\le \sum_{A \subseteq F} \Pr_{G \sim \sP}\left(A \subseteq G_{\mathcal{O}}, \text{ and } F \setminus A \subseteq G_{\mathcal{R} }\right)\\
        &= \sum_{A \subseteq F} \Pr_{G \sim \sP}\left(A \subseteq G_{\mathcal{O}}\right)\Pr_{G \sim \sP}\left(F \setminus A \subseteq G_{\mathcal{R}} \vert A \subseteq G_{\mathcal{O}}\right).
    \end{align*}
    First, we analyze the probability that $A \subseteq G_{\mathcal{O}}$. The probability for $A$ to be a subgraph of the noiseless random lift, even before the random deletion, can be upper bounded by
    \begin{align*}
        \Pr_{G \sim \sP}\left(A \subseteq G_{\mathcal{O}}\right) &\le d^{|A|} \cdot \left(\frac{k + o(1)}{n}\right)^{|A|},
    \end{align*}
    where $d^{|A|}$ comes from a union bound over all color choices $e \in E(H)$ for each edge of $A$, and $((k+o(1))/n)^{|A|}$ comes from applying Lemma~\ref{lem:matching-prob} to the random bijection/perfect matching colored by the edges $E(H)$ and using that $|A| \le |F| = o(n)$.

    Next, we analyze the conditional probability that $F \setminus A \subseteq G_{\mathcal{R}}$ conditioned on $A \subseteq G_{\mathcal{O}}$. Let $G'$ be the noiseless random lift before the random deletion. Conditioned on $A \subseteq G_{\mathcal{O}} \subseteq G'$, there exists at least one edge in $G'$ between every $\{i,j\} \in A$. In order for $F \setminus A$ to be a subgraph of $G_{\mathcal{R}}$, there should be at least one edge in $G_{\mathcal{R}}$ between every edge $\{i,j\} \in F \setminus A$, and moreover these edges are rematched in $G_{\mathcal{R}}$ after the random deletion process is applied to $G'$. 
    
    There are at most $d^{2|F \setminus A|}$ choices of the stub-pairs that realize $F \setminus A$. For any fixed choice of the $|F \setminus A|$ stub-pairs that realize $F \setminus A$, there are $x = x(G')$ edges in $G'$ incident to the chosen stubs of $F \setminus A$ that must be deleted in the random deletion process in order to expose the stubs and make it possible for the chosen stub-pairs to be rematched and appear in $G_{\mathcal{R}}$. Notice that $|F \setminus A| \le x \le 2|F \setminus A|$. Recall $r = \lceil \delta \frac{nd}{2} \rceil$ is the number of deleted edges in the random deletion process. If $x > r$, the probability that chosen $|F \setminus A|$ stub-pairs are paired in $G_{\mathcal{R}}$ is $0$. Otherwise, the probability that the chosen $|F \setminus A|$ stub-pairs are paired in $G_{\mathcal{R}}$ is at most
    \begin{align*}
        &\frac{\binom{\frac{dn}{2}-x}{r-x}}{\binom{\frac{dn}{2}}{r}}\cdot \frac{(2r - 2|F \setminus A| -1)!!}{(2r-1)!!}\\
        &\le \frac{\binom{\frac{dn}{2}-|F \setminus A|}{r-|F \setminus A|}}{\binom{\frac{dn}{2}}{r}}\cdot \frac{(2r - 2|F \setminus A| -1)!!}{(2r-1)!!} \quad && (\text{Using }|F \setminus A| \le x)\\
        &\le \frac{\left(\frac{dn}{2} - |F \setminus A|\right)!}{\left(\frac{dn}{2}\right)!} \cdot \frac{r(r-1)\dots (r -|F \setminus A| +1)}{(2r-1)(2r-3) \dots (2r - 2|F \setminus A| + 1)}\\
        &\le \left(\frac{2+o(1)}{dn}\right)^{|F \setminus A|}, \quad && (\text{Using }|F \setminus A| \le |F| = o(n))
    \end{align*}
    where $\frac{\binom{\frac{dn}{2}-x}{r-x}}{\binom{\frac{dn}{2}}{r}}$ accounts for the deletion probability for the $x$ edges incident to the chosen $|F \setminus A|$ stub-pairs, and $\frac{(2r - 2|F \setminus A| -1)!!}{(2r-1)!!}$ is the probability for the chosen $|F \setminus A|$ stub-pairs that realize $F \setminus A$ to be rematched in $G_{\mathcal{R}}$.
    
    By a union bound over at most $d^{2|F \setminus A|}$ choices of the stub-pairs that realize $F \setminus A$, we have
    \begin{align*}
        &\Pr_{G \sim \sP}\left(F \setminus A \subseteq G_{\mathcal{R}} \vert A \subseteq G_{\mathcal{O}}\right)\\
        &\le d^{2|F \setminus A|} \left(\frac{2+o(1)}{dn}\right)^{|F \setminus A|}\\
        &\le \left(\frac{2d+o(1)}{n}\right)^{|F \setminus A|}.
    \end{align*}

    Combining the bound for $\Pr(A \subseteq G_{\mathcal{O}})$ and $\Pr(F \setminus A \subseteq G_{\mathcal{R}} \vert A \subseteq G_{\mathcal{O}})$, we get
    \begin{align*}
        \Pr_{G \sim \sP}(F \subseteq G) &\le \sum_{A \subseteq F} \Pr_{G \sim \sP}\left(A \subseteq G_{\mathcal{O}}\right)\Pr_{G \sim \sP}\left(F \setminus A \subseteq G_{\mathcal{R}} \vert A \subseteq G_{\mathcal{O}}\right)\\
        &\le \sum_{A \subseteq F} d^{|A|} \cdot \left(\frac{k + o(1)}{n}\right)^{|A|} \left(\frac{2d+o(1)}{n}\right)^{|F \setminus A|}\\
        &\le \left(\frac{2d(k+2d+o(1))}{n}\right)^{|F|},
    \end{align*}
    proving the claim.

    The proof for the bipartite case is essentially identical, with small adjustments for the bipartite configuration models. We omit the details here.
\end{proof}

\section{Proof of PTF Hardness}\label{sec:proof-ptf-hardness}

In this section, we prove Theorem~\ref{thm:ptf-hardness} using Theorem~\ref{thm:cycle-counts}, and deduce Corollary~\ref{cor:multigraph-hardness} and Corollary~\ref{cor:low-deg-hardness}. First, we review the space of polynomials on simple regular graphs that we work with.

\subsection{Polynomials on Regular Simple Graph}
Fix $d \ge 3$. A polynomial on a $d$-regular simple graph $G = ([n],E)$ is a polynomial in the entries of the adjacency matrix $Y \in \{0,1\}^{n \times n}$ of $G$. Since $G$ is $d$-regular and simple, the entries of $Y$ satisfy the following set $\sS$ of constraints:
\begin{align*}
    \left.\begin{aligned}
        Y_{i,i} &= 0, \hspace{1cm} && \text{ for all } i\in [n], \\
    Y_{i,j} &= Y_{j,i}, \hspace{1cm} && \text{ for all } i,j\in [n],\quad\\
    Y_{i,j}^2 &= Y_{i,j}, \hspace{1cm} && \text{ for all } i,j\in [n],\\
    \sum_{j \in [n]} Y_{i,j} &= d, \hspace{1cm} && \text{ for all } i\in [n].
    \end{aligned}\right\} \sS
\end{align*}
Thus, reducing by these constraints, it is sufficient to consider the space of multilinear polynomials in the entries $\{Y_{i,j}: \{i,j\} \in \binom{[n]}{2}\}$. Let $\RR[Y]_{\le D}$ denote the space of multilinear polynomials in $Y$ with degree at most $D$. One convenient basis for working with $\RR[Y]_{\le D}$ is the monomial basis indexed by subgraphs of the complete graph $K_n$ without isolated vertex, each of which is specified by its edge set $A \subseteq \binom{[n]}{2}$.

\begin{fact}
    For $A \subseteq \binom{[n]}{2}$, set $Y^A \colonequals \prod_{\{i,j\} \in A} Y_{i,j}$. Then, the collection $\{Y^A: A \subseteq \binom{[n]}{2}, |A| \le D\}$ is a basis of polynomials for $\RR[Y]_{\le D}$.
\end{fact}

For symmetric PTFs, we will work with the following subspace of symmetric polynomials.

\begin{definition}[Symmetric Polynomials and Symmetric PTFs]\label{def:sym-PTF}
    A polynomial $f \in \RR[Y]_{\le D}$ is symmetric if for any permutation $\pi \in \Sym([n])$, $f(Y)$ is invariant under permuting vertex labels by $\pi$, i.e.,
    \begin{align*}
        f(Y) = f(\pi \cdot Y),
    \end{align*}
    where $\pi \cdot Y$ is the natural action of $\pi$ on $Y$ given by $(\pi \cdot Y)_{i,j} = Y_{\pi(i), \pi(j)}$. 

    The subspace of symmetric polynomials is denoted as $\RR_{\sym}[Y]_{\le D} \subseteq \RR[Y]_{\le D}$. We call a PTF symmetric if its underlying polynomial is symmetric.
\end{definition}

\begin{fact}
    The space of symmetric polynomials $\RR_{\sym}[Y]_{\le D}$ is spanned by the following polynomials:
    \begin{align*}
        \left\{N_{\alpha}(Y) \colonequals \sum_{A \in \Emb(\alpha)} Y^A\,:\, \alpha \in \sG_{\le D} \right\},
    \end{align*}
    where $\sG_{\le D}$ denotes the collection of unlabelled simple graphs with at most $D$ edges without isolated vertex, and $\Emb(\alpha)$ denotes the collection of subgraphs of the complete graph $K_n$ isomorphic to $\alpha$.
\end{fact}

Since both $\sQ^S$ and $\sP^S$ are permutation invariant, symmetric PTFs are a natural class of PTFs to consider, whose outputs are functions of the isomorphism class of the input graph. Next, we exploit the special structure of simple regular graphs and show that we can further restrict our attention to the following subspace of polynomials in $\RR_{\sym}[Y]_{\le D}$ without loss of generality.

\begin{definition}
    Let $\sK_{\le D} \subseteq \RR_{\sym}[Y]_{\le D}$ be the subspace of polynomials with degree at most $D$ spanned by the following polynomials:
    \begin{align*}
        \left\{N_{\alpha}(Y) \colonequals \sum_{A \in \Emb(\alpha)} Y^A\,:\, \alpha \in \sG_{\le D}, \delta(\alpha) \ge 2 \right\},
    \end{align*}
    where $\delta(\alpha)$ denotes the minimum degree of $\alpha$.
\end{definition}

\begin{lemma}\label{lem:leaf-peeling}
    Recall that $\sS$ is the set of constraints imposed by $d$-regular simple graphs. For any $f \in \RR_{\sym}[Y]_{\le D}$, there exists a polynomial $g \in \sK_{\le D}$ such that $f \equiv g$ modulo the constraints $\sS$.
\end{lemma}

\begin{proof}
    Any $f \in \RR_{\sym}[Y]_{\le D}$ can be written as a polynomial in $N_{\alpha}(Y)$ with $\alpha \in \sG_{\le D}$. It is enough to prove that any $N_{\alpha}(Y)$ with $\alpha \in \sG_{\le D}$, after being reduced by $\sS$, can be written as a polynomial in $\sK_{\le D}$. We will prove this by induction on the number of vertices $v(\alpha) \colonequals |V(\alpha)|$.

    The base cases are $v(\alpha) \le 2$. Note that $\alpha$ with at most $2$ vertices is either an empty graph or a single edge. In both cases, modulo the constraints $\sS$, $N_{\alpha}(Y)$ is a constant polynomial since $\sS$ forces $\sum_{\{i,j\} \in \binom{[n]}{2}}Y_{i,j} = \frac{nd}{2}$, and hence obviously belongs to $\sK_{\le D}$.

    Suppose that for some $h \ge 2$, any $N_{\alpha'}(Y)$ with $\alpha' \in \sG_{\le D}$ and $v(\alpha') \le h$ can be written as a polynomial in $\sK_{\le D}$ after being reduced by $\sS$. Now consider an arbitrary $\alpha \in \sG_{\le D}$ with $v(\alpha) = h+1$. If $\delta(\alpha) \ge 2$, then we are done and by definition $N_{\alpha}(Y) \in K_{\le D}$. Otherwise, $\delta(\alpha) = 1$, and there exists a leaf $x \in V(\alpha)$ that has degree $1$. Let $y \in V(\alpha)$ be the unique vertex in $\alpha$ adjacent to $x$. In this case, we expand the definition of $N_\alpha(Y)$ and get
    \begin{align*}
        &N_{\alpha}(Y)\\
        &= \sum_{A \in \Emb(\alpha)} Y^A\\
        &= \frac{1}{|\Aut(\alpha)|}\sum_{\tau: V(\alpha) \hookrightarrow [n] } \prod_{\{i,j\} \in \alpha} Y_{\tau(i), \tau(j)}
        \intertext{where $|\Aut(\alpha)|$ is the size of the automorphism group of $\alpha$, and the summation is over all injective maps from $V(\alpha)$ to $[n]$. Now we first sum over the labels of $V(\alpha) \setminus \{x\}$, and then the label $a$ of $x$.}
        &= \frac{1}{|\Aut(\alpha)|} \sum_{\tau: V(\alpha) \setminus \{x\} \hookrightarrow [n]} \prod_{\{i,j\} \in \alpha \setminus\{ \{x,y\} \} } Y_{\tau(i), \tau(j)} \sum_{a \in [n] \setminus \tau(V(\alpha) \setminus \{x\})} Y_{\tau(y), a}\\
        &= \frac{1}{|\Aut(\alpha)|} \sum_{\tau: V(\alpha) \setminus \{x\} \hookrightarrow [n]} \prod_{\{i,j\} \in \alpha \setminus\{ \{x,y\} \} } Y_{\tau(i), \tau(j)} \left(\sum_{a \in [n]} Y_{\tau(y), a} - \sum_{z \in V(\alpha) \setminus \{x\}} Y_{\tau(y), \tau(z)}\right)
        \intertext{Note that $\sum_{a \in [n]} Y_{\tau(y), a} = d$ and $Y_{\tau(y), \tau(y)} = 0$ by $\sS$. Thus,}
        &= \frac{1}{|\Aut(\alpha)|} \sum_{\tau: V(\alpha) \setminus \{x\} \hookrightarrow [n]} \prod_{\{i,j\} \in \alpha \setminus\{ \{x,y\} \} } Y_{\tau(i), \tau(j)} \left(d - \sum_{z \in V(\alpha) \setminus \{x,y\}} Y_{\tau(y), \tau(z)}\right)\\
        &= \frac{1}{|\Aut(\alpha)|} \left(d |\Aut(\alpha^{-})| \cdot N_{\alpha^{-}}(Y) - \sum_{z \in V(\alpha)\setminus\{x,y\}} |\Aut(\alpha^z)| \cdot N_{\alpha^z}(Y)\right),
    \end{align*}
    where $\alpha^{-}$ is the graph obtained from $\alpha$ by removing vertex $x$ and edge $\{x,y\}$, and $\alpha^z$ for $z \in V(\alpha) \setminus \{x,y\}$ is the graph obtained from $\alpha^{-}$ by adding the edge $\{y,z\}$ if it is not already present in $\alpha^{-}$. Note that $\alpha^{-}$ and $\alpha^z$ all have at most $v(\alpha) - 1 = h$ vertices, and moreover $\alpha^{-}$ and $\alpha^z$ have at most as many edges as $\alpha$. By the inductive hypothesis, there are polynomials $g^{-}, g^z \in \sK_{\le D}$ such that modulo the constraints $\sS$, $N_{\alpha^{-}}(Y) \equiv g^{-}$ and $N_{\alpha^z}(Y) \equiv g^{z}$ for all $z \in V(\alpha) \setminus \{x,y\}$. Therefore, there exists $g \in \sK_{\le D}$ such that $N_{\alpha}(Y) \equiv g$ modulo $\sS$.
    
    By induction, this finishes the proof.
\end{proof}

\subsection{Auxiliary Results}

We first prove a lemma on the following rare event.

\begin{lemma}\label{lem:rare-event}
    Fix $d \ge 3$ and $\delta > 0$.
    
    Let $G$ be a $d$-regular simple graph on $n$ vertices. Let $R_{\le t}$ denote the event that $G$ contains a subgraph $A \subseteq G$ such that $A$ is connected, $|E(A)| \le t$, and $|E(A)| - |V(A)| = 1$. Suppose $H$ is a base $d$-regular graph on $k$ vertices potentially with half-loops. Then, there exists a constant $c = c(H) > 0$ such that for any $D \le c\log(n)$,
    \begin{align*}
        \Pr_{G \sim \sQ^S}(R_{\le D}) = o(1), \;\Pr_{G \sim \sP^S}(R_{\le D}) = o(1), \quad \text{ as } n \to \infty.
    \end{align*}

    Similarly, suppose $H$ is a base bipartite $d$-regular graph on $2k$ vertices with a given bipartition. Then, there exists a constant $c = c(H) > 0$ such that for any $D \le c\log(n)$,
    \begin{align*}
        \Pr_{G \sim \sQ^{b,S}}(R_{\le D}) = o(1), \;\Pr_{G \sim \sP^{b,S}}(R_{\le D}) = o(1), \quad \text{ as } n \to \infty.
    \end{align*}
\end{lemma}

\begin{proof}
    Let $\sG^+_{t}$ denote the collection of unlabelled connected simple graphs with exactly $t$ edges and $t-1$ vertices. Note that any graph in $\sG^+_{t}$ can be expressed as a spanning tree on $t-1$ vertices plus two additional edges not inside the spanning tree. Moreover, the number of non-isomorphic spanning trees\footnote{The exact asymptotics of the number of non-isomorphic spanning trees on $v$ vertices is given by $C\alpha^v v^{-\frac{5}{2}}$, where $C \approx 0.535$ is a constant and $1/\alpha \approx 0.338$ is the Otter's tree constant.} on $t - 1$ vertices is upper bounded by the number of rooted ordered trees\footnote{This is the set of trees with a designated root, in which the children of every node are ordered from left to right. There is a classical correspondence between rooted ordered trees with $e$ edges and Dyck paths of order $e$.} on $t - 1$ vertices, given by the Catalan number $\frac{1}{t-1} \binom{2t-4}{t-2} \le 2^{2t}$. Given a spanning tree on $t-1$ vertices, we need to enumerate the choices for two additional edges in order to get a graph in $\sG^+_t$, which gives at most $\binom{\binom{t-1}{2}}{2} \le t^4$ choices. Thus, the total number of graphs in $\sG^+_t$ is at most $t^4 2^{2t}$.

    By a union bound over all $\alpha \in \sG^+_t$ for $1 \le t \le D$, and for each $\alpha \in \sG^+_t$ at most $n^{|V(\alpha)|}$ copies of $A \subseteq \binom{[n]}{2}$ isomorphic to $\alpha$, we have
    \begin{align*}
        \Pr_{G \sim \sQ}\left(R_{\le D}\right) &\le \sum_{t=1}^D\sum_{\alpha \in \sG_{t}^+} \sum_{\substack{A \subseteq \binom{[n]}{2}:\\ A \cong \alpha}} \Pr(A \subseteq G)\\
        &\le \sum_{t=1}^D\sum_{\alpha \in \sG_{t}^+} \sum_{\substack{A \subseteq \binom{[n]}{2}:\\ A \cong \alpha}} \left(\frac{K}{n}\right)^{|E(\alpha)|} \quad &&(\text{Lemma~\ref{lem:subgraph-containment}})\\
        &\le \sum_{t=1}^D \sum_{\alpha \in \sG_{t}^+} n^{|V(\alpha)|} \left(\frac{K}{n}\right)^{|E(\alpha)|}\\
        &\le \sum_{t=1}^D t^42^{2t} \frac{K^t}{n}, \quad && (|E(\alpha)|-|V(\alpha)|=1)
    \end{align*}
    where $K = K(H)$ is the constant given by Lemma~\ref{lem:subgraph-containment}. Thus, there exists $c = c(H) > 0$ such that for any $D \le c\log(n)$,
    \begin{align*}
        \Pr_{G \sim \sQ}\left(R_{\le D}\right) &\le \sum_{t=1}^D t^42^{2t} \frac{K^t}{n} = o(1).
    \end{align*}
    By Corollary~\ref{cor:simple-prob}, the probability of $G \sim \sQ$ being simple is bounded away from $0$, and we conclude
    \begin{align*}
        \Pr_{G \sim \sQ^S}\left(R_{\le D}\right) &= \frac{\Pr_{G \sim \sQ}\left(R_{\le D}\right)}{\Pr_{G \sim \sQ}\left(G\text{ is simple}\right)} = o(1).
    \end{align*}

    The proof is identical for $\sP^S$ and the bipartite versions.
\end{proof}

\begin{remark}\label{rem:bicycles}
    Note that on the good event $R_{\le L}^c$, $G$ does not contain bicycles of small sizes, i.e., a pair of distinct short cycles that share at least one vertex. This will be particularly useful, as it allows for a nice factorization of $N_{\alpha}(Y)$ when $\alpha$ is a disjoint union of cycles, as shown in the following lemma. 
\end{remark}

\begin{lemma}\label{lem:function-of-cycles}
    Let $f \in \sK_{\le D}$. There exists a function $P: \RR^{D-2} \to \RR$ such that
    \begin{align*}
        f(Y) = P( Z_3, Z_4, \dots, Z_D)
    \end{align*}
    whenever $Y = \{Y_{i,j}: \{i,j\} \in \binom{[n]}{2}\} \in \{0,1\}^{\binom{[n]}{2}}$ encodes a $d$-regular simple graph $G$ and $G \in R_{\le 2D}^c$.
    Here $Z_t \colonequals N_{C_t}(Y)$ is the number of length-$t$ cycles in $G$.
\end{lemma}

\begin{proof}
    By definition, any $f \in \sK_{\le D}$ is a linear combination of $N_{\alpha}(Y)$ where $\alpha \in \sG_{\le D}$ and $\delta(\alpha) \ge 2$. It is therefore sufficient to show that for any $\alpha \in \sG_{\le D}$ with $\delta(\alpha) \ge 2$, we can write $N_{\alpha}(Y)$ as a function of the cycle counts $Z_3, Z_4, \dots, Z_D$ when $Y$ encodes a $d$-regular simple graph $G \in R_{\le D}^c$.

    Recall that $R_{\le D}^c$ is the complement of $R_{\le D}$, the collection of graphs that contain some connected subgraph $A$ with $|E(A)| - |V(A)|=1$. Since $G \in R_{\le D}^c$, $N_{\alpha}(Y) = \sum_{A \in \Emb(\alpha)} Y^A$ evaluates to $0$ whenever $\alpha$ has a connected component with more edges than vertices. Consequently, the only nonzero $N_{\alpha}(Y)$ for $\alpha \in \sG_{\le D}$ with $\delta(\alpha) \ge 2$ must satisfy that every connected component has a number of edges at most the number of vertices, and has minimum degree at least $2$. Therefore, we conclude that $N_{\alpha}(Y)$ is nonzero only if $\alpha$ is a disjoint union of cycles.

    For $\alpha \in \sG_{\le D}$ that is a disjoint union of cycles, let us express it as
    \begin{align*}
        \alpha = \bigsqcup_{t=3}^D (C_t)^{i_t}, 
    \end{align*}
    where the exponents $i_t \in \NN$ indicate the number of length-$t$ cycles $C_t$ in $\alpha$. Note that all the cycles of length at most $D$ in $G \in R_{\le 2D}^c$ are vertex-disjoint, since otherwise $G$ would contain a subgraph $F$ with $|E(F)| \ge |V(F)| + 1$ and $|E(F)| \le 2D$, contradicting $G \in R_{\le 2D}^c$. As a result, $N_{\alpha}(Y)$, the number of subgraphs of $G$ isomorphic to $\alpha$, is equal to
    \begin{align}
        N_{\alpha}(Y) = \prod_{t=3}^D \binom{Z_t}{i_t}. \label{eq:alpha-count}
    \end{align}
    Thus, when $Y$ encodes a $d$-regular simple graph $G \in R_{\le 2D}^c$, for any $\alpha \in \sG_{\le D}$ with $\delta(\alpha) \ge 2$,
    \begin{align*}
        N_{\alpha}(Y) = \begin{cases}
            P_{\alpha}(Z_3, Z_4, \dots, Z_D) & \quad \text{ if } \alpha \text{ is a disjoint union of cycles},\\
            0 & \quad \text{ otherwise}. 
        \end{cases}
    \end{align*}
    for the explicit function $P_{\alpha}$ given by \eqref{eq:alpha-count}. This finishes the proof.
\end{proof}

As a last step before we assemble all the pieces needed for the proof of Theorem~\ref{thm:ptf-hardness}, we prove the following TV distance bound between sequences of independent Poisson distributions.

\begin{proposition}\label{prop:Ramanujan-TV-bound}
    Fix $d \ge 3$ and $0 <\delta \le 1$. 
    
    Suppose $H$ is a base $d$-regular graph on $k$ vertices potentially with half-loops, and $H$ is Ramanujan. Let $B$ be the non-backtracking matrix of $H$. Then, there exists a constant $K = K(H,\delta) > 0$ such that for any $D = D(n)$,
    \begin{align*}
        d_{\TV}\left(\bigotimes_{t=3}^D \Pois\left(\frac{(d-1)^t}{2t}\right),  \bigotimes_{t=3}^D \Pois(\mu_t)\right) \le 1 - K,
    \end{align*}
    where $\mu_t \colonequals (1-\delta)^t \left(\nu_t - \frac{(d-1)^t}{2t}\right) + \frac{(d-1)^t}{2t}$ and $\nu_t \colonequals \frac{\tr(B^t)}{2t}$.

    Suppose $H$ is a base bipartite $d$-regular graph on $2k$ vertices, and $H$ is bipartite Ramanujan. Let $B$ be the non-backtracking matrix of $H$. Then, there exists a constant $K = K(H, \delta) > 0$ such that for any $L = L(n)$,
    \begin{align*}
        d_{\TV}\left(\bigotimes_{t=2}^L \Pois\left(\frac{(d-1)^{2t}}{2t}\right),  \bigotimes_{t=2}^L \Pois(\mu_{2t}')\right) \le 1 - K,
    \end{align*}
    where $\mu_{2t}' \colonequals (1-\delta)^{2t} \left(\nu_{2t} - \frac{(d-1)^{2t}}{2t}\right) + \frac{(d-1)^{2t}}{2t}$ and $\nu_{2t} \colonequals \frac{\tr(B^{2t})}{4t}$.
\end{proposition}

\begin{proof}
    First we consider the non-bipartite case. Suppose $H$ is $d$-regular Ramanujan on $k$ vertices potentially with half-loops, and $B$ be the non-backtracking matrix of $H$. Let us first analyze the spectrum of $B$.

    Let $M$ be the adjacency matrix of $H$. Recall that the non-backtracking matrix $B$ of $H$, defined in Definition~\ref{def:nb-matrix}, is indexed by $\vec{E}(H)$, the set of oriented edges of $H$, and has entries
    \begin{align*}
        B_{\vec{e}_1,\vec{e}_2} = \One\{\tail(\vec{e}_2) = \head(\vec{e}_1), \vec{e}_2 \ne \iota(\vec{e}_1)\}.
    \end{align*}

    The Ihara-Bass formula in the presence of half-loops \cite[equation (1.1)]{friedman2014relativized} gives
    \begin{align*}
        \det(\lambda I - B) &=  (\lambda - 1)^{|E_1(H)|} (\lambda^2 - 1)^{|E_2(H)| - |V(H)|} \cdot \det(\lambda^2 I - \lambda M + (d-1)I),
    \end{align*}
    where $|E_1(H)|$ is the number of half-loops in $H$ and $|E_2(H)|$ is the number of edges of $H$ that are not half-loops. Thus, any eigenvalue $\lambda \ne \pm 1$ of $B$ satisfies $\det(\lambda^2 I - \lambda M + (d-1)I) = 0$.
    Since $M$ is symmetric, we may diagonalize $M$, from which we see that any eigenvalue $\lambda \ne \pm 1$ of $B$ is a solution to
    \begin{align}
        \lambda^2 - \lambda \rho + d-1 = 0 \label{eq:eig-relation}
    \end{align}
    for an eigenvalue $\rho$ of $M$. Since $H$ is Ramanujan, $M$ has $1$ eigenvalue of $d$ and all the other eigenvalues at most $2\sqrt{d-1}$ in absolute value. When $\rho = d$, the two solutions to \eqref{eq:eig-relation} are $\lambda = 1$ and $\lambda = d-1$. When $|\rho| \le 2\sqrt{d-1}$, the two solutions to \eqref{eq:eig-relation}, potentially complex, have absolute value at most $\sqrt{d-1}$. Therefore, we conclude that $B$ has one eigenvalue of $d-1$, and all the other eigenvalues at most $\sqrt{d-1}$ in absolute value.

    Note that $|\vec{E}(H)| = kd$. As a result, for any $t \ge 1$, we have
    \begin{align*}
        (d-1)^t - (kd-1) (d-1)^{\frac{t}{2}}\le \tr(B^t) \le (d-1)^t + (kd-1) (d-1)^{\frac{t}{2}},
    \end{align*}
    and
    \begin{align}
        \left|\mu_t - \frac{(d-1)^t}{2t}\right| &= (1-\delta)^t \left|\frac{\tr(B^t)}{2t} - \frac{(d-1)^t}{2t}\right| \le (1-\delta)^t(kd - 1) \frac{(d-1)^{\frac{t}{2}}}{2t}. \label{ineq:mean-difference}
    \end{align}
    
    Using the standard fact about KL divergence between two Poissons \[d_{\KL}(\Pois(a)\|\Pois(b)) = a\log\left(\frac{a}{b}\right) + b - a,\] and the tensorization of KL divergence for product distributions, we now bound the KL divergence between the two sequences of independent Poisson distributions as follows.
    \begin{align*}
        &d_{\KL}\left(\bigotimes_{t=3}^D \Pois(\mu_t)\bigg\| \bigotimes_{t=3}^D \Pois\left(\frac{(d-1)^t}{2t}\right)\right)\\
        &= \sum_{t=3}^D d_{\KL}\left(\Pois(\mu_t)\bigg\| \Pois\left(\frac{(d-1)^t}{2t}\right)\right)\\
        &= \sum_{t=3}^D \mu_t \log\left(\frac{\mu_t}{\frac{(d-1)^t}{2t}}\right) + \frac{(d-1)^t}{2t} - \mu_t\\
        &\le \sum_{t=3}^D \mu_t \left(\frac{\mu_t - \frac{(d-1)^t}{2t}}{\frac{(d-1)^t}{2t}}\right) + \frac{(d-1)^t}{2t} - \mu_t \quad && (\text{Using }\log(1 + r)\le r)\\
        &= \sum_{t=3}^D\frac{\left(\mu_t - \frac{(d-1)^t}{2t}\right)^2}{\frac{(d-1)^t}{2t}}\\
        &\le (kd-1)^2\sum_{t=3}^D \frac{(1-\delta)^{2t}}{2t} \quad && (\text{Using }\eqref{ineq:mean-difference})\\
        &\le U,
    \end{align*}
    for some constant $U = U(H, \delta) < \infty$. By Bretagnolle-Huber inequality \cite{bretagnolle1979estimation},
    \begin{align*}
        &d_{\TV}\left(\bigotimes_{t=3}^D \Pois\left(\frac{(d-1)^t}{2t}\right),  \bigotimes_{t=3}^D \Pois(\mu_t)\right)\\
        &\le \sqrt{1 - \exp\left(-d_{\KL}\left(\bigotimes_{t=3}^D \Pois(\mu_t)\bigg\| \bigotimes_{t=3}^D \Pois\left(\frac{(d-1)^t}{2t}\right)\right)\right)}\\
        &\le \sqrt{1 - \exp\left(-U\right)}\\
        &= 1 - K,
    \end{align*}
    for some $K = K(H, \delta) > 0$. This finishes the proof for the non-bipartite case.

    We briefly explain how the proof changes for the bipartite case. We again apply the Ihara-Bass formula to analyze the spectrum of $B$. For a bipartite Ramanujan $H$, we may show that $B$ has one eigenvalue of $d-1$, one eigenvalue of $-d+1$, and all the other eigenvalues at most $\sqrt{d-1}$ in absolute value. For any $t \ge 1$, we now get that
    \begin{align*}
        2(d-1)^{2t} - (2kd-2) (d-1)^{t}\le \tr(B^{2t}) \le 2(d-1)^{2t} + (2kd-2) (d-1)^{t},
    \end{align*}
    and the rest of the proof for the bipartite case follows similarly to the non-bipartite case.
\end{proof}

Now, we are ready to present the proof of the PTF hardness for detecting noisy random lift of Ramanujan graph.

\subsection{Proof of Theorem~\ref{thm:ptf-hardness}}

Suppose $H$ is $d$-regular Ramanujan on $k$ vertices potentially with half-loops. By Lemma~\ref{lem:rare-event}, there exists $c_1 = c_1(H) > 0$ such that for any $D \le c_1 \log(n)$,
\begin{align}
    \Pr_{G \sim \sQ^S}(R_{\le 2D}) = o(1), \;\Pr_{G \sim \sP^S}(R_{\le 2D}) = o(1), \label{ineq:rare-event}
\end{align}
where $R_{\le 2D}$ denotes the event that $G$ contains a subgraph $A$ with at most $2D$ edges such that $|E(A)| - |V(A)| = 1$.

Recall that $\sP = \sP_n(H, \delta)$ is the distribution of $r$-noisy random $m$-lift of $H$ where $r \colonequals \lceil \delta \frac{nd}{2}\rceil$ and $n \colonequals mk$. By Theorem~\ref{thm:cycle-counts}, there exists $c_2 = c_2(H) > 0$ such that for any $D \le c_2\log(n)$, the distribution of the short cycle counts converges to independent Poissons under both $\sQ$ and $\sP$ as
\begin{align}
        d_{\TV}\left(\sL_{\sQ}(Z_1, Z_2, Z_3, \dots, Z_D),  \bigotimes_{t=1}^D \Pois\left(\frac{(d-1)^t}{2t}\right)\right) &= o(1), \label{ineq:TV-Q}\\
        d_{\TV}\left(\sL_{\sP}(Z_1, Z_2, Z_3, \dots, Z_D),  \bigotimes_{t=1}^D \Pois\left(\mu_t\right)\right) &= o(1),\label{ineq:TV-P}
    \end{align}
    where $\mu_t \colonequals (1-\delta')^t \left(\nu_t - \frac{(d-1)^t}{2t}\right) + \frac{(d-1)^t}{2t}$, $\delta' \colonequals \frac{2r}{nd}$, $\nu_t \colonequals \frac{\tr(B^t)}{2t}$, and $B$ is the non-backtracking matrix of $H$. Note that $\delta' = \frac{2r}{nd} = \lceil \delta \frac{nd}{2} \rceil \frac{2}{nd} \ge \delta$ is bounded away from $0$. Conditioning on simplicity, i.e., $Z_1 = Z_2 = 0$, we get
    \begin{align*}
        d_{\TV}\left(\sL_{\sQ^S}(Z_3, \dots, Z_D),  \bigotimes_{t=3}^D \Pois\left(\frac{(d-1)^t}{2t}\right)\right) &= o(1),\\
        d_{\TV}\left(\sL_{\sP^S}(Z_3, \dots, Z_D),  \bigotimes_{t=3}^D \Pois\left(\mu_t\right)\right) &= o(1),
    \end{align*}
    which follows because both the probability of $\Pois((d-1)/2) \otimes \Pois((d-1)^2/4) = (0,0)$ and the probability of $\Pois(\mu_1) \otimes \Pois(\mu_2) = (0,0)$ are bounded away from 0.

    Let $c = \min\{c_1, c_2\}$. For $D \le c\log(n)$, consider an arbitrary symmetric polynomial $f \in \RR_{\sym}[Y]_{\le D}$ of degree at most $D$. By Lemma~\ref{lem:leaf-peeling}, there exists a polynomial $g \in \sK_{\le D}$ such that $f \equiv g$ modulo the constraints $S$. In other words, $f(Y)$ and $g(Y)$ takes the same value when $Y$ encodes a $d$-regular simple graph. By Lemma~\ref{lem:function-of-cycles}, there exists a function $P: \RR^{D-2} \to \RR$ such that
\begin{align*}
    g(Y) = P(Z_3, Z_4, \dots, Z_D)
\end{align*}
when $Y$ encodes a $d$-regular simple graph $G \in R_{\le 2D}^c$. The testing advantage of the PTF specified by $f$ can then be bounded as
\begin{align*}
    &\left|\Pr_{\sQ^S}(f(Y) \ge 0) - \Pr_{\sP^S}(f(Y) \ge 0)\right|\\
    &= \left|\Pr_{\sQ^S}(g(Y) \ge 0) - \Pr_{\sP^S}(g(Y) \ge 0)\right|\\
    &\le \left|\Pr_{\sQ^S}\left(P(Z_3, \dots, Z_D) \ge 0 \right) - \Pr_{\sP^S}\left(P(Z_3, \dots, Z_D) \ge 0\right)\right| + \Pr_{\sQ^S}(R_{\le 2D}) + \Pr_{\sP^S}(R_{\le 2D})\\
    &\le \left|\Pr_{\sQ^S}\left(P(Z_3, \dots, Z_D) \ge 0 \right) - \Pr_{\sP^S}\left(P(Z_3, \dots, Z_D) \ge 0\right)\right| + o(1)\quad && (\text{Using } \eqref{ineq:rare-event})\\
    &\le d_{\TV}\left(\sL_{\sQ^S}(Z_3, \dots, Z_D), \sL_{\sP^S}(Z_3, \dots, Z_D)\right) + o(1)\\
    &\le d_{\TV}\left(\bigotimes_{t=3}^D \Pois\left(\frac{(d-1)^t}{2t}\right),  \bigotimes_{t=3}^D \Pois(\mu_t)\right) + o(1) \quad && (\text{Using } \eqref{ineq:TV-Q} \text{ and }\eqref{ineq:TV-P})\\
    &\le 1 - K, \quad && (\text{Proposition~\ref{prop:Ramanujan-TV-bound}})
\end{align*}
for some $K = K(H,\delta') > 0$. Thus, we conclude that for Ramanujan $H$, there exists $c = c(H) > 0$ such that no symmetric PTF of degree at most $c\log(n)$ achieves strong detection between $\sQ^S$ and $\sP^S$.

The proof for the case of bipartite Ramanujan $H$ follows the same way and we omit the details.

\subsection{Proof of the Corollaries}

\begin{proof}[Proof of Corollary~\ref{cor:multigraph-hardness}]
    Assume for contradiction that some symmetric PTF of degree at most $D$ achieves strong detection between $\sQ$ and $\sP$. Then, there exists $f \in \RR_{\sym}[Y]_{\le D}$, not necessarily multilinear, such that
    \begin{align*}
        \Pr_{Y \sim \sQ}(f(Y) \ge 0) + \Pr_{Y \sim \sP}(f(Y) < 0) = o(1).
    \end{align*}
    By Corollary~\ref{cor:simple-prob}, the probabilities that $Y \sim \sQ$ is simple and that $Y \sim \sP$ is simple are both bounded away from $0$, and we get
    \begin{align*}
        &\Pr_{Y \sim \sQ^S}(f(Y) \ge 0) + \Pr_{Y \sim \sP^S}(f(Y) < 0)\\
        &= \frac{\Pr_{Y \sim \sQ}(f(Y) \ge 0, Y\text{ is simple})}{\Pr_{Y \sim \sQ}\left(Y\text{ is simple}\right)} + \frac{\Pr_{Y \sim \sP}(f(Y) < 0, Y\text{ is simple})}{\Pr_{Y \sim \sP}\left(Y\text{ is simple}\right)}\\
        &\le \frac{\Pr_{Y \sim \sQ}(f(Y) \ge 0)}{\Pr_{Y \sim \sQ}\left(Y\text{ is simple}\right)} + \frac{\Pr_{Y \sim \sP}(f(Y) < 0)}{\Pr_{Y \sim \sP}\left(Y\text{ is simple}\right)}\\
        &= o(1).
    \end{align*}
    This is a contradiction to Theorem~\ref{thm:ptf-hardness}, and thus the same PTF hardness applies to $\sQ$ and $\sP$, and similarly to $\sQ^b$ and $\sP^b$.
\end{proof}

\begin{proof}[Proof of Corollary~\ref{cor:low-deg-hardness}]
    Assume for contradiction that some $f \in \RR[Y]_{\le D}$ achieves strong separation between $\sQ^S$ and $\sP^S$. Then,
    \begin{align*}
        \left|\EE_{\sP^S}[f(Y)] - \EE_{\sQ^S}[f(Y)]\right| = \omega\left(\max\left\{\sqrt{\Var_{\sP^S}[f(Y)]}, \sqrt{\Var_{\sP^S}[f(Y)]}\right\}\right).
    \end{align*}
    Consider the symmetric polynomial $g \in \RR_{\sym}[Y]_{\le D}$ defined by
    \begin{align*}
        g(Y) = \frac{1}{n!} \sum_{\pi \in \Sym([n])} f(\pi \cdot Y).
    \end{align*}
    Note that since $\sP^S$ and $\sQ^S$ are both permutation-invariant, we have
    \begin{align*}
        \EE_{\sP^S}[g(Y)] &= \EE_{\sP^S}[f(Y)],\\
        \EE_{\sQ^S}[g(Y)] &= \EE_{\sQ^S}[f(Y)].
    \end{align*}
    By Jensen's inequality and permutation-invariance of $\sP^S$,
    \begin{align*}
        \Var_{\sP^S}[g(Y)] &= \EE_{\sP^S}\left[\left(g(Y) - \EE_{\sP^S}[g(Y)]\right)^2\right]\\
        &= \EE_{\sP^S}\left[\left(\frac{1}{n!} \sum_{\pi \in \Sym([n])}f(\pi \cdot Y) - \EE_{\sP^S}[f(Y)]\right)^2\right]\\
        &\le \frac{1}{n!} \sum_{\pi \in \Sym([n])} \EE_{\sP^S}\left[\left(f(\pi \cdot Y) - \EE_{\sP^S}[f(Y)]\right)^2\right] \quad && (\text{Jensen's inequality})\\
        &= \frac{1}{n!} \sum_{\pi \in \Sym([n])} \EE_{\sP^S}\left[\left(f(Y) - \EE_{\sP^S}[f(Y)]\right)^2\right] \quad && (\sP^S\text{ is permutation-invariant})\\
        &= \Var_{\sP^S}[f(Y)].
    \end{align*}
    Similarly, $\Var_{\sQ^S}[g(Y)] \le \Var_{\sQ^S}[f(Y)]$. Consequently,
    \begin{align*}
        &\left|\EE_{\sP^S}[g(Y)] - \EE_{\sQ^S}[g(Y)]\right| = \left|\EE_{\sP^S}[f(Y)] - \EE_{\sQ^S}[f(Y)]\right|\\ 
        &= \omega\left(\max\left\{\sqrt{\Var_{\sP^S}[f(Y)]}, \sqrt{\Var_{\sP^S}[f(Y)]}\right\}\right) = \omega\left(\max\left\{\sqrt{\Var_{\sP^S}[g(Y)]}, \sqrt{\Var_{\sP^S}[g(Y)]}\right\}\right).
    \end{align*}
    By Chebyshev's inequality, the symmetric PTF specified by $g(Y) - \frac{\EE_{\sP^S}[g(Y)] + \EE_{\sQ^S}[g(Y)]}{2} \in \RR_{\sym}[Y]_{\le D}$ achieves strong detection between $\sP^S$ and $\sQ^S$, contradicting that no symmetric PTF of degree at most $D$ achieves strong detection. We therefore conclude that no polynomial of degree at most $D$ achieves strong separation between $\sQ^S$ and $\sP^S$.
    
    The proof is identical for the bipartite versions.
\end{proof}

\section{Proof of Distribution of Short Cycle Counts}\label{sec:proof-cycle-counts}

In this section, we prove the convergence to independent Poissons for short cycle counts as stated in Theorem~\ref{thm:cycle-counts}. We need the following Chen-Stein method for Poisson approximation.

\begin{theorem}[{\cite[Theorem 2]{arratia1989two}}]\label{thm:chen-stein}
    Let $L \in \NN$, $I$ be a finite index set, and $I = I_1 \sqcup I_2 \sqcup \dots \sqcup I_L$ be a partition of $I$. Let $(X_{\alpha})_{\alpha \in I}$ be a collection of Bernoulli variables, with $p_{\alpha} \colonequals \EE[X_{\alpha}]$. Form a collection of random variables $(Z_{i})_{i=1}^L$ by
    \begin{align*}
        Z_i \colonequals \sum_{\alpha \in I_i} X_{\alpha}.
    \end{align*}
    Let $(N({\alpha}))_{\alpha \in I}$ be a collection of subsets of indices, where $N({\alpha}) \subseteq I$ for every $\alpha \in I$.
    
    If we denote
    \begin{align*}
        b_1 &\colonequals \sum_{\alpha \in I} \sum_{\beta \in N({\alpha})} p_{\alpha} p_{\beta},\\
        b_2 &\colonequals \sum_{\alpha \in I} \sum_{\substack{\beta \in N(\alpha):\\ \beta \ne \alpha}} \EE[X_{\alpha} X_{\beta}],\\
        b_3 &\colonequals \sum_{\alpha \in I} \EE\Big|\EE[X_{\alpha} - p_{\alpha} \vert \sigma(X_{\beta}: \beta \not\in N(\alpha))]\Big|,
    \end{align*}
    then
    \begin{align*}
        d_{\TV}\left(\sL(Z_1, Z_2, \dots, Z_L), \bigotimes_{i=1}^L \Pois(\nu_i)\right) &\le 2(2b_1 + 2b_2 + b_3),
    \end{align*}
    where $\nu_i \colonequals \sum_{\alpha \in I_i} p_{\alpha}$.
\end{theorem}

\paragraph{Proof of Theorem~\ref{thm:cycle-counts}}

We first work with the non-bipartite case. Suppose $H$ is a base $d$-regular graph on $k$ vertices potentially with half-loops. Let $M$ be its adjacency matrix, and $B$ be its non-backtracking matrix as in Definition~\ref{def:nb-matrix}.

Since $Z_t$ the number of length-$t$ cycles is invariant under vertex relabelling, instead of working with $\sP$ where the fiber partition of the vertex set $V = \bigsqcup_{i=1}^k V_i$ is a uniformly random equitable partition, we may consider a model $\Tilde{\sP}$ of noisy random lifts where the underlying fiber partition $V = \bigsqcup_{i=1}^k V_i$ is fixed. To be precise, let $\Tilde{\sP} = \Tilde{\sP}_n(H, \delta)$ be the distribution of $r$-noisy random lift of $H$ where $r = \lceil \delta \frac{dn}{2} \rceil$, whose fiber partition is fixed as $V = \bigsqcup_{i=1}^k V_i$ for $V_i \colonequals \{(i-1)m+1, (i-1)m+2, \dots, im\}$ with $m = \frac{n}{k}$. The distribution of cycle counts for $\Tilde{\sP}$ is thus the same as that for $\sP$. 

\paragraph{How we index the edges:}

Before we move on to applying the Chen-Stein method, we explain our indexing notation for the edges of the noisy random lift $G$. Recall
\begin{align*}
    E(H) &= E_1(H) \sqcup E_2(H),
\end{align*}
where $E_1(H)$ is the collection of the half-loops, and $E_2(H)$ is the collection of edges connecting distinct vertices. We further recall the set of directed edges $\vec{E}(H)$ discussed in Remark~\ref{rem:directed-edges}.

In a noisy random lift $G$, for every vertex $v \in V_i, i\in [k]$, we attach a stub $(v , \vec{e})$ to $v$ for every directed edge $\vec{e} \in \vec{E}(H)$ with $\tail(\vec{e}) = i$. To form the noiseless random lift, we independently draw a random bijection $\sigma_e: V_i \to V_j$ for every $e \in E_2(H)$ connecting distinct vertices $i\ne j$ and a random matching $\sigma_e: V_i \to V_i$ for every $e \in E_1(H)$ attached to $i$, and then add edges by connecting the stub-pairs 
\begin{align*}
    \{(v, (i,e)), (\sigma_e(v), (j,e))\}, \quad \text{ for every } e \in E_2(H) \text{ connecting } i\ne j, v \in V_i,\\
    \text{and } \{(v, (i,e)), (\sigma_e(v), (i,e))\}, \quad \text{ for every } e \in E_1(H) \text{ attached to } i, v \in V_i.
\end{align*}
The noisy random lift is then obtained by first deleting a random subset of $r$ edges, exposing $2r$ stubs, and then adding back $r$ edges by a uniformly random perfect matching on the exposed stubs. In the final graph of the noisy random lift, every edge is a stub-pair $\{(v_1, \vec{e}_1), (v_2, \vec{e}_2)\}$.

Now we explain the setup on which we apply the Chen-Stein method.

\paragraph{Setup:} Let $D \in \NN$ be a parameter to be set later, such that $D = o\left(n^{1/2}\right)$. For $1 \le t \le D$, let $I_t$ consist of local configurations of length-$t$ cycles. The precise definition of local configurations of cycles is as follows.
\begin{definition}[Local Configurations]
    A local configuration of a cycle of length $t$ is specified by:
    \begin{itemize}
        \item A sequence of $t$ distinct vertices, $(v_1, v_2, \dots, v_t)$, determined up to a cyclic permutation and reversal.
        \item A type $\in \{\originalretained, \rematched\}$ associated to the pair of consecutive vertices $(v_i, v_{i+1})$ for every $i \in [t]$, where we denote $v_{t+1} \colonequals v_1$.
        \item For every $\{v_i, v_{i+1}\}$ of type $\originalretained$, a color $e \in E(H)$.
        \item For every $\{v_i, v_{i+1}\}$ of type $\rematched$, two additional vertices $s_i$ and $t_i$ that are not necessarily distinct from $v_i$ and $v_{i+1}$. Assign types to $\{v_i, s_i\}$ and $\{v_{i+1}, t_i\}$ as $\originaldeleted$. Assign a color $e_s \in E(H)$ to $\{v_i, s_i\}$, and a color $e_t \in E(H)$ to $\{v_{i+1}, t_i\}$.
    \end{itemize}
\end{definition}

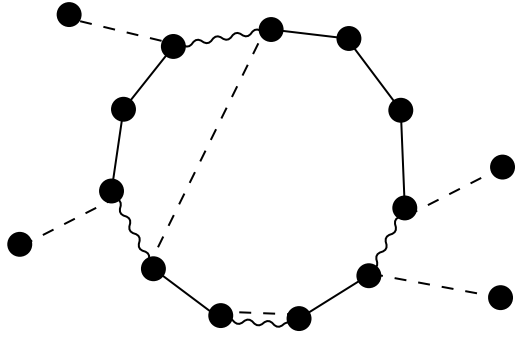
\begin{figure}
    \centering

\tikzset{every picture/.style={line width=0.75pt}} 

\begin{tikzpicture}[x=0.75pt,y=0.75pt,yscale=-1,xscale=1]

\draw    (160,64.5) -- (135,96) ;
\draw    (259,178.5) -- (223,201) ;
\draw    (247,60.5) -- (209,56) ;
\draw    (180,198.5) -- (150,176) ;
\draw    (135,96) -- (129,137) ;
\draw    (247,60.5) -- (274,96.5) ;
\draw    (274,96.5) -- (276,145.5) ;
\draw  [dash pattern={on 4.5pt off 4.5pt}]  (132,140) -- (82,165.5) ;
\draw  [fill={rgb, 255:red, 0; green, 0; blue, 0 }  ,fill opacity=1 ] (123.25,137) .. controls (123.25,133.82) and (125.82,131.25) .. (129,131.25) .. controls (132.18,131.25) and (134.75,133.82) .. (134.75,137) .. controls (134.75,140.18) and (132.18,142.75) .. (129,142.75) .. controls (125.82,142.75) and (123.25,140.18) .. (123.25,137) -- cycle ;
\draw  [fill={rgb, 255:red, 0; green, 0; blue, 0 }  ,fill opacity=1 ] (129.25,96) .. controls (129.25,92.82) and (131.82,90.25) .. (135,90.25) .. controls (138.18,90.25) and (140.75,92.82) .. (140.75,96) .. controls (140.75,99.18) and (138.18,101.75) .. (135,101.75) .. controls (131.82,101.75) and (129.25,99.18) .. (129.25,96) -- cycle ;
\draw  [fill={rgb, 255:red, 0; green, 0; blue, 0 }  ,fill opacity=1 ] (144.25,176) .. controls (144.25,172.82) and (146.82,170.25) .. (150,170.25) .. controls (153.18,170.25) and (155.75,172.82) .. (155.75,176) .. controls (155.75,179.18) and (153.18,181.75) .. (150,181.75) .. controls (146.82,181.75) and (144.25,179.18) .. (144.25,176) -- cycle ;
\draw  [fill={rgb, 255:red, 0; green, 0; blue, 0 }  ,fill opacity=1 ] (178,199.5) .. controls (178,196.32) and (180.57,193.75) .. (183.75,193.75) .. controls (186.93,193.75) and (189.5,196.32) .. (189.5,199.5) .. controls (189.5,202.68) and (186.93,205.25) .. (183.75,205.25) .. controls (180.57,205.25) and (178,202.68) .. (178,199.5) -- cycle ;
\draw  [fill={rgb, 255:red, 0; green, 0; blue, 0 }  ,fill opacity=1 ] (154.25,64.5) .. controls (154.25,61.32) and (156.82,58.75) .. (160,58.75) .. controls (163.18,58.75) and (165.75,61.32) .. (165.75,64.5) .. controls (165.75,67.68) and (163.18,70.25) .. (160,70.25) .. controls (156.82,70.25) and (154.25,67.68) .. (154.25,64.5) -- cycle ;
\draw  [fill={rgb, 255:red, 0; green, 0; blue, 0 }  ,fill opacity=1 ] (217.25,201) .. controls (217.25,197.82) and (219.82,195.25) .. (223,195.25) .. controls (226.18,195.25) and (228.75,197.82) .. (228.75,201) .. controls (228.75,204.18) and (226.18,206.75) .. (223,206.75) .. controls (219.82,206.75) and (217.25,204.18) .. (217.25,201) -- cycle ;
\draw  [fill={rgb, 255:red, 0; green, 0; blue, 0 }  ,fill opacity=1 ] (252.25,179.5) .. controls (252.25,176.32) and (254.82,173.75) .. (258,173.75) .. controls (261.18,173.75) and (263.75,176.32) .. (263.75,179.5) .. controls (263.75,182.68) and (261.18,185.25) .. (258,185.25) .. controls (254.82,185.25) and (252.25,182.68) .. (252.25,179.5) -- cycle ;
\draw  [fill={rgb, 255:red, 0; green, 0; blue, 0 }  ,fill opacity=1 ] (270.25,145.5) .. controls (270.25,142.32) and (272.82,139.75) .. (276,139.75) .. controls (279.18,139.75) and (281.75,142.32) .. (281.75,145.5) .. controls (281.75,148.68) and (279.18,151.25) .. (276,151.25) .. controls (272.82,151.25) and (270.25,148.68) .. (270.25,145.5) -- cycle ;
\draw  [fill={rgb, 255:red, 0; green, 0; blue, 0 }  ,fill opacity=1 ] (268.25,96.5) .. controls (268.25,93.32) and (270.82,90.75) .. (274,90.75) .. controls (277.18,90.75) and (279.75,93.32) .. (279.75,96.5) .. controls (279.75,99.68) and (277.18,102.25) .. (274,102.25) .. controls (270.82,102.25) and (268.25,99.68) .. (268.25,96.5) -- cycle ;
\draw  [fill={rgb, 255:red, 0; green, 0; blue, 0 }  ,fill opacity=1 ] (203.25,56) .. controls (203.25,52.82) and (205.82,50.25) .. (209,50.25) .. controls (212.18,50.25) and (214.75,52.82) .. (214.75,56) .. controls (214.75,59.18) and (212.18,61.75) .. (209,61.75) .. controls (205.82,61.75) and (203.25,59.18) .. (203.25,56) -- cycle ;
\draw  [fill={rgb, 255:red, 0; green, 0; blue, 0 }  ,fill opacity=1 ] (242.25,60.5) .. controls (242.25,57.32) and (244.82,54.75) .. (248,54.75) .. controls (251.18,54.75) and (253.75,57.32) .. (253.75,60.5) .. controls (253.75,63.68) and (251.18,66.25) .. (248,66.25) .. controls (244.82,66.25) and (242.25,63.68) .. (242.25,60.5) -- cycle ;
\draw  [dash pattern={on 4.5pt off 4.5pt}]  (147,176) -- (206,56) ;
\draw  [dash pattern={on 4.5pt off 4.5pt}]  (165.75,64.5) -- (107.75,50.5) ;
\draw  [dash pattern={on 4.5pt off 4.5pt}]  (324,190.5) -- (259,178.5) ;
\draw  [dash pattern={on 4.5pt off 4.5pt}]  (325,125) -- (275,150.5) ;
\draw  [dash pattern={on 4.5pt off 4.5pt}]  (223,199) -- (183.75,197.5) ;
\draw    (209,56) .. controls (207.64,57.93) and (206,58.21) .. (204.07,56.85) .. controls (202.14,55.5) and (200.5,55.78) .. (199.15,57.71) .. controls (197.79,59.64) and (196.15,59.92) .. (194.22,58.56) .. controls (192.29,57.21) and (190.65,57.49) .. (189.29,59.42) .. controls (187.94,61.35) and (186.3,61.63) .. (184.37,60.27) .. controls (182.44,58.92) and (180.8,59.2) .. (179.44,61.13) .. controls (178.09,63.06) and (176.45,63.34) .. (174.52,61.98) .. controls (172.59,60.63) and (170.95,60.91) .. (169.59,62.84) .. controls (168.23,64.77) and (166.59,65.05) .. (164.66,63.69) -- (160,64.5) -- (160,64.5) ;
\draw    (276,145.5) .. controls (276.69,147.75) and (275.91,149.23) .. (273.66,149.92) .. controls (271.41,150.61) and (270.63,152.09) .. (271.32,154.34) .. controls (272.01,156.59) and (271.23,158.07) .. (268.98,158.76) .. controls (266.73,159.45) and (265.95,160.93) .. (266.64,163.18) .. controls (267.33,165.43) and (266.55,166.9) .. (264.3,167.59) .. controls (262.05,168.28) and (261.27,169.76) .. (261.96,172.01) .. controls (262.65,174.26) and (261.87,175.74) .. (259.62,176.43) -- (258,179.5) -- (258,179.5) ;
\draw    (150,176) .. controls (147.74,175.33) and (146.95,173.86) .. (147.63,171.6) .. controls (148.31,169.34) and (147.52,167.87) .. (145.26,167.2) .. controls (143,166.52) and (142.21,165.05) .. (142.89,162.79) .. controls (143.57,160.53) and (142.78,159.06) .. (140.52,158.39) .. controls (138.26,157.72) and (137.47,156.25) .. (138.15,153.99) .. controls (138.83,151.73) and (138.04,150.26) .. (135.78,149.59) .. controls (133.52,148.91) and (132.73,147.44) .. (133.41,145.18) .. controls (134.09,142.92) and (133.3,141.45) .. (131.04,140.78) -- (129,137) -- (129,137) ;
\draw    (225,204.75) .. controls (223.22,206.29) and (221.56,206.17) .. (220.01,204.39) .. controls (218.47,202.61) and (216.81,202.49) .. (215.03,204.03) .. controls (213.25,205.57) and (211.59,205.45) .. (210.04,203.67) .. controls (208.49,201.89) and (206.83,201.77) .. (205.05,203.31) .. controls (203.27,204.85) and (201.61,204.73) .. (200.06,202.95) .. controls (198.52,201.17) and (196.86,201.05) .. (195.08,202.59) .. controls (193.3,204.13) and (191.64,204.01) .. (190.09,202.23) .. controls (188.54,200.45) and (186.88,200.33) .. (185.1,201.87) .. controls (183.32,203.41) and (181.66,203.29) .. (180.12,201.51) -- (180,201.5) -- (180,201.5) ;
\draw  [fill={rgb, 255:red, 0; green, 0; blue, 0 }  ,fill opacity=1 ] (102,48.5) .. controls (102,45.32) and (104.57,42.75) .. (107.75,42.75) .. controls (110.93,42.75) and (113.5,45.32) .. (113.5,48.5) .. controls (113.5,51.68) and (110.93,54.25) .. (107.75,54.25) .. controls (104.57,54.25) and (102,51.68) .. (102,48.5) -- cycle ;
\draw  [fill={rgb, 255:red, 0; green, 0; blue, 0 }  ,fill opacity=1 ] (77.25,163.75) .. controls (77.25,160.57) and (79.82,158) .. (83,158) .. controls (86.18,158) and (88.75,160.57) .. (88.75,163.75) .. controls (88.75,166.93) and (86.18,169.5) .. (83,169.5) .. controls (79.82,169.5) and (77.25,166.93) .. (77.25,163.75) -- cycle ;
\draw  [fill={rgb, 255:red, 0; green, 0; blue, 0 }  ,fill opacity=1 ] (319.25,125) .. controls (319.25,121.82) and (321.82,119.25) .. (325,119.25) .. controls (328.18,119.25) and (330.75,121.82) .. (330.75,125) .. controls (330.75,128.18) and (328.18,130.75) .. (325,130.75) .. controls (321.82,130.75) and (319.25,128.18) .. (319.25,125) -- cycle ;
\draw  [fill={rgb, 255:red, 0; green, 0; blue, 0 }  ,fill opacity=1 ] (318.25,190.5) .. controls (318.25,187.32) and (320.82,184.75) .. (324,184.75) .. controls (327.18,184.75) and (329.75,187.32) .. (329.75,190.5) .. controls (329.75,193.68) and (327.18,196.25) .. (324,196.25) .. controls (320.82,196.25) and (318.25,193.68) .. (318.25,190.5) -- cycle ;

\end{tikzpicture}
    \caption{An example of a local configuration of a cycle of length $11$: the solid edges are the edges of type $\originalretained$, the dashed edges are the edges of type $\originaldeleted$, and the squiggly edges are the edges of type $\rematched$.}
    \label{fig:placeholder}
\end{figure}

\begin{remark}
    Let $G \sim \Tilde{\sP}$ be a random lift. We say a local configuration of a length-$t$ cycle as in the above definition is realized in $G$ if:
\begin{itemize}
    \item Every pair $\{v_i, v_{i+1}\}$ is connected in $G$, where $v_{t+1} \colonequals v_1$.
    \item For every $\{u, v\}$ of type $\originalretained$ and of color $e \in E(H)$, the edge $\{(u,\vec{e}), (v,\iota(\vec{e}))\}$ is in the original noiseless lift and is not deleted in the random deletion process, where $\vec{e}$ is an directed edge corresponding to $e$.
    \item For every $\{u, v\}$ of type $\originaldeleted$ and of color $e \in E(H)$, the edge $\{(u,\vec{e}), (v,\iota(\vec{e}))\}$ is in the original noiseless and is later deleted in the random deletion process, where $\vec{e}$ is an directed edge corresponding to $e$.
    \item For every $\{v_i, v_{i+1}\}$ of type $\rematched$, after the prescribed edges $\{(v_i,\vec{e}_s), (s_i,\iota(\vec{e}_s))\}$ and \newline $\{(v_{i+1},\vec{e}_t), (t_i,\iota(\vec{e}_t))\}$, where $\vec{e}_s$ and $\vec{e}_t$ are directed edges corresponding to $e_s$ and $e_t$, are deleted from the noiseless random lift, the two exposed stubs are paired to form the edge $\{(v_i,\vec{e}_s), (v_{i+1},\vec{e}_t)\}$ in the rematching process.
\end{itemize}
    We point out that $\{v_i, v_{i+1}\}$ can both have type $\originaldeleted$ and $\rematched$, since it is possible for an edge to be first included in the original noiseless random lift, then deleted, and finally rematched after the deletion.

    We also note that while we use the phrase ``a local configuration is realized in $G \sim \Tilde{\sP}$'', the realization actually depends on more than just the final observed graph $G$, and needs the full information on the noiseless random lift before the random deletion, the deletion process, and the rematching process. In this section, we will always work with the full information of $G \sim \Tilde{\sP}$ and abuse the notation to simply say ``a local configuration is realized in $G$''.

    In plain language, a local configuration of a length-$t$ cycle describes the local information of how a length-$t$ cycle can be formed in $G \sim \Tilde{\sP}$. It specifies the sequence of vertices $(v_1, \dots, v_t)$ of the cycle, and whether each consecutive pair is added in the original noiseless lift that is not deleted ($\originalretained$) or paired in the rematching process after the deletion process ($\rematched$). 
    
    Moreover, for every $\{v_i, v_{i+1}\}$ of type $\rematched$ that should be paired in the rematching process, the configuration further specifies which edges $\{(v_i,\vec{e}_s), (s_i,\iota(\vec{e}_s))\}$ and $\{(v_{i+1},\vec{e}_t), (t_i,\iota(\vec{e}_t))\}$ should be deleted from the original noiseless lift, after which the two exposed stubs $(v_i,\vec{e}_s), (v_{i+1},\vec{e}_t)$ incident to $v_i$ and $v_{i+1}$ after the deletion should be paired in the rematching process.

    Finally, for every $\{u,v\}$ of type $\in \{\originalretained, \originaldeleted\}$, the configuration specifies a color $e \in E(H)$ so that the edge $\{(u,\vec{e}), (v,\iota(\vec{e}))\}$ should be in the noiseless random lift.
\end{remark}

    \begin{definition}[Color Classes]
        Recall in Definition~\ref{def:random-lift}, the edges in the noiseless random lift are naturally partitioned according to $E(H)$. For $e \in E(H)$, the color class $\sC_{e}$ is the collection of the edges $\{(u,\vec{e}), (v,\iota(\vec{e}))\}$ where $\vec{e} \in \vec{E}(H)$ is a directed edge corresponding to $e$, i.e., the edges added to the noiseless random lift by $\sigma_e$.
    \end{definition}

    \begin{remark}
        We further explain some notations we will use to reason about local configurations of cycles. For $\alpha$ a local configuration of a length-$t$ cycle, we let $\stub(\alpha)$ denote the collection of stubs used by any edge of $\alpha$ of type $\in \{\originalretained, \originaldeleted, \rematched\}$. We will use $\mathcal{O}(\alpha)$ to denote the collection of edges in $\alpha$ of type $\originalretained$ or $\originaldeleted$, $\mathcal{O}^{\ret}(\alpha)$ to denote the collection of edges in $\alpha$ of type $\originalretained$, $\mathcal{O}^{\del}(\alpha)$ to denote the collection of edges in $\alpha$ of type $\originaldeleted$, and $\mathcal{R}(\alpha)$ to denote the collection of edges in $\alpha$ of type $\rematched$. We will also use $\mathcal{O}^{\del, 2}(\alpha)$ to denote the collection of edges in $\alpha$ of type $\originaldeleted$, whose stubs are both incident to $\mathcal{R}(\alpha)$. Note that $\mathcal{O}^{\ret}(\alpha) \sqcup \mathcal{O}^{\del}(\alpha) = \mathcal{O}(\alpha)$, $\mathcal{O}^{\del,2}(\alpha) \subseteq \mathcal{O}^{\del}(\alpha)$, and $\mathcal{O}^{\del,2}(\alpha)$ might share some edges with $\mathcal{R}(\alpha)$.

        For $e \in E(H)$, we will use $\mathcal{O}_{e}(\alpha)$ to denote the collection of edges in $\mathcal{O}(\alpha)$ of color $e$, and let $c_{e}(\alpha) \colonequals |\mathcal{O}_{e}(\alpha)|$. We will use $c(\alpha) \colonequals |\mathcal{O}(\alpha)|$, $c^{\ret}(\alpha) \colonequals |\mathcal{O}^{\ret}(\alpha)|$, $c^{\del}(\alpha) \colonequals |\mathcal{O}^{\del}(\alpha)|$, $c^{\del,2}(\alpha) \colonequals |\mathcal{O}^{\del,2}(\alpha)|$, and $\rho(\alpha) \colonequals |\mathcal{R}(\alpha)|$.
    
        We call a local configuration $\alpha$ feasible, if $\alpha$ can ever be realized in $G \sim \Tilde{\sP}$. Note that in particular, the following holds for a feasible $\alpha$.
        \begin{itemize}
            \item All the edges in $\mathcal{O}(\alpha)$ use distinct stubs, and $\stub(\alpha)$ is exactly the collections of stubs incident to $\mathcal{O}(\alpha)$. Thus, $|\stub(\alpha)| = 2|\mathcal{O}(\alpha)| = 2e(\alpha)$.
            \item The edges in $\mathcal{O}(\alpha)$ are partitioned as $\mathcal{O}(\alpha) = \bigsqcup_{e \in E(H)} \mathcal{O}_{e}(\alpha)$.
            \item For every stub incident to an edge of type $\rematched$ in $\alpha$, it is also incident to an edge of type $\originaldeleted$ in $\alpha$.
            \item For every edge of type $\originaldeleted$ in $\alpha$, at least one of its stubs is incident to an edge of type $\rematched$. Every edge in $\mathcal{O}^{\del}(\alpha) \setminus \mathcal{O}^{\del, 2}(\alpha)$ has exactly one stub incident to $\mathcal{R}(\alpha)$, and every edge in $\mathcal{O}^{\del,2}(\alpha)$ has both stubs incident to $\mathcal{R}(\alpha)$. Since the collection of stubs incident to $\mathcal{R}(\alpha)$ is distributed between $\mathcal{O}^{\del,2}(\alpha)$ and $\mathcal{O}^{\del}(\alpha) \setminus \mathcal{O}^{\del,2}(\alpha)$, we have that $2 |\mathcal{O}^{\del,2}(\alpha)| + |\mathcal{O}^{\del}(\alpha) \setminus \mathcal{O}^{\del,2}(\alpha)| = 2|\mathcal{R}(\alpha)|$.
        \end{itemize}
        The quantities $c(\alpha), c_{e}(\alpha), c^{\ret}(\alpha), c^{\del}(\alpha), c^{\del,2}(\alpha)$, and $\rho(\alpha)$ satisfy the following relations.
        \begin{enumerate}
            \item $c(\alpha) = \sum_{e \in E(H)} c_{e}(\alpha)$.
            \item $c(\alpha) = c^{\ret}(\alpha) + c^{\del}(\alpha)$.
            \item $c^{\ret}(\alpha) + \rho(\alpha) = t$, if $\alpha$ is a feasible local configuration of a length-$t$ cycle.
            \item $c^{\del,2}(\alpha) = 2\rho(\alpha) - c^{\del}(\alpha)$.
            \item $\rho(\alpha) \le c^{\del}(\alpha) \le 2\rho(\alpha)$, $0 \le c^{\del,2}(\alpha) \le c^{\del}(\alpha)$.
        \end{enumerate}
        Relation $4$ and $5$ can be deduced as follows. Since $2 |\mathcal{O}^{\del,2}(\alpha)| + |\mathcal{O}^{\del}(\alpha) \setminus \mathcal{O}^{\del,2}(\alpha)| = 2|\mathcal{R}(\alpha)|$, we have
        $2c^{\del,2}(\alpha) + (c^{\del}(\alpha) - c^{\del,2}(\alpha)) = 2\rho(\alpha)$, and thus $c^{\del,2}(\alpha) = 2\rho(\alpha) - c^{\del}(\alpha)$. Since $0 \le c^{\del,2}(\alpha) \le c^{\del}(\alpha)$, we have $\rho(\alpha) = \frac{c^{\del}(\alpha) + c^{\del,2}(\alpha)}{2} \le c^{\del}(\alpha) = 2\rho(\alpha) - c^{\del,2}(\alpha) \le 2\rho(\alpha)$.
    \end{remark}

    For $1 \le t \le D$, let $I_t$ denote the collection of feasible local configurations of length-$t$ cycles, and let $I \colonequals \bigsqcup_{i=1}^t I_t$. For a feasible local configuration $\alpha \in I$, let $X_{\alpha}$ denote the indicator variable that $\alpha$ is realized in $G \sim \Tilde{\sP}$. For $1 \le t \le D$, define
    \begin{align*}
        Z_t \colonequals \sum_{\alpha \in I_t} X_{\alpha}.
    \end{align*}
    Since every length-$t$ cycle in $G \sim \Tilde{\sP}$ has a unique realized local configuration $\alpha \in I_t$, $Z_t$ counts the number of length-$t$ cycles in $G$.

    Let $N(\alpha) \subseteq I$ denote the collection of local configurations whose stubs are not disjoint from those of $\alpha$, i.e.,
    \begin{align*}
        N(\alpha) \colonequals \{\beta \in I: \stub(\beta) \cap \stub(\alpha) \ne \emptyset\}.
    \end{align*}

    Since we will apply Theorem~\ref{thm:chen-stein} to prove Theorem~\ref{thm:cycle-counts}, the proof is naturally split into four parts: bounding $b_1$ in Section~\ref{sec:b1-bound}, bounding $b_2$ in Section~\ref{sec:b2-bound}, bounding $b_3$ in Section~\ref{sec:b3-bound}, and computing the mean $\nu_t$ in Section~\ref{sec:mean-computation}.

    \subsection{Auxiliary Results}

    The following gives the exact probability that a feasible local configuration $\alpha$ is realized.
    \begin{fact}\label{fact:config-prob}
        Let $\alpha$ be a feasible local configuration of a length-$t$ cycle, and $m \colonequals \frac{n}{k}$. If $c^{\del}(\alpha) > r$, $\EE[X_{\alpha}] = 0$. If $c^{\del}(\alpha) \le r$, then
        \begin{align*}
            \EE[X_{\alpha}] &=\left(\prod_{e \in E_1(H) }\frac{(m-2c_{e}(\alpha)-1)!!}{(m-1)!!}\right) \left(\prod_{e \in E_2(H)} \frac{(m - c_e(\alpha))!}{m!}\right)\left( \frac{\binom{\frac{dn}{2} - c(\alpha)}{r - c^{\del}(\alpha)}}{\binom{\frac{dn}{2}}{r}} \frac{(2r-2\rho(\alpha)-1)!!}{(2r-1)!!} \right).
        \end{align*}
        The first two factors give the exact probability for edges in $\mathcal{O}(\alpha)$ to be paired in the noiseless random lift with their prescribed colors in $E(H)$ before the random deletion process. The last factor gives the exact probability for $\mathcal{O}^{\del}(\alpha)$ to be deleted and $\mathcal{O}(\alpha) \setminus \mathcal{O}^{\del}(\alpha)$ to be retained after the random deletion process, and for edges in $\mathcal{R}(\alpha)$ to be paired in the rematching process.
    \end{fact}

    We may upper bound the probability that $\alpha$ is realized as follows.

    \begin{proposition} \label{prop:p-alpha-bound}
        Let $\alpha \in I_t$ be a feasible local configuration, where $t = o(n^{1/2})$. Then, there exists constant $K_1 = K_1(H) > 0$ such that
        \begin{align*}
            p_{\alpha} = \EE[X_{\alpha}] \le \left(\frac{K_1}{n}\right)^{t + c^{\del}(\alpha)}.
        \end{align*}
    \end{proposition}

    \begin{proof}
        For $\alpha \in I_t$ where $t = o(n^{1/2})$, we have \[\rho(\alpha) \le t = o(n^{1/2}), \quad c^{\del}(\alpha) \le 2\rho(\alpha) = o(n^{1/2}), \quad c(\alpha) = c^{\ret}(\alpha) + c^{\del}(\alpha) = o(n^{1/2}).\] 
    By Fact~\ref{fact:config-prob}, if $c^{\del}(\alpha) > r$, then $p_{\alpha} = 0$. If $c^{\del}(\alpha) \le r$,
    since $m = \frac{n}{k} = \Omega(n)$, for a constant $K = K(H) > 0$ the following holds
    \begin{align*}
        &p_{\alpha} = \EE[X_{\alpha}]\\
        &= \left(\prod_{e \in E_1(H) }\frac{(m-2c_{e}(\alpha)-1)!!}{(m-1)!!}\right) \left(\prod_{e \in E_2(H)} \frac{(m - c_e(\alpha))!}{m!}\right)\left( \frac{\binom{\frac{dn}{2} - c(\alpha)}{r - c^{\del}(\alpha)}}{\binom{\frac{dn}{2}}{r}} \frac{(2r-2\rho(\alpha)-1)!!}{(2r-1)!!} \right)\\
        &\le \left(\frac{K}{n}\right)^{c(\alpha)} \frac{\frac{r!}{(r-c^{\del}(\alpha))!} \frac{\left(\frac{dn}{2} - r\right)!}{\left(\frac{dn}{2} - r - (c(\alpha) - c^{\del}(\alpha)))\right)}}{\frac{\left(\frac{dn}{2}\right)!}{\left(\frac{dn}{2} - c(\alpha)\right)!} } \frac{(2r-2\rho(\alpha)-1)!!}{(2r-1)!!}\\
        &\le \left(\frac{2K}{n}\right)^{c(\alpha)} \frac{\left(\frac{dn}{2} - r\right)^{c(\alpha) - c^{\del}(\alpha)}}{\left(\frac{dn}{2}\right)^{c(\alpha)}} \frac{r!}{(r-c^{\del}(\alpha))!}\frac{(r-\rho(\alpha))!}{r!}\\
        &\le \left(\frac{2K}{n}\right)^{c(\alpha)} \frac{\left(\frac{dn}{2} - r\right)^{c(\alpha) - c^{\del}(\alpha)} r^{c^{\del}(\alpha) - \rho(\alpha)}}{\left(\frac{dn}{2}\right)^{c(\alpha)}}\\
        &\le \left(\frac{2K}{n}\right)^{c(\alpha)} \frac{1}{\left(\frac{dn}{2}\right)^{\rho(\alpha)}}.
    \end{align*}
    Since
    \begin{align*}
        c(\alpha) &= c^{\ret}(\alpha) + c^{\del}(\alpha) = (t - \rho(\alpha)) + c^{\del}(\alpha),
    \end{align*}
    we conclude that for some constant $K_1 = K_1(H) > 0$,
    \begin{align}
        p_{\alpha} \le \left(\frac{K_1}{n}\right)^{t+c^{\del}(\alpha)}, \label{ineq:p-bound}
    \end{align}
    regardless of whether $c^{\del}(\alpha) \le r$.
    \end{proof}

    Next, we establish a simple bound on the sum of the probabilities of $\alpha \in I$ that we will make repeated use of.
    \begin{proposition}\label{prop:p-sum-bound}
        For $D = o(n^{1/2})$, there exists a constant $K_2 = K_2(H) > 0$ such that
        \begin{align*}
            \sum_{\alpha \in I} p_{\alpha} \le 2(D+1)\sum_{t=1}^D K_2^t.
        \end{align*}
    \end{proposition}

    \begin{proof}
        For a pair of $(f,h)$ such that $0 \le h \le t$ and $h \le f \le 2h$ \footnote{Note that the restrictions $0 \le h \le t$ and $h \le f \le 2h$ are enforced because for any feasible $\alpha \in I_t$, $0 \le \rho(\alpha) \le t$ and $\rho(\alpha) \le c^{\del}(\alpha) \le 2\rho(\alpha)$.}, the number of $\alpha \in I_t$ with $f = c^{\del}(\alpha)$ and $h = \rho(\alpha)$ can be bounded by
    \begin{align*}
        &n^t d^{t} \cdot \binom{t}{h} \cdot \binom{2h}{2f-2h} \cdot (nd)^{2f - 2h} \cdot (4h-2f-1)!! \\
        &\le (2d n)^{t + 2f-2h} \frac{(2h)!(4h-2f-1)!!}{(2f-2h)!(4h-2f)!}\\
        &\le (2d n)^{t + 2f-2h} (2h)^{4h-2f}, \numberthis \label{ineq:count-bound}
    \end{align*}
    since
    \begin{itemize}
        \item $n^t d^{t}$ bounds the number of choices for the $t$ cycle vertices and the color choices $e \in E(H)$ for the cycle edges of type $\originalretained$,
        \item $\binom{t}{h} = \binom{t}{\rho(\alpha)}$ bounds the number of choices of the cycle edges of the type $\rematched$,
        \item $\binom{2h}{2f-2h} = \binom{2\rho(\alpha)}{2(c^{\del}(\alpha) - \rho(\alpha))} = \binom{2\rho(\alpha)}{c^{\del}(\alpha) - c^{\del,2}(\alpha)}$ bounds the number of choices of the stubs of $\mathcal{R}(\alpha)$ that are incident to $\mathcal{O}^{\del}(\alpha) \setminus \mathcal{O}^{\del,2}(\alpha)$,
        \item $(nd)^{2f-2h} = (nd)^{c^{\del}(\alpha) - c^{\del,2}(\alpha)}$ bounds the number of choices of stubs that are paired to the $2f-2h$ stubs of $\mathcal{R}(\alpha)$ by $\mathcal{O}^{\del}(\alpha) \setminus \mathcal{O}^{\del,2}(\alpha)$,
        \item and finally, $(4h-2f-1)!! = (2\rho(\alpha) - (c^{\del}(\alpha) - c^{\del,2}(\alpha))-1)!!$ bounds the number of choices of $\mathcal{O}^{\del,2}(\alpha)$.
    \end{itemize}

    Combining \eqref{ineq:count-bound} and Proposition~\ref{prop:p-alpha-bound}, we may bound
    \begin{align*}
        \sum_{\alpha \in I} p_{\alpha} &= \sum_{t=1}^D \sum_{\substack{0 \le h \le t,\\ h \le f \le 2h}} \sum_{\substack{\alpha \in I_t:\\ f = c^{\del}(\alpha),\\
        h = \rho(\alpha)}} p_{\alpha}\\
        &\le \sum_{t=1}^D \sum_{\substack{0 \le h \le t,\\ h \le f \le 2h}} (2d n)^{t + 2f-2h} (2h)^{4h-2f}\left(\frac{K_1}{n}\right)^{t+f}\\
        &\le \sum_{t=1}^D K_2^t\sum_{\substack{0 \le h \le t,\\ h \le f \le 2h}} \left(\frac{h^2}{n}\right)^{2h-f}
        \intertext{for some constant $K_2 = K_2(H) > 0$ since $t+2f-2h, 4h-2f$, and $t+f$ are all at most $3t$,}
        &\le \sum_{t=1}^D K_2^t\sum_{j=0}^\infty (t+1)\left(\frac{t^2}{n}\right)^{j}\\
        &\le 2(D+1) \sum_{t=1}^D K_2^t. \numberthis \label{ineq:p-sum-bound}
    \end{align*}
    \end{proof}

    \subsection{Bound for $b_1$} \label{sec:b1-bound}

    We may upper bound
    \begin{align*}
        b_1 &= \sum_{\alpha \in I} \sum_{\beta \in N(\alpha)} p_{\alpha} p_{\beta}\\
        &\le \sum_{\alpha \in I} p_{\alpha} \sum_{x \in \stub(\alpha)} \sum_{\substack{\beta \in I:\\ x\in \stub(\beta)}} p_{\beta}.
    \end{align*}

    Similar to the proof of Proposition~\ref{prop:p-sum-bound}, for a fixed stub $x$ and a pair $(f,h)$ such that $0 \le h \le t$ and $h \le f \le 2h$, the number of $\alpha \in I_t$ with $f = c^{\del}(\alpha)$, $h = \rho(\alpha)$, and $x \in \stub(\alpha)$ can be bounded by
    \begin{align*}
        &n^{t-1} d^{t} \cdot \binom{t}{h} \cdot \binom{2h}{2f-2h} \cdot (nd)^{2f - 2h} \cdot (4h-2f-1)!!\\
        &\quad+  n^{t} d^{t} \cdot \binom{t}{h} \cdot \binom{2h}{2f-2h} \cdot (nd)^{2f - 2h-1} \cdot (4h-2f-1)!!\\
        &\le 2(2d n)^{t + 2f-2h-1} (2h)^{4h-2f}, \numberthis \label{ineq:count-bound-fixed-stub}
    \end{align*}
    where the first term bounds the situation when the stub $x$ is incident to the $t$ cycle edges, and the second term bounds the situation when the stub $x$ is incident to $\mathcal{O}^{\del}(\alpha) \setminus \mathcal{O}^{\del,2}(\alpha)$ outside of the $2t$ cycle stubs.

    Combining \eqref{ineq:count-bound-fixed-stub} and Proposition~\ref{prop:p-alpha-bound}, for any fixed stub $x$ we may bound
    \begin{align*}
        \sum_{\substack{\alpha \in I:\\ x \in \stub(\alpha)}} p_{\alpha} &= \sum_{t=1}^D \sum_{\substack{0 \le h \le t,\\ h \le f \le 2h}} \sum_{\substack{\alpha \in I_t:\\x \in \stub(\alpha),\\ f = c^{\del}(\alpha),\\
        h = \rho(\alpha)}} p_{\alpha}\\
        &\le \sum_{t=1}^D \sum_{\substack{0 \le h \le t,\\ h \le f \le 2h}} 2(2d n)^{t + 2f-2h-1} (2h)^{4h-2f}\left(\frac{K_1}{n}\right)^{t+f}\\
        &\le 2\sum_{t=1}^D \frac{K_2^t}{n}\sum_{\substack{0 \le h \le t,\\ h \le f \le 2h}} \left(\frac{h^2}{n}\right)^{2h-f}\\
        &\le 4(D+1) \frac{\sum_{t=1}^D K_2^t}{n}. \numberthis \label{ineq:p-sum-bound-fixed-stub}
    \end{align*}

    Finally, since for $\alpha \in I_t$,
    \begin{align*}
        |\stub(\alpha)| \le 2|\mathcal{O}(\alpha)| = 2c(\alpha) = 2(t - \rho(\alpha) + c^{\del}(\alpha)) \le 2(t + \rho(\alpha))  \le 4t \le 4D,
    \end{align*}
    we get\begin{align*}
        b_1 &\le \sum_{\alpha \in I} p_{\alpha} \sum_{x \in \stub(\alpha)} \sum_{\substack{\beta \in I:\\ x\in \stub(\beta)}} p_{\beta}\\
        &\le 4D \sum_{\alpha \in I} p_{\alpha} \cdot 4(D+1) \frac{\sum_{t=1}^D K_2^t}{n} && \quad (\text{Using }\eqref{ineq:p-sum-bound-fixed-stub})\\
        &\le 4D\cdot 4(D+1) \sum_{t=1}^D K_2^t \cdot 2(D+1) \frac{\sum_{t=1}^D K_2^t}{n} && \quad (\text{Using Proposition~\ref{prop:p-sum-bound}})\\
        &= \frac{32D(D+1)^2}{n} \left(\sum_{t=1}^D K_2^t\right)^2,
    \end{align*}
    which is $o(1)$ provided that $D \le c_1\log(n)$ for a constant $c_1 = c_1(H) > 0$.

    \subsection{Bound for $b_2$} \label{sec:b2-bound}

    Consider a pair of feasible local configurations $\alpha \in I_t$ and $\beta \in I_s$. Then, $\EE[X_{\alpha} X_{\beta}]$ is the probability that both configurations $\alpha$ and $\beta$ are realized in $G \sim \Tilde{\sP}$. We say a pair of feasible local configurations are compatible if they can be simultaneously realized. Note that if $\alpha, \beta$ are compatible, then
    \begin{itemize}
        \item $\mathcal{O}(\alpha)$ and $\mathcal{O}(\beta)$ are compatible in the sense that every edge of $\mathcal{O}(\alpha)$ and every edge of $\mathcal{O}(\beta)$ either use disjoint stubs, or are identical. Moreover, if an edge belongs to both $\mathcal{O}(\alpha)$ and $\mathcal{O}(\beta)$, then they receive the same type $\in \{\originalretained, \originaldeleted\}$ and the same color $e \in E(H)$.
        \item $\mathcal{R}(\alpha)$ and $\mathcal{R}(\beta)$ are compatible in the sense that every edge of $\mathcal{R}(\alpha)$ and every edge of $\mathcal{R}(\beta)$ either use disjoint stubs, or are identical.
    \end{itemize}

    For a pair of compatible local configurations $(\alpha, \beta)$, We will denote $\mathcal{O}(\alpha \cup \beta) \colonequals \mathcal{O}(\alpha) \cup \mathcal{O}(\beta)$, $\mathcal{O}^{\ret}(\alpha \cup \beta) \colonequals \mathcal{O}^{\ret}(\alpha) \cup \mathcal{O}^{\ret}(\beta)$, $\mathcal{O}^{\del}(\alpha \cup \beta) \colonequals \mathcal{O}^{\del}(\alpha) \cup \mathcal{O}^{\del}(\beta)$, and $\mathcal{R}(\alpha \cup \beta) \colonequals \mathcal{R}(\alpha) \cup \mathcal{R}(\beta)$. We will also use \[\mathcal{O}^{\del,2}(\alpha \cup \beta) \colonequals \{e \in \mathcal{O}^{\del}(\alpha \cup \beta): \text{ both stubs of } e \text{ are incident to } \mathcal{R}(\alpha \cup \beta)\}\footnote{We note that $\mathcal{O}^{\del,2}(\alpha \cup \beta)$ is different from $\mathcal{O}^{\del,2}(\alpha) \cup \mathcal{O}^{\del,2}(\beta)$. A potential edge $\{x,y\} \in \mathcal{O}^{\del,2}(\alpha \cup \beta)$ can use one stub incident to $\mathcal{R}(\alpha) \setminus \mathcal{R}(\beta)$ and another stub incident to $\mathcal{R}(\beta) \setminus \mathcal{R}(\alpha)$, and thus $\{x,y\} \in \mathcal{O}^{\del}(\alpha) \setminus \mathcal{O}^{\del,2}(\alpha)$ and $\{x,y\} \in \mathcal{O}^{\del}(\beta) \setminus \mathcal{O}^{\del,2}(\beta)$.}\]  to denote the collection of edges in $\mathcal{O}^{\del}(\alpha \cup \beta) \colonequals \mathcal{O}^{\del}(\alpha) \cup \mathcal{O}^{\del}(\beta)$ with both stubs incident to $\mathcal{R}(\alpha \cup \beta) \colonequals \mathcal{R}(\alpha) \cup \mathcal{R}(\beta)$. Let $c(\alpha \cup \beta) \colonequals |\mathcal{O}(\alpha \cup \beta)|$, $c^{\ret}(\alpha \cup \beta) \colonequals |\mathcal{O}^{\ret}(\alpha \cup \beta)|$, $c^{\del}(\alpha \cup \beta) \colonequals |\mathcal{O}^{\del}(\alpha \cup \beta)|$, $c^{\del, 2}(\alpha \cup \beta) \colonequals |\mathcal{O}^{\del, 2}(\alpha \cup \beta)|$, and $\rho(\alpha \cup \beta) \colonequals |\mathcal{R}(\alpha \cup \beta)|$. For $e \in E(H)$, we denote $\mathcal{O}_{e}(\alpha \cup \beta) \colonequals \mathcal{O}_{e}(\alpha) \cup \mathcal{O}_{e}(\beta)$, and let $c_{e}(\alpha \cup \beta) \colonequals |\mathcal{O}_{e}(\alpha \cup \beta)|$.

    Now consider a pair of compatible local configurations $\alpha \in I_t$ and $\beta \in I_s$ such that $\beta \in N(\alpha)$. Then, the following relations hold:
    \begin{enumerate}
        \item $c(\alpha \cup \beta) = \sum_{e \in E(H)} c_{e}(\alpha \cup \beta)$.
        \item $c(\alpha \cup \beta) = c^{\ret}(\alpha \cup \beta) + c^{\del}(\alpha \cup \beta)$.
        \item $c^{\del,2}(\alpha \cup \beta) = 2\rho(\alpha \cup \beta) - c^{\del}(\alpha \cup \beta)$.
        \item $\rho(\alpha \cup \beta) \le c^{\del}(\alpha \cup \beta) \le 2\rho(\alpha \cup \beta)$, $0 \le c^{\del,2}(\alpha \cup \beta) \le c^{\del}(\alpha \cup \beta)$.
        \item Either $c^{\ret}(\alpha \cup \beta) + \rho(\alpha + \beta) < t + s$,\\
        \hspace*{0.6cm} or $c^{\ret}(\alpha \cup \beta) + \rho(\alpha + \beta) = t + s$ and $\rho(\alpha \cup \beta) \le c^{\del}(\alpha \cup \beta) \le 2\rho(\alpha \cup \beta)-1$.
    \end{enumerate}
    Relation $3$ and $4$ can be deduced similarly as before: the stubs of $\mathcal{R}(\alpha \cup \beta)$ are distributed between $\mathcal{O}^{\del,2}(\alpha \cup \beta)$ and $\mathcal{O}^{\del}(\alpha \cup \beta) \setminus \mathcal{O}^{\del,2}(\alpha \cup \beta)$, with every edge in $\mathcal{O}^{\del,2}(\alpha \cup \beta)$ incident to $2$ stubs of $\mathcal{R}(\alpha \cup \beta)$ and every edge in $\mathcal{O}^{\del}(\alpha \cup \beta) \setminus \mathcal{O}^{\del,2}(\alpha \cup \beta)$ incident to $1$ stub of $\mathcal{R}(\alpha \cup \beta)$. Thus, $2 |\mathcal{O}^{\del,2}(\alpha \cup \beta)| + |\mathcal{O}^{\del}(\alpha \cup \beta) \setminus \mathcal{O}^{\del,2}(\alpha \cup \beta)| = 2|\mathcal{R}(\alpha \cup \beta)|$, we have
    $2c^{\del,2}(\alpha \cup \beta) + (c^{\del}(\alpha \cup \beta) - c^{\del,2}(\alpha \cup \beta)) = 2\rho(\alpha \cup \beta)$, and thus $c^{\del,2}(\alpha \cup \beta) = 2\rho(\alpha \cup \beta) - c^{\del}(\alpha \cup \beta)$. Since $0 \le c^{\del,2}(\alpha \cup \beta) \le c^{\del}(\alpha \cup \beta)$, we have $\rho(\alpha \cup \beta) = \frac{c^{\del}(\alpha \cup \beta) + c^{\del,2}(\alpha \cup \beta)}{2} \le c^{\del}(\alpha \cup \beta) = 2\rho(\alpha \cup \beta) - c^{\del,2}(\alpha \cup \beta) \le 2\rho(\alpha \cup \beta)$.

    Next we explain why relation $5$ holds. Since $\beta \in N(\alpha)$, $\stub(\alpha) \cap \stub(\beta) \ne \emptyset$. We consider two cases below:
    \begin{itemize}
        \item Suppose the $2t$ cycle stubs of $\alpha$ and the $2s$ cycle stubs of $\beta$ share at least one stub $a$, then the stub $a$ is incident to a cycle edge $e \in (\mathcal{O}^{\ret}(\alpha) \sqcup \mathcal{R}(\alpha)) \cap (\mathcal{O}^{\ret}(\beta) \sqcup \mathcal{R}(\beta))$ of the same type $\in \{\originalretained, \rematched\}$ and the same color in $E(H)$. Since $\alpha$ and $\beta$ share a cycle edge, the total number of cycle edges in $\alpha \cup \beta$ is then $c^{\ret}(\alpha \cup \beta) + \rho(\alpha + \beta) < t + s$.
        \item  Now suppose $c^{\ret}(\alpha \cup \beta) + \rho(\alpha + \beta) = t + s$, and thus the $2t$ cycle stubs of $\alpha$ and the $2s$ cycle stubs of $\beta$ are disjoint. Then, at least one edge in $\mathcal{O}^{\del}(\alpha)$ is shared with $\mathcal{O}^{\del}(\beta)$, since only edges of type $\originaldeleted$ could be incident to stubs outside of the collection of cycle stubs. Since the stubs of $\mathcal{R}(\alpha)$ and the stubs of $\mathcal{R}(\beta)$ are disjoint, and every edge in $\mathcal{O}^{\del}(\alpha)$ ($\mathcal{O}^{\del}(\beta)$ resp.) is incident to the stubs of $\mathcal{R}(\alpha)$ ($\mathcal{R}(\beta)$ resp.), we know that the shared edge between $\mathcal{O}^{\del}(\alpha)$ and $\mathcal{O}^{\del}(\beta)$ must be incident to one stub in $\mathcal{R}(\alpha)$ and one stub in $\mathcal{R}(\beta)$. Hence, the shared edge belongs to $\mathcal{O}^{\del,2}(\alpha \cup \beta)$. As $c^{\del,2}(\alpha \cup \beta) \ge 1$, relation $4$ and $5$ imply
        \begin{align*}
            \rho(\alpha \cup \beta) \le c^{\del}(\alpha \cup \beta) = 2\rho(\alpha \cup \beta) - c^{\del,2}(\alpha \cup \beta) \le 2\rho(\alpha \cup \beta) - 1.
        \end{align*}
    \end{itemize}
    
    Analogous to Fact~\ref{fact:config-prob}, if $c^{\del}(\alpha \cup \beta) > r$, then $\EE[X_{\alpha} X_{\beta}] = 0$. If $c^{\del}(\alpha \cup \beta) \le r$, since $t,s \le D = o(n^{1/2})$ and $m = \frac{n}{k} = \Omega(n)$, for the same constant $K_1 = K_1(H) > 0$ as in Proposition~\ref{prop:p-alpha-bound}, the following holds for the probability that the pair of compatible local configurations $\alpha \in I_t$ and $\beta \in I_s$ is realized
    \begin{align*}
        &\EE[X_{\alpha} X_{\beta}]\\
        &= \left(\prod_{e \in E_1(H)}\frac{(m-2c_{e}(\alpha \cup \beta)-1)!!}{(m-1)!!}\right) \left(\prod_{e \in E_2(H)} \frac{(m - c_{e}(\alpha \cup \beta))!}{m!}\right)\left( \frac{\binom{\frac{dn}{2} - c(\alpha \cup \beta)}{r - c^{\del}(\alpha \cup \beta)}}{\binom{\frac{dn}{2}}{r}} \frac{(2r-2\rho(\alpha \cup \beta)-1)!!}{(2r-1)!!} \right)\\
        &\le \left(\frac{K_1}{n}\right)^{c(\alpha \cup \beta) + \rho(\alpha \cup \beta)}.
    \end{align*}
    If we denote $f = c^{\del}(\alpha \cup \beta)$, $h = \rho(\alpha \cup \beta)$, and $q = c^{\ret}(\alpha \cup \beta)$, then
    \begin{align*}
        c(\alpha \cup \beta) &= c^{\ret}(\alpha \cup \beta) + c^{\del}(\alpha \cup \beta) = f + q,
    \end{align*}
    and,
    \begin{align}
        \EE[X_{\alpha} X_{\beta}] \le \left(\frac{K_1}{n}\right)^{f+h+q}, \label{ineq:p-bound-b2}
    \end{align}
    regardless of whether $f = c^{\del}(\alpha \cup \beta) \le r$.

    Next, for $t,s \le D = o(n^{1/2})$ and a tuple $(f, h, q)$ such that $\max\{s,t\} \le h+q \le t+s$ and $h \le f \le 2h$ \footnote{Note that the restrictions $ \max\{s,t\}\le h+q\le t+s$ and $h \le f \le 2h$ are enforced because for any pair of compatible $\alpha \in I_t$ and $\beta \in I_s$, the total number of cycle edges $h+q = \rho(\alpha \cup \beta) + c^{\ret}(\alpha \cup \beta)$ is at least $\max\{s,t\}$ and at most $s+t$, and $\rho(\alpha \cup \beta) \le c^{\del}(\alpha \cup \beta) \le 2\rho(\alpha \cup \beta)$.}, the number of compatible pairs of $\alpha \in I_t$ and $\beta \in N(\alpha) \cap I_s$ such that $\alpha \ne \beta$ with $f = c^{\del}(\alpha \cup \beta)$, $h = \rho(\alpha \cup \beta)$, and $q = c^{\ret}(\alpha \cup \beta)$ can be bounded by
    \begin{align*}
        &\Bigg(\One\{h+q < t+s\}\sum_{c=1}^{s+t-(h+q)} (2t)^c \binom{s+2c-1}{2c-1} n^{h+q-c} d^{h+q}\\
        &\quad + \One\{h+q = t+s, f \le 2h - 1\} n^{h+q}d^{h+q} \Bigg) \cdot \binom{h+q}{h} \cdot \binom{2h}{2f-2h} \cdot (nd)^{2f-2h} \cdot (4h-2f-1)!!\\
        &\le \left(\One\{h+q < t+s\}\sum_{c=1}^{s+t} 2^{s+2c-1}\left(\frac{2t}{n}\right)^c + \One\left\{
        \begin{array}{c}
               h+q = t+s,\\
               f \le 2h-1
        \end{array} \right\}\right) \cdot (2d n)^{q + 2f - h} \frac{(2h)!(4h-2f-1)!!}{(2f-2h)!(4h-2f)!}\\
        &\le \left(\One\{h+q < t+s\}\frac{2^s(2t)}{n} + \One\left\{
        \begin{array}{c}
               h+q = t+s,\\
               f \le 2h-1
        \end{array} \right\}\right) \cdot (2d n)^{q + 2f - h} (2h)^{4h-2f}, \numberthis \label{ineq:count-bound-b2}
    \end{align*}
    since the first term counts the situation when the $t$ cycle edges in $\mathcal{O}^{\ret}(\alpha) \sqcup \mathcal{R}(\alpha)$ and the $s$ cycle edges in $\mathcal{O}^{\ret}(\beta) \sqcup \mathcal{R}(\beta)$ share at least one edge, in which case
    \begin{itemize}
        \item $\sum_{c=1}^{s+t-(h+q)} (2t)^c \binom{s+2c-1}{2c-1} n^{h+q-c} d^{h+q}$ bounds the number of choices of the $h+q = \rho(\alpha \cup \beta) + c^{\ret}(\alpha \cup \beta)$ cycle vertices of $\alpha$ and $\beta$ and the color choices among $E(H)$ for the cycle edges of type $\originalretained$. 
        
        To see this, note that $s+t-(h+q)$ is the total number of shared edges between the cycle edges $\mathcal{O}^{\ret}(\alpha) \sqcup \mathcal{R}(\alpha)$ of $\alpha$ and the cycle edges $\mathcal{O}^{\ret}(\beta) \sqcup \mathcal{R}(\beta)$ of $\beta$. Since $\alpha \ne \beta$ are compatible, the collection of edges $\mathcal{O}^{\ret}(\alpha) \sqcup \mathcal{R}(\alpha)$ is different from $\mathcal{O}^{\ret}(\beta) \sqcup \mathcal{R}(\beta)$, as otherwise compatibility would also force $\mathcal{O}^{\del}(\alpha) = \mathcal{O}^{\del}(\beta)$, contradicting $\alpha \ne \beta$. Thus, the shared cycle edges $(\mathcal{O}^{\ret}(\alpha) \sqcup \mathcal{R}(\alpha)) \cap (\mathcal{O}^{\ret}(\beta) \sqcup \mathcal{R}(\beta))$ between $\alpha$ and $\beta$ form a common subgraph of a length-$t$ cycle and a different length-$s$ cycle, which is a union of disjoint paths. Let $c$ be the number of connected components of the disjoint union of paths of $(\mathcal{O}^{\ret}(\alpha) \sqcup \mathcal{R}(\alpha)) \cap (\mathcal{O}^{\ret}(\beta) \sqcup \mathcal{R}(\beta))$, which satisfies $1 \le c \le s+t-(h+q)$. Then, the $t$ cycle edges of $\beta$ are partitioned into $2c$ consecutive segments, where $c$ segments consisting of non-shared edges of $\beta$ interleave with $c$ segments consisting of shared edges between $\alpha$ and $\beta$. We first bound the number of choices for this decomposition of the cycle edges of $\beta$ into $2c$ segments, which is at most $\binom{s+2c-1}{2c-1}$ by the standard stars-and-bars argument.

        Next, we enumerate how the $c$ segments of shared edges of $\beta$, with lengths specified by the decomposition of edges of $\beta$ in the previous paragraph, are positioned in $\alpha$. To do so, we enumerate the starting vertices for each of $c$ segments and one of the two directions that each segment follow on $\alpha$, giving at most $(2t)^c$ choices.

        Finally, after the intersection pattern is enumerated, the number of vertices that the cycle edges of $\alpha$ and $\beta$ are attached to is exactly \[t + (s - \#\{\text{shared cycle edges of } \alpha \text{ and }\beta\} - c) = t + s - (s+t-(h+q)) - c = h+q-c.\]
        The number of choices for these vertices and their colors among $E(H)$ is then bounded by $n^{h+q-c}d^{h+q}$.

        Combining the choices for $1\le c \le s+t-(h+q)$ and the factors $\binom{s+2c-1}{2c-1}$, $(2t)^c$, and $n^{h+q-c}d^{2(h+q)}$, we see that $\sum_{c=1}^{s+t-(h+q)} (2t)^c \binom{s+2c-1}{2c-1} n^{h+q-c} d^{h+q}$ bounds the number of choices of the cycle vertices and the color choices among $E(H)$ for the cycle edges in $\alpha$ and $\beta$ of type $\originalretained$;
    \end{itemize}
    and the second term counts the situation when the $t$ cycle edges in $\mathcal{O}^{\ret}(\alpha) \sqcup \mathcal{R}(\alpha)$ and the $s$ cycle edges in $\mathcal{O}^{\ret}(\beta) \sqcup \mathcal{R}(\beta)$ are disjoint, in which case
    \begin{itemize}
        \item $h+q = t+s$ and $h \le f\le 2h - 1$ according to relation $5$ that if $c^{\ret}(\alpha \cup \beta) + \rho(\alpha + \beta) = t + s$, then $\rho(\alpha \cup \beta) \le c^{\del}(\alpha \cup \beta) \le 2\rho(\alpha \cup \beta)-1$. Thus, \[n^{t+s}d^{t+s} = \One\left\{
        \begin{array}{c}
               h+q = t+s,\\
               f \le 2h-1
        \end{array} \right\} n^{h+q} d^{h+q}\] bounds the number of choices for the cycle vertices choices and the color choices among $E(H)$ for the cycle edges in $\alpha$ and $\beta$ of type $\originalretained$;
    \end{itemize}
    and the rest of the common terms arise because
    \begin{itemize}
        \item $\binom{h+q}{h} = \binom{\rho(\alpha \cup \beta) + \mathcal{O}^{\ret}(\alpha \cup \beta)}{\rho(\alpha \cup \beta)}$ bounds the number of choices of the cycle edges of $\alpha$ and $\beta$ of the type $\rematched$,
        \item $\binom{2h}{2f-2h} = \binom{2\rho(\alpha \cup \beta)}{2(c^{\del}(\alpha \cup \beta) - \rho(\alpha \cup \beta))} = \binom{2\rho(\alpha \cup \beta)}{c^{\del}(\alpha \cup \beta) - c^{\del,2}(\alpha \cup \beta)}$ bounds the number of choices of the stubs of $\mathcal{R}(\alpha \cup \beta)$ that are incident to $\mathcal{O}^{\del}(\alpha \cup \beta) \setminus \mathcal{O}^{\del,2}(\alpha \cup \beta)$,
        \item $(nd)^{2f-2h} = (nd)^{c^{\del}(\alpha \cup \beta) - c^{\del,2}(\alpha \cup \beta)}$ bounds the number of choices of stubs that are paired to the $2f-2h$ stubs of $\mathcal{R}(\alpha \cup \beta)$ by $\mathcal{O}^{\del}(\alpha \cup \beta) \setminus \mathcal{O}^{\del,2}(\alpha \cup \beta)$,
        \item and finally, $(4h-2f-1)!! = (2\rho(\alpha \cup \beta) - (c^{\del}(\alpha \cup \beta) - c^{\del,2}(\alpha \cup \beta))-1)!!$ bounds the number of choices of $\mathcal{O}^{\del,2}(\alpha \cup \beta)$.
    \end{itemize}

    Combining \eqref{ineq:p-bound-b2} and \eqref{ineq:count-bound-b2}, we may bound
    \begin{align*}
        b_2 &= \sum_{\alpha \in I} \sum_{\substack{\beta \in N(\alpha): \beta \ne \alpha}} \EE[X_{\alpha}X_{\beta}]\\
        &= \sum_{t=1}^D \sum_{s=1}^D \sum_{\substack{f,h,q \ge 0:\\ \max\{s,t\} \le h+q \le t+s,\\ h \le f \le 2h}} \sum_{\substack{\alpha \in I_t,\\ \beta \in N(\alpha) \cap I_s:\\ \alpha \ne \beta,\\ f = c^{\del}(\alpha \cup \beta),\\ h = \rho(\alpha \cup \beta),\\ q = c^{\ret}(\alpha \cup \beta)} } \EE[X_{\alpha} X_{\beta}]\\
        &\le \sum_{t=1}^D \sum_{s=1}^D \sum_{\substack{f,h,q \ge 0:\\ \max\{s,t\} \le h+q \le t+s,\\ h \le f \le 2h}}  \left(\One\{h+q < t+s\}\frac{2^s(2t)}{n} + \One\left\{
        \begin{array}{c}
               h+q = t+s,\\
               f \le 2h-1
        \end{array} \right\}\right) \\
        &\quad\cdot (2d n)^{q + 2f - h} (2h)^{4h-2f} \left(\frac{K_1}{n}\right)^{f+h+q}\\
        &\le \sum_{t=1}^D \sum_{s=1}^D \sum_{\substack{f,h,q \ge 0:\\ \max\{s,t\} \le h+q < t+s,\\ h \le f \le 2h}}  \frac{1}{n} \cdot (2K_2)^{t+s}\left(\frac{h^2}{n}\right)^{2h-f} \\
        &\quad+ \sum_{t=1}^D \sum_{s=1}^D \sum_{\substack{f,h,q \ge 0:\\  h+q = t+s,\\ h \le f \le 2h-1}} (2K_2)^{t+s} \left(\frac{h^2}{n}\right)^{2h-f}
        \intertext{for the same constant $K_2 = K_2(H) > 0$ as before, since $q+2f-h, 4h-2f$, and $f+h+q$ are all at most $3(t+s)$, and $t \le 2^t$,}
        &\le \frac{1}{n} \sum_{j=2}^{2D} j (2K_2)^j \sum_{\ell = 0}^\infty j^2\left(\frac{j^2}{n}\right)^{\ell} + \sum_{j=2}^{2D} j(2K_2)^j \sum_{\ell = 1}^\infty j^2 \left(\frac{j^2}{n}\right)^{\ell}\\
        &\le \frac{2(2D)^4 (2K_2)^{2D}}{n}  + \frac{2(2D)^4 (2K_2)^{2D}}{n},
    \end{align*}
    which is $o(1)$ provided that $D \le c_2\log(n)$ for a constant $c_2 = c_2(H) > 0$.

    \subsection{Bound for $b_3$} \label{sec:b3-bound}

    Recall that
    \begin{align*}
        b_3 &= \sum_{\alpha \in I} \EE\Big|\EE[X_{\alpha} - p_{\alpha} \vert \sigma(X_{\beta}: \beta \not\in N(\alpha))]\Big|.
    \end{align*}
    We may rewrite it as
    \begin{align*}
        &\sum_{\alpha \in I} \EE\Big|\EE[X_{\alpha} - p_{\alpha} \vert \sigma(X_{\beta}: \beta \not\in N(\alpha))]\Big|\\
        &= \sum_{\alpha \in I} \sum_{z_{\beta}: \beta \not\in N(\alpha)}\bigg|\Pr(X_{\alpha} = 1, (X_{\beta}: \beta \not\in N(\alpha)) = (z_{\beta}: \beta \not\in N(\alpha)))\\
        &\quad - p_{\alpha} \Pr((X_{\beta}: \beta \not\in N(\alpha)) = (z_{\beta}: \beta \not\in N(\alpha)))\bigg|\\
        &= \sum_{\alpha \in I} p_{\alpha} \sum_{z_{\beta}: \beta \not\in N(\alpha)}\bigg|\frac{\Pr(X_{\alpha} = 1, (X_{\beta}: \beta \not\in N(\alpha)) = (z_{\beta}: \beta \not\in N(\alpha)))}{\Pr(X_{\alpha} = 1)}\\
        &\quad - \Pr((X_{\beta}: \beta \not\in N(\alpha)) = (z_{\beta}: \beta \not\in N(\alpha)))\bigg|\\
        &= \sum_{\alpha \in I} p_{\alpha} \left\|\Pr((X_{\beta}: \beta \not\in N(\alpha)) = \bullet \vert X_{\alpha} = 1 ) - \Pr((X_{\beta}: \beta \not\in N(\alpha)) = \bullet )\right\|_1\\
        &= 2\sum_{\alpha \in I} p_{\alpha} \cdot d_{\TV}\left(\sL(X_{\beta}: \beta \not\in N(\alpha) \vert X_{\alpha} = 1), \sL(X_{\beta}: \beta \not\in N(\alpha))\right). \numberthis \label{eq:b3-TV}
    \end{align*}
    Since $Z = (X_{\beta}: \beta \not\in N(\alpha))$ is a random vector that depends on the full information of the noiseless random lift before the random deletion, the deletion process, and the rematching process, we work with the probability space on $\{\text{ feasible }(\mathcal{O}^{\ret}, \mathcal{O}^{\del}, \mathcal{R})\}$, where $\mathcal{O}^{\ret}$ is the collection of edges from the original noiseless random lift that are not deleted by the random deletion process, $\mathcal{O}^{\del}$ is the collection of $r$ edges from the original noiseless random lift that are deleted by the random deletion process, and $\mathcal{R}$ is the collection of $r$ edges that are formed in the rematching process after the random deletion process. We say $(\mathcal{O}^{\ret}, \mathcal{O}^{\del}, \mathcal{R})$ is feasible if the edge configurations in $(\mathcal{O}^{\ret}, \mathcal{O}^{\del}, \mathcal{R})$ can be realized by the noisy random lift distribution $\Tilde{\sP}$. Note that the original noiseless random lift is then the graph formed by the edges of $\mathcal{O}^{\ret} \sqcup \mathcal{O}^{\del}$, and the final observed noisy random lift is the graph formed by the edges of $\mathcal{O}^{\ret} \sqcup \mathcal{R}$.
    
    For a specific $W = (\mathcal{O}^{\ret}, \mathcal{O}^{\del}, \mathcal{R})$, we use $X_{\beta}(W)$ to denote the indicator variable of the realization of the local configuration $\beta$ in $W$. Using of the coupling characterization of the total variation distance, we get
    \begin{align*}
        &d_{\TV}\left(\sL(X_{\beta}: \beta \not\in N(\alpha) \vert X_{\alpha} = 1), \sL(X_{\beta}: \beta \not\in N(\alpha))\right)\\
        &= \inf_{\Gamma \in \mathfrak{C}\left(\sL(X_{\beta}: \beta \not\in N(\alpha) \vert X_{\alpha} = 1), \sL(X_{\beta}: \beta \not\in N(\alpha))\right) } \Pr_{(\Tilde{Z}, Z) \sim \Gamma}\left(\Tilde{Z} \ne Z\right)\\
        &\le \inf_{\Pi \in \mathfrak{C}(\sL_{\alpha}, \sL) } \Pr_{(W_{\alpha}, W) \sim \Pi}\left((X_{\beta}(W_{\alpha}): \beta \not\in N(\alpha)) \ne (X_{\beta}(W): \beta \not\in N(\alpha))\right), \numberthis \label{ineq:TV-coupling}
    \end{align*}
    where $\mathfrak{C}(\sL_{\alpha}, \sL)$ is the collection of couplings between $\sL_{\alpha} \colonequals \sL(W = (\mathcal{O}^{\ret}, \mathcal{O}^{\del}, \mathcal{R}) \vert X_{\alpha}(W) = 1)$ and $\sL \colonequals \sL(W = (\mathcal{O}^{\ret}, \mathcal{O}^{\del}, \mathcal{R}))$.

    To upper bound $b_3$, we will 
    \begin{itemize}
        \item for every $\alpha \in I$ construct a specific coupling $\Pi$ between $W_{\alpha} \sim \sL_{\alpha}$ and $W \sim \sL$, and
        \item for $(W_{\alpha}, W) \sim \Pi$ drawn from this coupling, control the probability that $X_{\beta}(W_{\alpha}) \ne X_{\beta}(W)$ for every $\beta \not\in N(\alpha)$. 
    \end{itemize} 

    To do so, we will actually define a sequence of distributions $\mathcal{L} \equalscolon \mathcal{L}_0, \mathcal{L}_1, \mathcal{L}_2, \dots, \mathcal{L}_{\ell} \colonequals \mathcal{L}_{\alpha}$ for some number of steps $\ell \in \mathbb{N}$, and construct a joint coupling $\Pi \in \mathfrak{C}(\mathcal{L}_0, \mathcal{L}_1, \dots, \mathcal{L}_{\ell})$. Before we describe this joint coupling, we first define a few primitive operations to help describe the coupling construction. We call these operations informally as switching operations, which have been extensively used in the analysis of random regular graphs (for example, see \cite[Section 3]{mckay2004short} and \cite[Section 4]{cook2018size}).

    \begin{definition}[Primitive Operations]
        We will need four types of primitive operations. We will always assume below that $(\mathcal{O}^{\ret}, \mathcal{O}^{\del})$ is a pair of disjoint edge collections, such that no stub is incident to at least two edges in $\mathcal{O}^{\ret} \sqcup \mathcal{O}^{\del}$, and similarly for edge collection $\mathcal{R}$ considered below, no stub is incident to at least two edges in $\mathcal{R}$.
        \begin{enumerate}
            \item \textbf{($\mathcal{O}$ switching):} For an edge $\{(x,\vec{e}),(y,\iota(\vec{e}))\}$ of color $e \in E(H)$, the $\{(x,\vec{e}),(y,\iota(\vec{e}))\}{-}\mathcal{O}$ switching does the following to any pair $(\mathcal{O}^{\ret}, \mathcal{O}^{\del})$ of disjoint edge collections such that $\mathcal{O}^{\ret} \sqcup \mathcal{O}^{\del}$ forms the graph of a noiseless random lift.
            
            If $\{(x,\vec{e}),(y,\iota(\vec{e}))\} \in \mathcal{O}^{\ret} \sqcup \mathcal{O}^{\del}$, output $(\mathcal{O}^{\ret}, \mathcal{O}^{\del})$. Otherwise, let $\{(x,\vec{e}),(a,\iota(\vec{e}))\}$, \newline$\{(b,\vec{e}),(y,\iota(\vec{e}))\} \in \mathcal{O}^{\ret} \sqcup \mathcal{O}^{\del}$ be the two edges incident to stubs $(x,\vec{e})$ and $(y,\iota(\vec{e}))$.
                \begin{itemize}
                    \item If both $\{(x,\vec{e}),(a,\iota(\vec{e}))\}$ and $\{(b,\vec{e}),(y,\iota(\vec{e}))\}$ belong to $\mathcal{O}^{\ret}$, set \[\Tilde{\mathcal{O}}^{\ret} \colonequals \left(\mathcal{O}^{\ret} \cup \{\{(x,\vec{e}),(y,\iota(\vec{e}))\}, \{(b,\vec{e}),(a,\iota(\vec{e}))\}\}\right) \setminus \{\{(x,\vec{e}),(a,\iota(\vec{e}))\}, \{(b,\vec{e}),(y,\iota(\vec{e}))\}\}\]
                    and output $(\Tilde{\mathcal{O}}^{\ret}, \mathcal{O}^{\del})$.
                    \item If both $\{(x,\vec{e}),(a,\iota(\vec{e}))\}$ and $\{(b,\vec{e}),(y,\iota(\vec{e}))\}$ belong to $\mathcal{O}^{\del}$, set \[\Tilde{\mathcal{O}}^{\del} \colonequals \left(\mathcal{O}^{\del} \cup \{\{(x,\vec{e}),(y,\iota(\vec{e}))\}, \{(b,\vec{e}),(a,\iota(\vec{e}))\}\}\right) \setminus \{\{(x,\vec{e}),(a,\iota(\vec{e}))\}, \{(b,\vec{e}),(y,\iota(\vec{e}))\}\}\]
                    and output $(\mathcal{O}^{\ret}, \Tilde{\mathcal{O}}^{\del})$.
                    \item If $\{(x,\vec{e}),(a,\iota(\vec{e}))\} \in \mathcal{O}^{\ret}$ and $\{(b,\vec{e}),(y,\iota(\vec{e}))\} \in \mathcal{O}^{\del}$, set
                    \begin{align*}
                        \hat{\mathcal{O}}^{\ret} &\colonequals \mathcal{O}^{\ret} \setminus \{\{(x,\vec{e}),(a,\iota(\vec{e}))\}\},\\
                        \hat{\mathcal{O}}^{\del} &\colonequals \mathcal{O}^{\del} \setminus \{\{(b,\vec{e}),(y,\iota(\vec{e}))\}\},\\
                        (\Tilde{\mathcal{O}}^{\ret}, \Tilde{\mathcal{O}}^{\del}) &\colonequals \begin{cases}
                            (\hat{\mathcal{O}}^{\ret} \cup \{\{(x,\vec{e}),(y,\iota(\vec{e}))\}\}, \hat{\mathcal{O}}^{\del} \cup \{\{(b,\vec{e}),(a,\iota(\vec{e}))\}\}) & \quad \text{w.p.~} 1/2,\\
                            (\hat{\mathcal{O}}^{\ret} \cup \{\{(b,\vec{e}),(a,\iota(\vec{e}))\}\}, \hat{\mathcal{O}}^{\del} \cup \{\{(x,\vec{e}),(y,\iota(\vec{e}))\}\}) & \quad \text{w.p.~} 1/2,
                        \end{cases}
                    \end{align*}
                    and output $(\Tilde{\mathcal{O}}^{\ret}, \Tilde{\mathcal{O}}^{\del})$.
                \end{itemize}
            \item \textbf{($\ret/\del$ switching):} For an edge $\{(x,\vec{e}),(y,\iota(\vec{e}))\}$ of color $e \in E(H)$, a type $\tau \in \{\ret, \del\}$, and a collection $\sA$ of edges, the $(\{(x,\vec{e}),(y,\iota(\vec{e}))\}, \sA){-}\tau$ switching does the following to any pair $(\mathcal{O}^{\ret}, \mathcal{O}^{\del})$ of disjoint edge collections such that $\mathcal{O}^{\ret} \sqcup \mathcal{O}^{\del}$ forms the graph of a noiseless random lift, $\{(x,\vec{e}),(y,\iota(\vec{e}))\} \in (\mathcal{O}^{\ret} \sqcup \mathcal{O}^{\del}) \setminus \sA$, and $\sA \subseteq \mathcal{O}^{\ret} \sqcup \mathcal{O}^{\del}$.

            \begin{itemize}
                \item If $\{(x,\vec{e}),(y,\iota(\vec{e}))\} \in \mathcal{O}^{\ret}$ and $\tau = \ret$, or $\{(x,\vec{e}),(y,\iota(\vec{e}))\} \in \mathcal{O}^{\del}$ and $\tau = \del$, output $(\mathcal{O}^{\ret}, \mathcal{O}^{\del})$.
                \item If $\{(x,\vec{e}),(y,\iota(\vec{e}))\} \in \mathcal{O}^{\ret}$ and $\tau = \del$, 
                \begin{itemize}
                    \item if $\mathcal{O}^{\del} \setminus \sA = \emptyset$, do nothing and declare ``failure'';
                    \item otherwise, let $\{(a,\vec{e}),(b,\iota(\vec{e}))\} \sim \Unif(\mathcal{O}^{\del} \setminus \sA)$ be a uniformly random edge of color $e$ not belonging to $\sA$, set
                    \begin{align*}
                        \Tilde{\mathcal{O}}^{\ret} &\colonequals (\mathcal{O}^{\ret} \cup \{\{(a,\vec{e}),(b,\iota(\vec{e}))\}\}) \setminus \{\{(x,\vec{e}),(y,\iota(\vec{e}))\}\},\\
                        \Tilde{\mathcal{O}}^{\del} &\colonequals (\mathcal{O}^{\del} \cup \{\{(x,\vec{e}),(y,\iota(\vec{e}))\}\}) \setminus \{\{(a,\vec{e}),(b,\iota(\vec{e}))\}\},
                    \end{align*}
                    and output $(\Tilde{\mathcal{O}}^{\ret}, \Tilde{\mathcal{O}}^{\del})$.
                \end{itemize}
                \item If $\{(x,\vec{e}),(y,\iota(\vec{e}))\} \in \mathcal{O}^{\del}$ and $\tau = \ret$, 
                \begin{itemize}
                    \item if $\mathcal{O}^{\ret} \setminus \sA = \emptyset$, do nothing and declare ``failure'';
                    \item otherwise, let $\{(a,\vec{e}),(b,\iota(\vec{e}))\} \sim \Unif(\mathcal{O}^{\ret} \setminus \sA)$ be a uniformly random edge of color $e$ not belonging to $\sA$, set
                    \begin{align*}
                        \Tilde{\mathcal{O}}^{\ret} &\colonequals (\mathcal{O}^{\ret} \cup \{\{(x,\vec{e}),(y,\iota(\vec{e}))\}\}) \setminus \{\{(a,\vec{e}),(b,\iota(\vec{e}))\}\},\\
                        \Tilde{\mathcal{O}}^{\del} &\colonequals (\mathcal{O}^{\del} \cup \{\{(a,\vec{e}),(b,\iota(\vec{e}))\}\}) \setminus \{\{(x,\vec{e}),(y,\iota(\vec{e}))\}\},
                    \end{align*}
                    and output $(\Tilde{\mathcal{O}}^{\ret}, \Tilde{\mathcal{O}}^{\del})$.
                \end{itemize}
            \end{itemize}
            \item \textbf{(Support switching):} For a pair $(S, T)$ of disjoint stub collections of equal size $|S| = |T|$, the $(S,T){-}\support$ switching does the following to a collection $\mathcal{R}$ of edges such that every stub of $S$ is incident to $\mathcal{R}$, and no stub of $T$ is incident to $\mathcal{R}$. 
            
            Draw a uniformly random bijection $f: S \to T$, and extend $f$ to outside of $S$ with the identity map. Let
            \begin{align*}
                \mathcal{D} \colonequals \{\{(x,\vec{e}_1),(y,\vec{e}_2)\} \in \mathcal{R}: (x,\vec{e}_1) \in S \text{ or } (y,\vec{e}_2) \in S\}
            \end{align*}
            be the collection of edges in $\mathcal{R}$ incident to $S$. We denote
            \begin{align*}
                f(\mathcal{D}) \colonequals \{\{f(x,\vec{e}_1),f(y,\vec{e}_2)\}: \{(x,\vec{e}_1),(y,\vec{e}_2)\} \in \mathcal{D}\}
            \end{align*}
            to be edge collection transformed by $f$. Then, set
            \begin{align*}
                \Tilde{\mathcal{R}} \colonequals \left(\mathcal{R} \cup f(\mathcal{D})\right) \setminus \mathcal{D}
            \end{align*}
            and output $\Tilde{\mathcal{R}}$.
            
            \item \textbf{($\mathcal{R}$ switching):} For an edge $\{(x,\vec{e}_1),(y,\vec{e}_2)\}$, the $\{(x,\vec{e}_1),(y,\vec{e}_2)\}{-}\mathcal{R}$ switching does the following to any edge collection $\mathcal{R}$ such that both stubs $(x,\vec{e}_1)$ and $(y,\vec{e}_2)$ are incident to $\mathcal{R}$.

            If $\{(x,\vec{e}_1),(y,\vec{e}_2)\} \in \mathcal{R}$, output $\mathcal{R}$. Otherwise, let $\{(x,\vec{e}_1),(a,\vec{e}_a)\}, \{(b,\vec{e}_b),(y,\vec{e}_2)\} \in \mathcal{R}$ be the two edges incident to stubs $(x,\vec{e}_1)$ and $(y,\vec{e}_2)$. Set \[\Tilde{\mathcal{R}} \colonequals \left(\mathcal{R} \cup \{\{(x,\vec{e}_1),(y,\vec{e}_2)\}, \{(b,\vec{e}_b),(a,\vec{e}_a)\}\}\right) \setminus \{\{(x,\vec{e}_1),(a,\vec{e}_a)\}, \{(b,\vec{e}_b),(y,\vec{e}_2)\}\}\]
            and output $\Tilde{\mathcal{R}}$.
        \end{enumerate}
    \end{definition}

    Now we are ready to describe the coupling construction. 

    \paragraph{Coupling between $\sL$ and $\sL_{\alpha}$:}

    Recall that a cycle local configuration $\alpha \in I$ consists of three collections of edges $\mathcal{O}^{\ret}(\alpha), \mathcal{O}^{\del}(\alpha)$, and $\mathcal{R}(\alpha)$ such that
    \begin{itemize}
        \item the edges in $\mathcal{O}^{\ret}(\alpha) \sqcup \mathcal{R}(\alpha)$ form a cycle,
        \item and every stub incident to $\mathcal{O}^{\del}(\alpha)$ is also incident to $\mathcal{R}$.
    \end{itemize}

    First, we list the edges 
    \begin{align*}
        \mathcal{O}^{\ret}(\alpha)\sqcup \mathcal{O}^{\del}(\alpha) &\equalscolon \{\{(x_1, \vec{e}_1),(y_1, \iota(\vec{e}_1))\}, \dots, \{(x_{\ell_1}, \vec{e}_{\ell_1}),(y_{\ell_1}, \iota(\vec{e}_{\ell_1}))\}\},\\
        \mathcal{R} &\equalscolon \{\{(p_1,\vec{f}_1),(q_1,\vec{g}_1)\}, \dots, \{(p_{\ell_2}, \vec{f}_{\ell_2}), (q_{\ell_2}, \vec{g}_{\ell_2})\}\}.
    \end{align*}
    
    We construct the coupling in stages. We start with $W_0 = (\mathcal{O}^{\ret}_0, \mathcal{O}^{\del}_0, \mathcal{R}_0) \sim \sL$.
    
    \begin{enumerate}
        \item \textbf{($\mathcal{O}$ stage):} 
        For $i = 1, 2, \dots, \ell_1$, do the following:
        \begin{itemize}
            \item Let $W_{i-1} = (\mathcal{O}^{\ret}_{i-1}, \mathcal{O}^{\del}_{i-1}, \mathcal{R}_{i-1})$ be the current configuration.
            \item Apply $\{(x_i,\vec{e}_i),(y_i, \iota(\vec{e}_i))\}{-}\mathcal{O}$ switching to $(\mathcal{O}^{\ret}_{i-1}, \mathcal{O}^{\del}_{i-1})$ to obtain $(\mathcal{O}^{\ret}_i, \mathcal{O}^{\del}_i)$.
            \item Let $S \colonequals \stub(\mathcal{O}^{\del}_{i-1}) \setminus \stub(\mathcal{O}^{\del}_i))$, and $T \colonequals \stub(\mathcal{O}^{\del}_i) \setminus \stub(\mathcal{O}^{\del}_{i-1}) $. Apply $(S,T){-}\support$ switching to $\mathcal{R}_{i-1}$ to obtain $\mathcal{R}_i$.
            \item Set the new configuration as $W_i \colonequals (\mathcal{O}^{\ret}_{i}, \mathcal{O}^{\del}_{i}, \mathcal{R}_{i})$.
        \end{itemize}
            
        \item \textbf{($\ret/\del$ stage):}
        Initialize $\mathcal{A}_0 \colonequals \emptyset$. For $i = 1, 2, \dots, \ell_1$, do the following:
        \begin{itemize}
            \item Let $W_{\ell_1 + i - 1} = (\mathcal{O}^{\ret}_{\ell_1 + i - 1}, \mathcal{O}^{\del}_{\ell_1 + i - 1}, \mathcal{R}_{\ell_1 + i - 1})$ be the current configuration.
            \item If $\{(x_i,\vec{e}_i),(y_i, \iota(\vec{e}_i))\} \in \mathcal{O}^{\ret}(\alpha)$, apply $(\{(x_i,\vec{e}_i),(y_i, \iota(\vec{e}_i))\}, \mathcal{A}_{i-1}){-}\ret$ switching to \newline $(\mathcal{O}^{\ret}_{\ell_1 + i - 1}, \mathcal{O}^{\del}_{\ell_1 + i - 1})$ to obtain $(\mathcal{O}^{\ret}_{\ell_1 + i}, \mathcal{O}^{\del}_{\ell_1 + i})$. Similarly, if $\{(x_i,\vec{e}_i),(y_i, \iota(\vec{e}_i))\} \in \mathcal{O}^{\del}(\alpha)$, apply $(\{(x_i,\vec{e}_i),(y_i, \iota(\vec{e}_i))\}, \mathcal{A}_{i-1}){-}\del$ switching to $(\mathcal{O}^{\ret}_{\ell_1 + i - 1}, \mathcal{O}^{\del}_{\ell_1 + i - 1})$ to obtain $(\mathcal{O}^{\ret}_{\ell_1 + i}, \mathcal{O}^{\del}_{\ell_1 + i})$. 
            \item Let $S \colonequals \stub(\mathcal{O}^{\del}_{\ell_1 + i-1}) \setminus \stub(\mathcal{O}^{\del}_{\ell_1 + i}))$, and $T \colonequals \stub(\mathcal{O}^{\del}_{\ell_1 + i}) \setminus \stub(\mathcal{O}^{\del}_{\ell_1 + i-1}) $. Apply $(S,T){-}\support$ switching to $\mathcal{R}_{\ell_1 + i-1}$ to obtain $\mathcal{R}_{\ell_1 + i}$.
            \item Set the new configuration as $W_{\ell_1 + i} \colonequals (\mathcal{O}^{\ret}_{\ell_1 + i}, \mathcal{O}^{\del}_{\ell_1 + i}, \mathcal{R}_{\ell_1 + i})$, and \newline $\mathcal{A}_{i} \colonequals \mathcal{A}_{i-1} \cup \{\{(x_i,\vec{e}_i),(y_i, \iota(\vec{e}_i))\}\}$.
        \end{itemize}
        \item \textbf{($\mathcal{R}$ stage):} For $i = 1, 2, \dots, \ell_2$, do the following:
        \begin{itemize}
            \item Let $W_{2\ell_1 + i - 1} = (\mathcal{O}^{\ret}_{2\ell_1 + i - 1}, \mathcal{O}^{\del}_{2\ell_1 + i - 1}, \mathcal{R}_{2\ell_1 + i - 1})$ be the current configuration.
            \item Apply $\{(p_i,\vec{f}_i),(q_i,\vec{g}_i)\}{-}\mathcal{R}$ switching to $\mathcal{R}_{2\ell_1 + i - 1}$ to obtain $\mathcal{R}_{2\ell_1 + i}$. 
            \item Set $(\mathcal{O}^{\ret}_{2\ell_1 + i}, \mathcal{O}^{\del}_{2\ell_1 + i}) \colonequals (\mathcal{O}^{\ret}_{2\ell_1 + i - 1}, \mathcal{O}^{\del}_{2\ell_1 + i - 1})$.
            \item Set the new configuration as $W_{2\ell_1 + i} \colonequals (\mathcal{O}^{\ret}_{2\ell_1 + i}, \mathcal{O}^{\del}_{2\ell_1 + i}, \mathcal{R}_{2\ell_1 + i})$.
        \end{itemize}
    \end{enumerate}
    We set the final configuration as $W_{\alpha} \colonequals W_{2\ell_1 + \ell_2}$.

    First, we argue that the joint distribution of $(W_0, W_1, \dots, W_{2\ell_1 + \ell_2})$ is a joint coupling between $\mathcal{L} \equalscolon \mathcal{L}_0, \mathcal{L}_1, \mathcal{L}_2, \dots, \mathcal{L}_{2\ell_1 + \ell_2} \colonequals \mathcal{L}_{\alpha}$, where
    \begin{itemize}
        \item for $1 \le i \le \ell_1$, $\mathcal{L}_i$ is the distribution of $(\mathcal{O}^{\ret}, \mathcal{O}^{\del}, \mathcal{R}) \sim \mathcal{L}$ conditioned on \[\{\{(x_1, \vec{e}_1),(y_1, \iota(\vec{e}_1))\}, \dots, \{(x_{i}, \vec{e}_{i}),(y_{i}, \iota(\vec{e}_{i}))\}\} \subseteq \mathcal{O}^{\ret} \sqcup \mathcal{O}^{\del};\]
        \item for $1\le i \le \ell_1$, $\mathcal{L}_{\ell_1 + i}$ is the distribution of $(\mathcal{O}^{\ret}, \mathcal{O}^{\del}, \mathcal{R}) \sim \mathcal{L}$ conditioned on \[\{\{(x_1, \vec{e}_1),(y_1, \iota(\vec{e}_1))\}, \dots, \{(x_{\ell_1}, \vec{e}_{\ell_1}),(y_{\ell_1}, \iota(\vec{e}_{\ell_1}))\}\} \subseteq \mathcal{O}^{\ret} \sqcup \mathcal{O}^{\del}\] and that for every $1 \le j \le i$, $\{(x_j, \vec{e}_j),(y_j, \iota(\vec{e}_j))\} \in \mathcal{O}^{\ret}$ whenever $\{(x_j, \vec{e}_j),(y_j, \iota(\vec{e}_j))\} \in \mathcal{O}^{\ret}(\alpha)$, and similarly $\{(x_j, \vec{e}_j),(y_j, \iota(\vec{e}_j))\} \in \mathcal{O}^{\del}$ whenever $\{(x_j, \vec{e}_j),(y_j, \iota(\vec{e}_j))\} \in \mathcal{O}^{\del}(\alpha)$;
        \item for $1\le i \le \ell_2$, $\sL_{2\ell_1 + i}$ is the distribution conditioned on 
        \begin{align*}
            \mathcal{O}^{\ret}(\alpha) &\subseteq \mathcal{O}^{\ret},\\
            \mathcal{O}^{\del}(\alpha) &\subseteq \mathcal{O}^{\del},
        \end{align*}
        and that
        \begin{align*}
            \{\{(p_1,\vec{f}_1),(q_1,\vec{g}_1)\}, \dots, \{(p_{i}, \vec{f}_{i}), (q_{i}, \vec{g}_{i})\}\} \subseteq \mathcal{R}.
        \end{align*}
    \end{itemize}
    To see this, we may easily verify using standard exchangeability argument that
    \begin{itemize}
        \item during $\mathcal{O}$ stage, at every step, the $\{(x_i,\vec{e}_i),(y_i,\iota(\vec{e}_i))\}{-}\mathcal{O}$ switching forces the edge \newline $\{(x_i,\vec{e}_i),(y_i,\iota(\vec{e}_i))\}$ to be inside $\mathcal{O}^{\ret}_i \sqcup \mathcal{O}^{\del}_i$ while maintaining the rest of the non-forced edges in the noiseless random lift to be uniformly random. Moreover, the $(S,T){-}\support$ switching ensures that $\mathcal{R}_i$ is a uniformly random stub pairing on the stubs incident to $\mathcal{O}^{\del}_i$;
        \item during $\ret/\del$ stage, at every step, the $(\{(x_i,\vec{e}_i),(y_i,\iota(\vec{e}_i))\}, \mathcal{A}_{i-1}){-}\tau$ switching where $\tau \in \{\ret, \del\}$ forces the edge $\{(x_i,\vec{e}_i),(y_i,\iota(\vec{e}_i))\}$ to fall into the required collection among $\mathcal{O}^{\ret}_{\ell_1 + i}$ and $\mathcal{O}^{\del}_{\ell_1 + i}$, while maintaining that $\mathcal{O}^{\del}_{\ell_1 + i}$ is a uniformly random $r$-subset of $\mathcal{O}^{\ret}_{\ell_1 + i} \sqcup \mathcal{O}^{\del}_{\ell_1 + i}$ conditioned on that the edges in $\mathcal{A}_i = \mathcal{A}_{i-1} \cup \{\{(x_i,\vec{e}_i),(y_i,\iota(\vec{e}_i))\}\}$ fall into the required collections. Similarly, the $(S,T){-}\support$ switching ensures that $\mathcal{R}_{\ell_1 + i}$ is a uniformly random stub pairing on the stubs incident to $\mathcal{O}^{\del}_{\ell_1 + i}$;
        \item during $\mathcal{R}$ stage, at every step, $\{(p_i,\vec{f}_i),(q_i,\vec{g}_i)\}{-}\mathcal{R}$ switching  forces the edge $\{(p_i,\vec{f}_i),(q_i,\vec{g}_i)\}$ to be inside $\mathcal{R}_{2\ell_1 + i}$, while maintaining that the rest of the non-forced edges in $\mathcal{R}_{2\ell_1 + i}$ follow a uniformly random stub pairing conditioned on \[\{\{(p_1,\vec{f}_1),(q_1,\vec{g}_1)\}, \dots, \{(p_{i}, \vec{f}_{i}), (q_{i}, \vec{g}_{i})\}\} \subseteq \mathcal{R}_{2\ell_1 + i}.\]
    \end{itemize}
    Therefore, we conclude that the distribution of $(W_0, W_1, \dots, W_{2\ell_1 + \ell_2})$ is a joint coupling between $\mathcal{L} \equalscolon \mathcal{L}_0, \mathcal{L}_1, \mathcal{L}_2, \dots, \mathcal{L}_{2\ell_1 + \ell_2} \colonequals \mathcal{L}_{\alpha}$. Call this constructed joint coupling $\Lambda$.

    Using the coupling $\Lambda$, we may further upper bound \eqref{ineq:TV-coupling} as
    \begin{align*}
        &d_{\TV}\left(\sL(X_{\beta}: \beta \not\in N(\alpha) \vert X_{\alpha} = 1), \sL(X_{\beta}: \beta \not\in N(\alpha))\right)\\
        &\le \inf_{\Pi \in \mathfrak{C}(\sL_{\alpha}, \sL) } \Pr_{(W_{\alpha}, W) \sim \Pi}\left((X_{\beta}(W_{\alpha}): \beta \not\in N(\alpha)) \ne (X_{\beta}(W): \beta \not\in N(\alpha))\right)\\
        &\le \Pr_{(W_0, W_1, \dots, W_{2\ell_1 + \ell_2} = W_{\alpha}) \sim \Lambda} \left((X_{\beta}(W_0): \beta \not\in N(\alpha))\ne (X_{\beta}(W_{\alpha}): \beta \not\in N(\alpha))\right)\\
        &\le \EE_{(W_0, W_1, \dots, W_{2\ell_1 + \ell_2} = W_{\alpha}) \sim \Lambda} \left[\sum_{i=1}^{2\ell_1 + \ell_2} \sum_{\beta 
        \not\in N(\alpha)}\One\left\{X_{\beta}(W_{i-1}) \ne X_{\beta}(W_{i})\right\}\right]. \numberthis \label{ineq:TV-constructed-coupling}
    \end{align*}

    Next, for $1 \le i \le 2\ell_1 + \ell_2$, we bound
    \[\EE\left[\sum_{\beta 
        \not\in N(\alpha)}\One\left\{X_{\beta}(W_{i-1}) \ne X_{\beta}(W_{i})\right\}\right].\]
    We consider the three cases: $\sO$ stage, $\ret/\del$ stage, and $\sR$ stage.
    \begin{itemize}
        \item During $\mathcal{O}$ stage, at the $i$-th step for $1 \le i \le \ell_1$, $W_{i-1} = (\sO^{\ret}_{i-1}, \sO^{\del}_{i-1}, \sR_{i-1}) \sim \sL_{i-1}$ is drawn from the distribution of $(\sO^{\ret}, \sO^{\del}, \sR) \sim \sL$ conditioned on 
        \[\{\{(x_1,\vec{e}_1),(y_1,\iota(\vec{e}_1))\}, \dots, \{(x_{i-1},\vec{e}_{i-1}),(y_{i-1},\iota(\vec{e}_{i-1}))\}\} \subseteq \sO^{\ret} \sqcup \sO^{\del}.\]
        $W_i = (\sO^{\ret}_{i}, \sO^{\del}_{i}, \sR_{i})$ is then obtained from $W_{i-1}$ by applying $\{(x_i,\vec{e}_i),(y_i,\iota(\vec{e}_i))\}{-}\sO$ switching to $(\sO^{\ret}_{i-1}, \sO^{\del}_{i-1})$ to obtain $(\sO^{\ret}_{i}, \sO^{\del}_{i})$, and $(S,T){-}\support$ switching to $\sR_{i-1}$ to $\sR_i$, where $S \colonequals \stub(\mathcal{O}^{\del}_{i-1}) \setminus \stub(\mathcal{O}^{\del}_i))$, and $T \colonequals \stub(\mathcal{O}^{\del}_i) \setminus \stub(\mathcal{O}^{\del}_{i-1})$.

        Note that $|S| = |T| \le 1$, and moreover if they are not empty, then one of them is a singleton collection of a stub inside $\stub(\alpha)$. Then, any newly created edge in $\sR_i$ is incident to $\alpha$. Now suppose some edge $\gamma$ in $\sR_{i-1}$ is removed in $\sR_i$, then that edge is incident to $S$, in which case any realized cycle local configuration $\beta$ in $W_{i-1}$ that uses $\gamma \in \rho(\beta)$ must also use either $\{(x_i,\vec{e}_i),(a,\iota(\vec{e}_i))\}$ or $\{(b,\vec{e}_i),(y_i,\iota(\vec{e}_i))\}$ in $\sO^{\del}(\beta)$, which means that no realization of any $\beta \not\in N(\alpha)$ is destroyed in $W_i$ compared to $W_{i-1}$. The $(S,T){-}\support$ switching thus does not change the realization of any local configuration $\beta \not\in N(\alpha)$.

        On the other hand, the $\{(x_i,\vec{e}_i),(y_i,\iota(\vec{e}_i))\}{-}\mathcal{O}$ switching can at most change the realization of two edgess from $\sO^{\ret}_{i-1}$ and $\sO^{\del}_{i-1}$. If $\{(x_i,\vec{e}_i),(y_i,\iota(\vec{e}_i))\} \in \sO^{\ret}_{i-1} \sqcup \sO^{\del}_{i-1}$, the switching does nothing. Otherwise, $\{(x_i,\vec{e}_i),(a,\iota(\vec{e}_i))\}, \{(b,\vec{e}_i),(y_i,\iota(\vec{e}_i))\}$ are removed from $\sO^{\ret}_{i-1} \sqcup \sO^{\del}_{i-1}$, and $\{(x_i,\vec{e}_i),(y_i,\iota(\vec{e}_i))\}, \{(b,\vec{e}_i), (a,\iota(\vec{e}_i))\}$ are added back to form  $\sO^{\ret}_{i} \sqcup \sO^{\del}_{i}$. Call $\{(b,\vec{e}_i), (a,\iota(\vec{e}_i))\}$ the swapped edge, if it exists. Note that $\{(x_i,\vec{e}_i),(a,\iota(\vec{e}_i))\}, \{(b,\vec{e}_i),(y_i,\iota(\vec{e}_i))\}$, and $\{(x_i,\vec{e}_i),(y_i,\iota(\vec{e}_i))\}$ are all incident to $\alpha$, and thus the only way that the indicator for a local configuration $\beta \not\in N(\alpha)$ to be changed is if the swapped edge $\{(b,\vec{e}_i), (a,\iota(\vec{e}_i))\}$ exists and $\{(b,\vec{e}_i), (a,\iota(\vec{e}_i))\} \in \sO(\beta)$.

        As a result, if $X_{\beta}(W_{i-1}) \ne X_{\beta}(W_i)$ for some $\beta \not\in N(\alpha)$, then it must hold that $X_{\beta}(W_i) = 1$ and $\{(b,\vec{e}_i), (a,\iota(\vec{e}_i))\} \in \sO(\beta)$. We may then bound
        \begin{align*}
            &\EE\left[\sum_{\beta 
        \not\in N(\alpha)}\One\left\{X_{\beta}(W_{i-1}) \ne X_{\beta}(W_{i})\right\}\right]\\
        &\le \sum_{\beta 
        \not\in N(\alpha)} \Pr(X_{\beta}(W_i) = 1, \{(b,\vec{e}_i), (a,\iota(\vec{e}_i))\} \text{ exists}, \text{and } \{(b,\vec{e}_i), (a,\iota(\vec{e}_i))\} \in \sO(\beta))\\
        &\le \sum_{\beta 
        \not\in N(\alpha)} \Pr(X_{\beta}(W_i) = 1) \Pr (\{(b,\vec{e}_i), (a,\iota(\vec{e}_i))\} \in \sO(\beta) \vert X_{\beta}(W_i) = 1, \{(b,\vec{e}_i), (a,\iota(\vec{e}_i))\} \text{ exists})\\
        &\le \frac{K_3 D}{n}\sum_{\beta 
        \not\in N(\alpha)} \Pr(X_{\beta}(W_i) = 1),
        \end{align*}
        for some constant $K_3 = K_3(H) > 0$, where in the last step we use that conditioned on the realization of $W_i$ and that the swapped edge $\{(a,e_i),(b,e_i)\}$ exists, the random $\{(a,e_i),(b,e_i)\}$ is distributed uniformly among $\sC_{e} \setminus \{\{(x_1,e_1), (y_1,e_1)\}, \dots, \{(x_i,e_i), (y_i,e_i)\}\}$ where $\sC_{e}$ is the edge collection of color $e$, which has size at least $\frac{n}{k} - D = \Omega(n)$, and $|\sO(\beta)| \le 2D$ for $\beta \in I \setminus N(\alpha)$.
        
        Note that the same bound in Proposition~\ref{prop:p-sum-bound} as before holds even for $W_i \sim \sL_i$, since $\sL_i$ only conditions on $i \le 2D$ edges of $\sO(\alpha)$ to be inside $\sO^{\ret}_i \sqcup \sO^{\del}_i$, which does not asymptotically affect the bound for any $\Pr(X_{\beta}(W_i) = 1)$, as there are still $\Omega(n)$ available edges not fixed by this conditioning in any color class $\sC_{e}$. As a result,
        \begin{align*}
            &\EE\left[\sum_{\beta 
        \not\in N(\alpha)}\One\left\{X_{\beta}(W_{i-1}) \ne X_{\beta}(W_{i})\right\}\right]\\
        &\le \frac{K_3 D}{n}\sum_{\beta 
        \not\in N(\alpha)} \Pr(X_{\beta}(W_i) = 1)\\
        &\le \frac{K_3 D}{n} \cdot 2(D+1) \sum_{t=1}^D K_2^t. \numberthis \label{ineq:O-stage-bound}
        \end{align*}
        \item During $\ret/\del$ stage, at the $i$-th step for $1\le i \le \ell_1$, $W_{\ell_1 + i-1} = (\sO^{\ret}_{\ell_1 + i-1}, \sO^{\del}_{\ell_1 + i-1}, \sR_{\ell_1 + i-1}) \sim \sL_{\ell_1 + i-1}$ is drawn from the distribution of $(\sO^{\ret}, \sO^{\del}, \sR) \sim \sL$ conditioned on 
        \[\{\{(x_1, \vec{e}_1),(y_1, \iota(\vec{e}_1))\}, \dots, \{(x_{i}, \vec{e}_{i}),(y_{i}, \iota(\vec{e}_{i}))\}\} \subseteq \mathcal{O}^{\ret} \sqcup \mathcal{O}^{\del}\] and that for every $1 \le j \le i-1$, $\{(x_j, \vec{e}_j),(y_j, \iota(\vec{e}_j))\} \in \mathcal{O}^{\ret}$ whenever $\{(x_j, \vec{e}_j),(y_j, \iota(\vec{e}_j))\} \in \mathcal{O}^{\ret}(\alpha)$, and similarly $\{(x_j, \vec{e}_j),(y_j, \iota(\vec{e}_j))\} \in \mathcal{O}^{\del}$ whenever $\{(x_j, \vec{e}_j),(y_j, \iota(\vec{e}_j))\} \in \mathcal{O}^{\del}(\alpha)$.
        
        $W_{\ell_1 + i} = (\sO^{\ret}_{\ell_1 + i}, \sO^{\del}_{\ell_1 + i}, \sR_{\ell_1 + i})$ is then obtained from $W_{\ell_1 + i-1}$ by applying \newline $(\{(x_i, \vec{e}_i),(y_i, \iota(\vec{e}_i))\}, \mathcal{A}_{i-1}){-}\tau$ switching to $(\sO^{\ret}_{\ell_1 + i-1}, \sO^{\del}_{\ell_1 + i-1})$ to obtain $(\sO^{\ret}_{\ell_1 + i}, \sO^{\del}_{\ell_1 + i})$, and \newline $(S,T){-}\support$ switching to $\sR_{\ell_1 + i-1}$ to $\sR_{\ell_1 + i}$, where $S \colonequals \stub(\mathcal{O}^{\del}_{\ell_1 + i-1}) \setminus \stub(\mathcal{O}^{\del}_{\ell_1 + i}))$, and $T \colonequals \stub(\mathcal{O}^{\del}_{\ell_1 + i}) \setminus \stub(\mathcal{O}^{\del}_{\ell_1 + i-1})$.

        Note that $|S| = |T| \le 2$, and either $S$ or $T$ is a subset of $\stub(\alpha)$. Similarly as before, any newly created edge in $\sR_{\ell_1 + i}$ is incident to $\alpha$. Now suppose some edge $\gamma$ in $\sR_{\ell_1 + i-1}$ is removed in $\sR_{\ell_1 + i}$, then that edge is incident to $S$, in which case any realized cycle local configuration $\beta$ in $W_{i-1}$ that uses $\gamma \in \rho(\beta)$ must also use an edge in $\sO^{\del}(\beta)$ that is incident to $\{(x_i, \vec{e}_i),(y_i, \iota(\vec{e}_i))\} \subseteq \stub(\alpha)$, which means that no realization of any $\beta \not\in N(\alpha)$ is destroyed in $W_{\ell_1 + i}$ compared to $W_{\ell_1 + i-1}$. The $(S,T){-}\support$ switching thus does not change the realization of any local configuration $\beta \not\in N(\alpha)$. 

        On the other hand, $(\{(x_i, \vec{e}_i),(y_i, \iota(\vec{e}_i))\}, \mathcal{A}_{i-1}){-}\tau$ switching at most switches the $\ret/\del$ collection of $\{(x_i, \vec{e}_i),(y_i, \iota(\vec{e}_i))\}$ with another edge $\{(a,\vec{f}), (b,\iota(\vec{f}))\}$. Call $\{(a,\vec{f}), (b,\iota(\vec{f}))\}$ the swapped edge, if it exists. Note that $\{(x_i, \vec{e}_i),(y_i, \iota(\vec{e}_i))\} \in \sO(\alpha)$, and thus the only way that the indicator for a local configuration $\beta 
        \not\in N(\alpha)$ to be changed is if the swapped edge $\{(a,\vec{f}), (b,\iota(\vec{f}))\}$ exists, and the swapping flipped the realization of $\beta$ in $W_{\ell_1 + i - 1}$ and $W_{\ell_1 + i}$. There are two cases to consider:
        \begin{itemize}
            \item If $\{(x_i, \vec{e}_i),(y_i, \iota(\vec{e}_i))\} \in \sO^{\ret}(\alpha)$, then $\{(a,\vec{f}), (b,\iota(\vec{f}))\} \in \sO^{\ret}_{\ell_1 + i - 1}$, $\{(x_i, \vec{e}_i),(y_i, \iota(\vec{e}_i))\} \in \sO^{\del}_{\ell_1 + i - 1}$, $\{(a,\vec{f}), (b,\iota(\vec{f}))\} \in \sO^{\del}_{\ell_1 + i}$, and $\{(x_i, \vec{e}_i),(y_i, \iota(\vec{e}_i))\} \in \sO^{\ret}_{\ell_1 + i}$;
            \item If $\{(x_i, \vec{e}_i),(y_i, \iota(\vec{e}_i))\} \in \sO^{\del}(\alpha)$, then $\{(a,\vec{f}), (b,\iota(\vec{f}))\} \in \sO^{\del}_{\ell_1 + i - 1}$, $\{(x_i, \vec{e}_i),(y_i, \iota(\vec{e}_i))\} \in \sO^{\ret}_{\ell_1 + i - 1}$, $\{(a,\vec{f}), (b,\iota(\vec{f}))\} \in \sO^{\ret}_{\ell_1 + i}$, and $\{(x_i, \vec{e}_i),(y_i, \iota(\vec{e}_i))\} \in \sO^{\del}_{\ell_1 + i}$.
        \end{itemize}
        Thus, if we use $\gamma \colonequals \{(a,\vec{f}), (b,\iota(\vec{f}))\}$ to denote the swapped edge, we have
        \begin{align*}
            &\One\left\{X_{\beta}(W_{\ell_1 + i-1}) \ne X_{\beta}(W_{\ell_1 + i})\right\} \\
            &= \One\left\{X_{\beta}(W_{\ell_1 + i-1}) \ne X_{\beta}(W_{\ell_1 + i}), \gamma \text{ exists}\right\}\\
            &\le \begin{cases}
                \One\{X_{\beta}(W_{\ell_1 + i - 1}) =1, \gamma \in \sO^{\ret}(\beta)\} + \One\{X_{\beta}(W_{\ell_1 + i}) =1, \gamma \in \sO^{\del}(\beta)\}\\  \hspace{3.5in} \text{ if } \{(x_i, \vec{e}_i),(y_i, \iota(\vec{e}_i))\} \in \sO^{\ret}(\alpha),\\
                \One\{X_{\beta}(W_{\ell_1 + i - 1}) =1, \gamma \in \sO^{\del}(\beta)\} + \One\{X_{\beta}(W_{\ell_1 + i}) =1, \gamma \in \sO^{\ret}(\beta)\} \\
                \hspace{3.5in} \text{ if } \{(x_i, \vec{e}_i),(y_i, \iota(\vec{e}_i))\} \in \sO^{\del}(\alpha).
            \end{cases}
        \end{align*}
        We deal with the two cases symmetrically.

        First, let us suppose $\{(x_i, \vec{e}_i),(y_i, \iota(\vec{e}_i))\} \in \sO^{\ret}(\alpha)$. Then,
        \begin{align*}
            &\EE[\One\left\{X_{\beta}(W_{\ell_1 + i-1}) \ne X_{\beta}(W_{\ell_1 + i})\right\}]\\
            &\le \EE\left[\One\{X_{\beta}(W_{\ell_1 + i - 1}) =1, \gamma \in \sO^{\ret}(\beta)\} + \One\{X_{\beta}(W_{\ell_1 + i}) =1, \gamma \in \sO^{\del}(\beta)\}\right]\\
            &\le \Pr\left(X_{\beta}(W_{\ell_1 + i - 1}) =1\right) \Pr\left(\gamma \in \sO^{\ret}(\beta) \vert X_{\beta}(W_{\ell_1 + i - 1}) =1, \gamma \text{ exists}\right)\\
            &\quad + \Pr\left(X_{\beta}(W_{\ell_1 + i}) =1\right) \Pr\left(\gamma \in \sO^{\del}(\beta) \vert X_{\beta}(W_{\ell_1 + i}) =1, \gamma \text{ exists}\right)\\
            &\le \Pr\left(X_{\beta}(W_{\ell_1 + i - 1}) =1\right) \frac{c^{\ret}(\beta)}{\frac{nd}{2} - r - c^{\ret, i-1}}  + \Pr\left(X_{\beta}(W_{\ell_1 + i}) =1\right) \frac{c^{\del}(\beta)}{r - c^{\del, i}},
        \end{align*}
        where in the last step, we used that conditioning on $W_{\ell_1 + i}$ and that the swapped edge $\gamma$ exists, then $\gamma$ is distributed uniformly at random among $\sO^{\del}_{\ell_1 + i}$ that has not been conditioned on in $\sL_{\ell_1 + i}$ in $W_{\ell_1 + i}$, and similarly conditioning on $W_{\ell_1 + i - 1}$ and that the swapped edge $\gamma$ exists, then $\gamma$ is distributed uniformly at random among $\sO^{\ret}_{\ell_1 + i - 1}$ that has not been conditioned on in $\sL_{\ell_1 + i-1}$ in $W_{\ell_1 + i-1}$. The quantity $c^{\ret, i}$ ($c^{\del,i}$ resp.) denotes the number of conditioned edges among $\{\{(x_1, \vec{e}_1),(y_1, \iota(\vec{e}_1))\}, \dots, \{(x_i, \vec{e}_i),(y_i, \iota(\vec{e}_i))\}\}$ that belong to $\sO^{\ret}(\alpha)$ ($\sO^{\del}(\alpha)$ resp.), and the denominator is thus the size of the collection of edges that are not fixed by the conditioning, from which $\gamma$ is drawn.

        Analogous to Fact~\ref{fact:config-prob}, since $\sL_{\ell_1 + i}$ only conditions on $\ell_1 \le 2D$ edges of $\sO(\alpha)$ to be inside $\sO^{\ret}_{\ell_1 + i}$, $\sO^{\del}_{\ell_1 + i}$, or $\sO^{\ret}_{\ell_1 + i} \sqcup \sO^{\del}_{\ell_1 + i}$, we have
        \begin{align*}
            &\Pr\left(X_{\beta}(W_{\ell_1 + i}) =1\right)\\
            &= \left(\prod_{e \in E_1(H)}\frac{(m-O(D)-2c_{e}(\beta)-1)!!}{(m-O(D)-1)!!}\right) \left(\prod_{e \in E_2(H)} \frac{(m-O(D) - c_{e}(\beta))!}{(m-O(D))!}\right)\\
            &\quad \cdot \left( \frac{\binom{\frac{dn}{2} -i - c(\beta)}{r - c^{\del,i} - c^{\del}(\beta)}}{\binom{\frac{dn}{2} - i}{r - c^{\del,i}}} \frac{(2r -2\rho(\beta)-1)!!}{(2r-1)!!} \right),
            \intertext{where the third factor is the probability that the $c^{\del}(\beta) = |\sO^{\del}(\beta)|$ edges are deleted and the $c^{\ret}(\beta) = |\sO^{\ret}(\beta)|$ edges are retained when $i = c^{\del,i} + c^{\ret, i}$ edges are already fixed in $\sO^{\ret}_{\ell_1 + i}$ and $\sO^{\del}_{\ell_1 + i}$, and $r - c^{\del,i}$ number of edges still need to be deleted uniformly at random from the remaining $\frac{dn}{2} - i$ edges, and that the $\rho(\beta) = |\sR(\beta)|$ edges are formed in $\sR_{\ell_1 + i}$. Since $c(\beta) =|\sO(\beta)| \le 2D$, and $\frac{n}{k} - O(D) = \Omega(n)$, we have for the same constant $K = K(H) > 0$ as in Proposition~\ref{prop:p-alpha-bound} that}
            &\le \left(\frac{K}{n}\right)^{c(\beta)} \frac{\frac{(r-c^{\del,i})!}{(r-c^{\del,i}-c^{\del}(\beta))!} \frac{\left(\frac{dn}{2} -(i - c^{\del,i}) - r\right)!}{\left(\frac{dn}{2} -(i - c^{\del,i}) - r - (c(\beta) - c^{\del}(\beta)))\right)}}{\frac{\left(\frac{dn}{2} - i\right)!}{\left(\frac{dn}{2} -i - c(\beta)\right)!} } \frac{(2r-2\rho(\beta)-1)!!}{(2r-1)!!}\\
            &\le \left(\frac{2K}{n}\right)^{c(\beta)} \frac{ \left(\frac{dn}{2} -c^{\ret,i} - r\right)^{c(\beta) - c^{\del}(\beta)}}{\left(\frac{dn}{2}\right)^{c(\beta)}} \frac{(r-c^{\del,i})!}{(r-c^{\del,i}-c^{\del}(\beta))!}\frac{(2r-2\rho(\beta)-1)!!}{(2r-1)!!}\\
            &\le \left(\frac{2K}{n}\right)^{c(\beta)} \frac{ \left(\frac{dn}{2} -c^{\ret,i} - r\right)^{c(\beta) - c^{\del}(\beta)} (r  - c^{\del,i})^{c^{\del}(\beta) - \rho(\beta)}}{\left(\frac{dn}{2}\right)^{c(\beta)}}, 
        \end{align*}
        and thus
        \begin{align*}
            &\Pr\left(X_{\beta}(W_{\ell_1 + i}) =1\right) \frac{c^{\del}(\beta)}{r - c^{\del, i}}\\
            &\le c^{\del}(\beta)\left(\frac{2K}{n}\right)^{c(\beta)} \frac{ \left(\frac{dn}{2} -c^{\ret,i} - r\right)^{c(\beta) - c^{\del}(\beta)} (r  - c^{\del,i})^{c^{\del}(\beta) - \rho(\beta)-1}}{\left(\frac{dn}{2}\right)^{c(\beta)}}\\
            &\le \begin{cases}
                (2D)\left(\frac{2K}{n}\right)^{c(\beta)} \frac{1}{\left(\frac{dn}{2}\right)^{\rho(\beta)+1}} & \quad \text{ if } c^{\del}(\beta) \ge \rho(\beta) + 1,\\
                (2D)\left(\frac{2K}{n}\right)^{c(\beta)} \frac{1}{\left(\frac{dn}{2}\right)^{\rho(\beta)}} & \quad \text{ if } c^{\del}(\beta) = \rho(\beta) \ge 1,\\
                0 & \quad \text{ if } c^{\del}(\beta) = 0.
            \end{cases} \\
            &\le \begin{cases}
                D\left(\frac{K_1}{n}\right)^{c(\beta) + \rho(\beta)+1} & \quad \text{ if } c^{\del}(\beta) \ge \rho(\beta) + 1,\\
                D\left(\frac{K_1}{n}\right)^{c(\beta) + \rho(\beta)} & \quad \text{ if } c^{\del}(\beta) = \rho(\beta) \ge 1,\\
                0 & \quad \text{ if } c^{\del}(\beta) = 0
            \end{cases} \numberthis \label{ineq:p-bound-del}
        \end{align*}

        for the same constant $K_1 = K_1(H) > 0$ in Proposition~\ref{prop:p-alpha-bound}. Similarly, we have
        \begin{align*}
            &\Pr\left(X_{\beta}(W_{\ell_1 + i-1}) =1\right) \frac{c^{\ret}(\beta)}{\frac{nd}{2} - r - c^{\ret, i}}\\
            &\le \left(\frac{2K}{n}\right)^{c(\beta)} \frac{ \left(\frac{dn}{2} -c^{\ret,i} - r\right)^{c(\beta) - c^{\del}(\beta)} (r  - c^{\del,i})^{c^{\del}(\beta) - \rho(\beta)}}{\left(\frac{dn}{2}\right)^{c(\beta)}} \frac{c^{\ret}(\beta)}{\frac{nd}{2} - r - c^{\ret, i}}\\
            &= c^{\ret}(\beta)\left(\frac{2K}{n}\right)^{c(\beta)} \frac{ \left(\frac{dn}{2} -c^{\ret,i} - r\right)^{c(\beta) - c^{\del}(\beta) - 1} (r  - c^{\del,i})^{c^{\del}(\beta) - \rho(\beta)}}{\left(\frac{dn}{2}\right)^{c(\beta)}}\\
            &\le \begin{cases}
                D\left(\frac{K_1}{n}\right)^{c(\beta) + \rho(\beta) + 1}  & \quad \text{ if } c(\beta) \ge c^{\del}(\beta) + 1,\\
                0 & \quad \text{ if } c(\beta) = c^{\del}(\beta),
            \end{cases}  \numberthis \label{ineq:p-bound-ret}
        \end{align*}
        where the second case uses that $c^{\ret}(\beta) = c(\beta) - c^{\del}(\beta)$.

        For $\beta \in I_t$, if we denote $f = c^{\del}(\beta)$ and $h = \rho(\beta)$, then \[c(\beta) + \rho(\beta) = c^{\ret}(\beta) + c^{\del}(\beta)  + \rho(\beta) = t + f.\] By \eqref{ineq:p-bound-del}, we may then reuse the bound \eqref{ineq:count-bound} and perform the same analysis as in Section~\ref{sec:b1-bound} to get
        \begin{align*}
            &\sum_{\beta \not\in N(\alpha)} \Pr\left(X_{\beta}(W_{\ell_1 + i}) =1\right) \frac{c^{\del}(\beta)}{r - c^{\del, i}}\\
            &\le \sum_{t=1}^D \left(\sum_{\substack{0 \le h \le t,\\ h < f \le 2h}} \sum_{\substack{\beta \in I_t \setminus N(\alpha):\\ f = c^{\del}(\beta),\\
        h = \rho(\beta)}} D\left(\frac{K_1}{n}\right)^{t+f+1} + \sum_{\substack{1 \le h \le t,\\ f=h}} \sum_{\substack{\beta \in I_t \setminus N(\alpha):\\ f = c^{\del}(\beta),\\
        h = \rho(\beta)}} D\left(\frac{K_1}{n}\right)^{t+f}\right)\\
        &\le D\sum_{t=1}^D \left(\sum_{\substack{0 \le h \le t,\\ h < f \le 2h}} (2d^2 n)^{t + 2f-2h} (2h)^{4h-2f}\left(\frac{K_1}{n}\right)^{t+f+1} + \sum_{\substack{1 \le h \le t,\\ h = f}} (2d^2 n)^{t + 2f-2h} (2h)^{4h-2f}\left(\frac{K_1}{n}\right)^{t+f}\right)\\
        &\le D\sum_{t=1}^D K_2^t\left(\sum_{\substack{0 \le h \le t,\\ h < f \le 2h}} \left(\frac{h^2}{n}\right)^{2h-f} + \sum_{1\le h \le t}\left(\frac{h^2}{n}\right)^{h}\right)\\
        &\le \frac{2D^2(D+1)}{n} \sum_{t=1}^D K_2^t, \numberthis \label{ineq:p-sum-bound-del}
        \end{align*}
        and similarly by \eqref{ineq:p-bound-ret},
        \begin{align*}
            &\sum_{\beta \not\in N(\alpha)} \Pr\left(X_{\beta}(W_{\ell_1 + i - 1}) =1\right) \frac{c^{\ret}(\beta)}{\frac{nd}{2} - r - c^{\ret, i-1}}\\
            &\le \frac{2D^2(D+1)}{n} \sum_{t=1}^D K_2^t, \numberthis \label{ineq:p-sum-bound-ret}.
        \end{align*}
        Combining \eqref{ineq:p-sum-bound-del} and \eqref{ineq:p-sum-bound-ret}, we get
        \begin{align*}
            &\EE\left[ \sum_{\beta \not\in N(\alpha)} \One\left\{X_{\beta}(W_{\ell_1 + i-1}) \ne X_{\beta}(W_{\ell_1 + i})\right\}\right]\\
            &\le \sum_{\beta \not\in N(\alpha)} \left(\Pr\left(X_{\beta}(W_{\ell_1 + i - 1}) =1\right) \frac{c^{\ret}(\beta)}{\frac{nd}{2} - r - c^{\ret, i-1}}  + \Pr\left(X_{\beta}(W_{\ell_1 + i}) =1\right) \frac{c^{\del}(\beta)}{r - c^{\del, i}}\right)\\
            &\le \frac{4D^2(D+1)}{n}\sum_{t=1}^D K_2^t. \numberthis \label{ineq:retdel-stage-bound}
        \end{align*}
        By a symmetric argument, the same bound holds for the case of $\{(x_i,e_i), (y_i,e_i)\} \in \sO^{\del}(\alpha)$.

        \item During $\sR$ stage, at the $i$-th step for $1 \le i \le \ell_2$, $W_{2\ell_1 + i - 1} = (\sO^{\ret}_{2\ell_1 + i - 1}, \sO^{\del}_{2\ell_1 + i - 1}, \sR_{2\ell_1 + i - 1}) \sim \sL_{2\ell_1 + i - 1}$ is drawn from the distribution of $(\sO^{\ret}, \sO^{\del}, \sR) \sim \sL$ conditioned on
        \begin{align*}
            \mathcal{O}^{\ret}(\alpha) &\subseteq \mathcal{O}^{\ret},\\
            \mathcal{O}^{\del}(\alpha) &\subseteq \mathcal{O}^{\del},
        \end{align*}
        and that
        \begin{align*}
            \{\{(p_1,\vec{f}_1), (q_1,\vec{g}_1)\}, \dots, \{(p_{i-1},\vec{f}_{i-1}), (q_{i-1},\vec{g}_{i-1})\}\} \subseteq \mathcal{R}.
        \end{align*}

        $W_{2\ell_1 + i} = (\sO^{\ret}_{2\ell_1 + i}, \sO^{\del}_{2\ell_1 + i}, \sR_{2\ell_1 + i})$ is then obtained from $W_{2\ell_1 + i - 1}$ by applying \newline $\{(p_i,\vec{f}_i), (q_i,\vec{g}_i)\}{-}\sR$ switching to $\sR_{2\ell_1 + i-1}$ to obtain $\sR_{2\ell_1 + i}$.

        If the $\{(p_i,\vec{f}_i), (q_i,\vec{g}_i)\}{-}\sR$ switching does nothing, then $W_{2\ell_1 + i} = W_{2\ell_1 + i-1}$. If the \newline $\{(p_i,\vec{f}_i), (q_i,\vec{g}_i)\}{-}\sR$ switching does take place, then $\{(p_i,\vec{f}_i), (a,\vec{e}_a)\}, \{(b,\vec{e}_b), (q_i,\vec{g}_i)\}$ are removed from $\sR_{2\ell_1 + i - 1}$, and $\{(p_i,\vec{f}_i), (q_i,\vec{g}_i)\}, \{(b,\vec{e}_b), (a,\vec{e}_a)\}$ are added back to form $\sR_{2\ell_1 + i}$. Call $\{(b,\vec{e}_b), (a,\vec{e}_a)\}$ the swapped edge, if it exists. Note that $\{(p_i,\vec{f}_i), (a,\vec{e}_a)\}, \{(b,\vec{e}_b), (q_i,\vec{g}_i)\}$, and $\{(p_i,\vec{f}_i), (q_i,\vec{g}_i)\}$ are all incident to $\alpha$, and thus the only way that the indicator for a local configuration $\beta \not\in N(\alpha)$ to be changed is if the swapped edge $\{(b,\vec{e}_b), (a,\vec{e}_a)\}$ exists and $\{(b,\vec{e}_b), (a,\vec{e}_a)\} \in \sR(\beta)$.

        As a result, if $X_{\beta}(W_{2\ell_1 + i-1}) \ne X_{\beta}(W_{2\ell_1 + i})$ for some $\beta \not\in N(\alpha)$, then it must hold that $X_{\beta}(W_{2\ell_1 + i}) = 1$ and $\{(b,\vec{e}_b), (a,\vec{e}_a)\} \in \sR(\beta)$. We may then bound
        \begin{align*}
            &\EE\left[\sum_{\beta 
        \not\in N(\alpha)}\One\left\{X_{\beta}(W_{2\ell_1 + i-1}) \ne X_{\beta}(W_{2\ell_1 + i})\right\}\right]\\
        &\le \sum_{\beta 
        \not\in N(\alpha)} \Pr(X_{\beta}(W_{2\ell_1 + i}) = 1, \{(b,\vec{e}_b), (a,\vec{e}_a)\} \text{ exists}, \text{and } \{(b,\vec{e}_b), (a,\vec{e}_a)\} \in \sR(\beta))\\
        &\le \sum_{\beta 
        \not\in N(\alpha)} \Pr(X_{\beta}(W_{2\ell_1 + i}) = 1) \Pr (\{(b,\vec{e}_b), (a,\vec{e}_a)\} \in \sR(\beta) \vert X_{\beta}(W_{2\ell_1 + i}) = 1, \{(b,\vec{e}_b), (a,\vec{e}_a)\} \text{ exists})\\
        &\le \sum_{\beta 
        \not\in N(\alpha)} \Pr(X_{\beta}(W_{2\ell_1 + i}) = 1) \frac{\rho(\beta)}{r - i},
        \end{align*}
        where in the last step we use that conditioned on the realization of $W_{2\ell_1 + i}$ and that the swapped edge $\{(b,\vec{e}_b), (a,\vec{e}_a)\}$ exists, the random $\{(b,\vec{e}_b), (a,\vec{e}_a)\}$ is distributed uniformly among the unconditioned rematched edges \[\sR_{2\ell_1 + i} \setminus \{\{(p_1,\vec{f}_1), (q_1,\vec{g}_1)\}, \dots, \{(p_{i},\vec{f}_{i}), (q_{i},\vec{g}_{i})\}\}.\]

        Analogous to Fact~\ref{fact:config-prob}, since $\sL_{2\ell_1 + i}$ only conditions on $\ell_1 \le 2D$ edges of $\sO^{\ret}(\alpha)$ and $\sO^{\del}(\alpha)$ to be inside $\sO^{\ret}_{2\ell_1 + i}$ and $\sO^{\del}_{2\ell_1 + i}$ respectively, and $i$ edges of $\sR(\alpha)$ to be inside $\sR_{2\ell_1 + i}$, we have
        \begin{align*}
            &\Pr(X_{\beta}(W_{2\ell_1 + i}) = 1)\\
            &= \left(\prod_{e \in E_1(H)}\frac{(m-O(D)-2c_{e}(\beta)-1)!!}{(m-O(D)-1)!!}\right) \left(\prod_{e \in E_2(H)} \frac{(m-O(D) - c_{e}(\beta))!}{(m-O(D))!}\right)\\
            &\quad \cdot \left( \frac{\binom{\frac{dn}{2} -c(\alpha) - c(\beta)}{r - c^{\del}(\alpha) - c^{\del}(\beta)}}{\binom{\frac{dn}{2} - c(\alpha)}{r - c^{\del}(\alpha)}} \frac{(2r -2i -2\rho(\beta)-1)!!}{(2r-2i-1)!!} \right)
            \intertext{where the third factor is the probability that the $c^{\del}(\beta) = |\sO^{\del}(\beta)|$ edges are deleted and the $c^{\ret}(\beta) = |\sO^{\ret}(\beta)|$ edges are retained when $c(\alpha) = |\sO(\alpha)|$ edges are already fixed in $\sO^{\ret}_{2\ell_1 + i}$ and $\sO^{\del}_{2\ell_1 + i}$, and $r - c^{\del}(\alpha)$ number of edges still need to be deleted uniformly at random from the remaining $\frac{dn}{2} - c(\alpha)$ edges, and that the $\rho(\beta) = |\sR(\beta)|$ edges are formed in $\sR_{2\ell_1 + i}$ conditioned on $i$ edges $\{(p_1,\vec{f}_1), (q_1,\vec{g}_1)\}, \dots, \{(p_{i},\vec{f}_{i}), (q_{i},\vec{g}_{i})\} \in \sR_{2\ell_1 + i}$. Since $c(\beta) =|\sO(\beta)| \le 2D$, and $\frac{n}{k} - O(D) = \Omega(n)$, we have for the same constant $K = K(H) > 0$ as in Proposition~\ref{prop:p-alpha-bound} that}
            &\le \left(\frac{K}{n}\right)^{c(\beta)} \frac{\frac{(r-c^{\del}(\alpha))!}{(r-c^{\del}(\alpha)-c^{\del}(\beta))!} \frac{\left(\frac{dn}{2} -(c(\alpha) - c^{\del}(\alpha)) - r\right)!}{\left(\frac{dn}{2} -(c(\alpha) - c^{\del}(\alpha)) - r - (c(\beta) - c^{\del}(\beta)))\right)}}{\frac{\left(\frac{dn}{2} - c(\alpha)\right)!}{\left(\frac{dn}{2} -c(\alpha) - c(\beta)\right)!} } \frac{(2r-2i-2\rho(\beta)-1)!!}{(2r-2i-1)!!}\\
            &\le \left(\frac{2K}{n}\right)^{c(\beta)} \frac{ \left(\frac{dn}{2} -c^{\ret}(\alpha) - r\right)^{c(\beta) - c^{\del}(\beta)}}{\left(\frac{dn}{2}\right)^{c(\beta)}} \frac{(r-c^{\del}(\alpha))!}{(r-c^{\del}(\alpha)-c^{\del}(\beta))!}\frac{(2r-2i-2\rho(\beta)-1)!!}{(2r-2i-1)!!}\\
            &\le c^{\rho(\beta)}\left(\frac{2K}{n}\right)^{c(\beta)} \frac{ \left(\frac{dn}{2} -c^{\ret}(\alpha) - r\right)^{c(\beta) - c^{\del}(\beta)}}{\left(\frac{dn}{2}\right)^{c(\beta)}} \frac{(r  - c^{\del}(\alpha))^{c^{\del}(\beta)}}{(r-i)^{\rho(\beta)}}.  
        \end{align*}
        Thus, for $\beta \in I_t$, if we denote $f = c^{\del}(\beta)$ and $h = \rho(\beta)$, we have
        \begin{align*}
            c(\beta) &= t+f-h,\\
            i \le h &\le f \le 2h,\\
        \end{align*}
        and for the same constant $K_1 = K_1(H) > 0$ as in Proposition~\ref{prop:p-alpha-bound}, we get
        \begin{align*}
            &\Pr(X_{\beta}(W_{2\ell_1 + i}) = 1) \frac{\rho(\beta)}{r - i}\\
            &\le c^{h}\left(\frac{2K}{n}\right)^{t+f-h} \frac{ \left(\frac{dn}{2} -c^{\ret}(\alpha) - r\right)^{t-h}}{\left(\frac{dn}{2}\right)^{t+f-h}} \frac{(r  - c^{\del}(\alpha))^{f}}{(r-i)^{h}}\frac{h}{r - i}\\
            &\le c^D h\left(\frac{2K}{n}\right)^{t+f-h} \frac{ \left(\frac{dn}{2} -c^{\ret}(\alpha) - r\right)^{t-h} (r-i)^{f-h-1}}{\left(\frac{dn}{2}\right)^{t+f-h}}\\
            &\le \begin{cases}
                Dc^D \left(\frac{K_1}{n}\right)^{t+f - 1} & \quad \text{ if } f \ge h + 1,\\
                Dc^D \left(\frac{K_1}{n}\right)^{t+f} & \quad \text{ if } f= h \ge 1,\\
                0 & \quad \text{ if } h = 0.
            \end{cases} \numberthis\label{ineq:p-bound-r}
        \end{align*}
        By \eqref{ineq:p-bound-r}, similarly as before we may then reuse the bound \eqref{ineq:count-bound} and get,
        \begin{align*}
            &\sum_{\beta \not\in N(\alpha)} \Pr(X_{\beta}(W_{2\ell_1 + i}) = 1) \frac{\rho(\beta)}{r - i}\\
            &\le \sum_{t=1}^D \left(\sum_{\substack{0 \le h \le t,\\ h < f \le 2h}} \sum_{\substack{\beta \in I_t \setminus N(\alpha):\\ f = c^{\del}(\beta),\\
        h = \rho(\beta)}} Dc^D\left(\frac{K_1}{n}\right)^{t+f+1} + \sum_{\substack{1 \le h \le t,\\ f=h}} \sum_{\substack{\beta \in I_t \setminus N(\alpha):\\ f = c^{\del}(\beta),\\
        h = \rho(\beta)}} Dc^D\left(\frac{K_1}{n}\right)^{t+f}\right)\\
        &\le Dc^D\sum_{t=1}^D \left(\sum_{\substack{0 \le h \le t,\\ h < f \le 2h}} (2d^2 n)^{t + 2f-2h} (2h)^{4h-2f}\left(\frac{K_1}{n}\right)^{t+f+1} + \sum_{\substack{1 \le h \le t,\\ h = f}} (2d^2 n)^{t + 2f-2h} (2h)^{4h-2f}\left(\frac{K_1}{n}\right)^{t+f}\right)\\
        &\le Dc^D\sum_{t=1}^D K_2^t\left(\sum_{\substack{0 \le h \le t,\\ h < f \le 2h}} \left(\frac{h^2}{n}\right)^{2h-f} + \sum_{1\le h \le t}\left(\frac{h^2}{n}\right)^{h}\right)\\
        &\le \frac{2D^2(D+1)c^D}{n} \sum_{t=1}^D K_2^t.
        \end{align*}
        As a result,
        \begin{align*}
            &\EE\left[\sum_{\beta 
        \not\in N(\alpha)}\One\left\{X_{\beta}(W_{2\ell_1 + i-1}) \ne X_{\beta}(W_{2\ell_1 + i})\right\}\right]\\
        &\le \sum_{\beta 
        \not\in N(\alpha)} \Pr(X_{\beta}(W_{2\ell_1 + i}) = 1) \frac{\rho(\beta)}{r - i}\\
        &\le \frac{2D^2(D+1)c^D}{n} \sum_{t=1}^D K_2^t. \numberthis \label{ineq:R-stage-bound}
        \end{align*}

    \end{itemize}

    \paragraph{Putting together the Bound for $b_3$:}

    Using \eqref{ineq:TV-constructed-coupling}, \eqref{ineq:O-stage-bound}, \eqref{ineq:retdel-stage-bound}, and \eqref{ineq:R-stage-bound}, and that $\ell_1 = c(\alpha), \ell_2 = \rho(\alpha) \le 2D$, we get for every $\alpha \in I$,
    \begin{align*}
        &d_{\TV}\left(\sL(X_{\beta}: \beta \not\in N(\alpha) \vert X_{\alpha} = 1), \sL(X_{\beta}: \beta \not\in N(\alpha))\right)\\
        &\le \EE_{(W_0, W_1, \dots, W_{2\ell_1 + \ell_2} = W_{\alpha}) \sim \Lambda} \left[\sum_{i=1}^{2\ell_1 + \ell_2} \sum_{\beta 
        \not\in N(\alpha)}\One\left\{X_{\beta}(W_{i-1}) \ne X_{\beta}(W_{i})\right\}\right]\\
        &\le \ell_1 \frac{2K_6 D(D+1)}{n} \sum_{t=1}^D K_2^t + \ell_1 \frac{4D^2(D+1)}{n}\sum_{t=1}^D K_2^t + \ell_2 \frac{2D^2(D+1)c^D}{n} \sum_{t=1}^D K_2^t\\
        &\le \frac{8K_6D^3(D+1)}{n} \sum_{t=1}^D K_2^t.
    \end{align*}
    
    Plugging this bound back to \eqref{eq:b3-TV} and using Proposition~\ref{prop:p-sum-bound}, we get
    \begin{align*}
        b_3 &\le 2\sum_{\alpha \in I} p_{\alpha} \cdot d_{\TV}\left(\sL(X_{\beta}: \beta \not\in N(\alpha) \vert X_{\alpha} = 1), \sL(X_{\beta}: \beta \not\in N(\alpha))\right)\\
        &\le \frac{16K_6D^3(D+1)}{n} \sum_{t=1}^D K_2^t \sum_{\alpha \in I} p_{\alpha}\\
        &\le \frac{32K_6D^3(D+1)^2}{n}\left(\sum_{t=1}^D K_2^t\right)^2,
    \end{align*}
    which is $o(1)$ provided that $D \le c_3\log(n)$ for a constant $c_3 = c_3(H) > 0$.

    \subsection{Mean Computation} \label{sec:mean-computation}

    Finally, let us compute the expected number of cycles $\mu_t = \EE[Z_t] = \sum_{\alpha \in I_t} p_{\alpha}$.

    Consider a local configuration $\alpha \in I_t$ of length-$t$ cycle. List its $t$ cycle edges with types in cyclic order as
    \begin{align*}
        \cyc(\alpha) \colonequals ((\{(v_1, \vec{f}_1), (v_2,\vec{g}_1)\}, \tau_1), (\{(v_2, \vec{f}_2), (v_3,\vec{g}_2)\}, \tau_2), \dots, (\{(v_t, \vec{f}_t), (v_1,\vec{g}_t)\},\tau_t)),
    \end{align*}
    where $\tau_i \in \{\originalretained, \rematched\}$ denotes the type of cycle edge $\{(v_i, \vec{f}_i), (v_{i+1},\vec{g}_i)\}$. Note that $\sO^{\ret}(\alpha) \sqcup \sR(\alpha) = \{\{(v_1, \vec{f}_1), (v_2,\vec{g}_1)\}, \{(v_2, \vec{f}_2), (v_3,\vec{g}_2)\}, \dots, \{(v_t, \vec{f}_t), (v_1,\vec{g}_t)\}\}$. Let
    \begin{align*}
        \sW_t \colonequals \{w: \exists \alpha \in I_t, \cyc(\alpha) = w\}
    \end{align*}
    to be the collection of all possible typed edge sequences of local configurations of length-$t$ cycles.

    Note that combining Proposition~\ref{prop:p-alpha-bound} and \eqref{ineq:count-bound}, the analysis in Section~\ref{sec:b1-bound} shows that for $\alpha \in I_t$ with $c^{\del}(\alpha) < 2\rho(\alpha)$, we have
    \begin{align*}
        \sum_{\substack{0 \le h \le t,\\ h \le f < 2h}} \sum_{\substack{\alpha \in I_t:\\ f = c^{\del}(\alpha),\\
        h = \rho(\alpha)}} p_{\alpha} &\le \sum_{\substack{0 \le h \le t,\\ h \le f < 2h}} (2d^2 n)^{t + 2f-2h} (2h)^{4h-2f}\left(\frac{K_1}{n}\right)^{t+f}\\
        &\le K_2^t\sum_{\substack{0 \le h \le t,\\ h \le f < 2h}} \left(\frac{h^2}{n}\right)^{2h-f}\\
        &\le \frac{2D^2K_2^D}{n}, \numberthis \label{ineq:small-f-bound}
    \end{align*}
    which will turn up to be negligible for $D$ at most a small constant times $\log(n)$. We will then focus on the situation of $c^{\del}(\alpha) = 2\rho(\alpha)$.
    
    Now consider a fixed \[w = ((\{(v_1, \vec{f}_1), (v_2,\vec{g}_1)\}, \tau_1), (\{(v_2, \vec{f}_2), (v_3,\vec{g}_2)\}, \tau_2), \dots, (\{(v_t, \vec{f}_t), (v_1,\vec{g}_t)\},\tau_t)) \in \sW_t,\] and we will sum up the probabilities $p_{\alpha}$ for all $\alpha \in I_t$ such that $\cyc(\alpha) = w$, and $c^{\del}(\alpha) = 2\rho(\alpha)$. Note that whenever $\tau_i = \originalretained$, the two stubs $(v_i, \vec{f}_i)$ and $(v_{i+1}, \vec{g}_i)$ satisfy the relation $\vec{g}_i = \iota(\vec{f}_i) \in \vec{E}(H)$. We will denote $\vec{e}_{\ell}(w) \colonequals \vec{f}_{\ell}$. By Fact~\ref{fact:config-prob}, we have
    \begin{align*}
        &\sum_{\substack{\alpha \in I_t:\\ \cyc(\alpha) = w,\\ c^{\del}(\alpha) = 2\rho(\alpha) }} p_{\alpha}\\
        &= \sum_{\substack{\alpha \in I_t:\\ \cyc(\alpha) = w,\\c^{\del}(\alpha) = 2\rho(\alpha) }}\left(\prod_{e \in E_1(H)}\frac{(m-2c_{e}(\alpha)-1)!!}{(m-1)!!}\right) \left(\prod_{e \in E_2(H)} \frac{(m - c_{e}(\alpha))!}{m!}\right)\left( \frac{\binom{\frac{dn}{2} - c(\alpha)}{r - c^{\del}(\alpha)}}{\binom{\frac{dn}{2}}{r}} \frac{(2r-2\rho(\alpha)-1)!!}{(2r-1)!!} \right)
        \intertext{Since $m = \frac{n}{k} = \Omega(n)$ and $c_{e}(\alpha) \le |\sO(\alpha)| \le 2D = o(n^{1/2})$, we get}
        &= \sum_{\substack{\alpha \in I_t:\\ \cyc(\alpha) = w,\\c^{\del}(\alpha) = 2\rho(\alpha) }} \left(1 + O\left(\frac{D^2}{n}\right)\right)\left(\prod_{e \in E(H)}\left(\frac{1}{m}\right)^{c_{e}(\alpha)}\right)\left( \frac{\binom{\frac{dn}{2} - c(\alpha)}{r - c^{\del}(\alpha)}}{\binom{\frac{dn}{2}}{r}} \frac{(2r-2\rho(\alpha)-1)!!}{(2r-1)!!} \right)\\
        &= \left(1 + O\left(\frac{D^2}{n}\right)\right)\sum_{\substack{\alpha \in I_t:\\ \cyc(\alpha) = w,\\c^{\del}(\alpha) = 2\rho(\alpha) }} \left(\frac{1}{m}\right)^{c^{\ret}(\alpha)}\left(\frac{1}{m}\right)^{c^{\del}(\alpha)} \frac{\binom{\frac{dn}{2} - c(\alpha)}{r - c^{\del}(\alpha)}}{\binom{\frac{dn}{2}}{r}} \frac{(2r-2\rho(\alpha)-1)!!}{(2r-1)!!}, \numberthis \label{ineq:w-prob}
    \end{align*}
    where we use $c^{\del}_{e}(\alpha) \colonequals |\sO^{\del}(\alpha) \cap \sC_{e}|$ to denote the edges of $\alpha$ of type $\originaldeleted$ that belong to the color class $\sC_{e}$. Note that since $c^{\del}(\alpha) = 2\rho(\alpha)$, each of the $2\rho(\alpha)$ stubs of $\sR(\alpha)$ must be incident to a distinct edge in $\sO^{\del}(\alpha)$. Since $c^{\del}(\alpha) \le D = o(n^{1/2})$, we have
    \begin{align*}
        &\sum_{\substack{\alpha \in I_t:\\ \cyc(\alpha) = w,\\c^{\del}(\alpha) = 2\rho(\alpha) }} \left(\frac{1}{m}\right)^{c^{\del}(\alpha)}\\
        &= \left(\prod_{e \in E(H)}\left(\frac{1}{m}\right)^{c_{e}^{\del}(\alpha)}\right) (m - O(D))^{c^{\del}(\alpha)}\\
        &= 1 + O\left(\frac{D^2}{n}\right), \numberthis \label{ineq:del-contribution}
    \end{align*}
    where we used that each of $c^{\del}(\alpha)$ stubs of $\sR(\alpha)$ has $m - O(D)$ choices to form an edge in $\sO^{\del}(\alpha)$.

    On the other hand, for $c^{\del}(\alpha) = f$ and $\rho(\alpha) = h$ such that $f = 2h$, consider the factor
    \begin{align*}
        \frac{\binom{\frac{dn}{2} - c(\alpha)}{r - c^{\del}(\alpha)}}{\binom{\frac{dn}{2}}{r}} \frac{(2r-2\rho(\alpha)-1)!!}{(2r-1)!!} &= \frac{\frac{r!}{(r-2h)!} \frac{\left(\frac{dn}{2} - r\right)!}{\left(\frac{dn}{2} -r - (t-h)\right)!}}{\frac{\left(\frac{dn}{2}\right)!}{\left(\frac{dn}{2} - (t+h)\right)!}} \frac{(2r-2h-1)!!}{(2r-1)!!}.
    \end{align*}
    We will show that
    \begin{align}
        \left|\frac{\frac{r!}{(r-2h)!} \frac{\left(\frac{dn}{2} - r\right)!}{\left(\frac{dn}{2} -r - (t-h)\right)!}}{\frac{\left(\frac{dn}{2}\right)!}{\left(\frac{dn}{2} - (t+h)\right)!}} \frac{(2r-2h-1)!!}{(2r-1)!!} - \left(1 - \frac{2r}{dn}\right)^{t-h} \left(\frac{2r}{(dn)^2}\right)^h \right| \le O\left(\frac{t^2}{n^{h+1}}\right). \label{ineq:matching-prob-bound}
    \end{align}
    We prove \eqref{ineq:matching-prob-bound} by a case discussion.
    
    \begin{itemize}
        \item If $r \le 5t^2$, then
    \begin{align*}
        \frac{\frac{r!}{(r-2h)!} \frac{\left(\frac{dn}{2} - r\right)!}{\left(\frac{dn}{2} -r - (t-h)\right)!}}{\frac{\left(\frac{dn}{2}\right)!}{\left(\frac{dn}{2} - (t+h)\right)!}} \frac{(2r-2h-1)!!}{(2r-1)!!} &\le \frac{\frac{r!}{(r-2h)!} \frac{\left(\frac{dn}{2} - r\right)!}{\left(\frac{dn}{2} -r - (t-h)\right)!}}{\frac{\left(\frac{dn}{2}\right)!}{\left(\frac{dn}{2} - (t+h)\right)!}} \frac{(r-h)!}{r!}\\
        &\le \left(\frac{2r}{\left(\frac{dn}{2}\right)^2}\right)^h\\
        &\le O\left(\frac{t^2}{\left(\frac{dn}{2}\right)^{h+1}}\right),
    \end{align*}
    and
    \begin{align*}
        \left(1 - \frac{2r}{dn}\right)^{t-h} \left(\frac{2r}{(dn)^2}\right)^h \le O\left(\frac{t^2}{\left(\frac{dn}{2}\right)^{h+1}}\right).
    \end{align*}
    Thus, \eqref{ineq:matching-prob-bound} holds when $r \le 5t^2$.

    \item If $\frac{dn}{2} - r \le 5t^2$ and $t \ge h + 1$, then
    \begin{align*}
        \frac{\frac{r!}{(r-2h)!} \frac{\left(\frac{dn}{2} - r\right)!}{\left(\frac{dn}{2} -r - (t-h)\right)!}}{\frac{\left(\frac{dn}{2}\right)!}{\left(\frac{dn}{2} - (t+h)\right)!}} \frac{(2r-2h-1)!!}{(2r-1)!!} &\le \frac{\frac{r!}{(r-2h)!} \frac{\left(\frac{dn}{2} - r\right)!}{\left(\frac{dn}{2} -r - (t-h)\right)!}}{\frac{\left(\frac{dn}{2}\right)!}{\left(\frac{dn}{2} - (t+h)\right)!}} \frac{(r-h)!}{r!}\\
        &\le \frac{r^h (t^2)^{t-h}}{\left(\frac{dn}{2} - O(D)\right)^{t+h}}\\
        &\le O\left(\frac{t^2}{\left(\frac{dn}{2}\right)^{h+1}}\right),
    \end{align*}
    and
    \begin{align*}
        \left(1 - \frac{2r}{dn}\right)^{t-h} \left(\frac{2r}{(dn)^2}\right)^h \le O\left(\frac{t^2}{\left(\frac{dn}{2}\right)^{h+1}}\right).
    \end{align*}
    Thus, \eqref{ineq:matching-prob-bound} holds when $\frac{dn}{2} -r\le 5t^2$ and $t \ge h+1$.

    \item Finally, if $r > 5t^2$, and $\frac{dn}{2} - r > 5t^2$ or $t = h$, we have
    \begin{align*}
        \frac{r!}{(r-2h)!} &= r^{2h}\left(1 + O\left(\frac{h^2}{r}\right)\right),\\
        \frac{\left(\frac{dn}{2} - r\right)!}{\left(\frac{dn}{2} -r - (t-h)\right)!} &= \left(\frac{dn}{2} - r\right)^{t-h} \left(1 + O\left(\frac{(t-h)^2}{\frac{dn}{2} - r}\right)\right),\\
        \frac{\left(\frac{dn}{2}\right)!}{\left(\frac{dn}{2} - (t+h)\right)!} &= \left(\frac{dn}{2}\right)^{t+h} \left(1 + O\left(\frac{(t+h)^2}{\frac{dn}{2}}\right)\right),\\
        \frac{(2r-2h-1)!!}{(2r-1)!!} &= \left(\frac{1}{2r}\right)^{h}\left(1 + O\left(\frac{h^2}{r}\right)\right).
    \end{align*}

    Then,
    \begin{align*}
        &\frac{\frac{r!}{(r-2h)!} \frac{\left(\frac{dn}{2} - r\right)!}{\left(\frac{dn}{2} -r - (t-h)\right)!}}{\frac{\left(\frac{dn}{2}\right)!}{\left(\frac{dn}{2} - (t+h)\right)!}} \frac{(2r-2h-1)!!}{(2r-1)!!}\\
        &= \left(1 - \frac{2r}{dn}\right)^{t-h}\frac{r^{2h}}{\left(\frac{dn}{2}\right)^{2h}(2r)^h} \left(1 + O\left(\frac{h^2}{r} + \frac{(t-h)^2}{\frac{dn}{2}-r} + \frac{(t+h)^2}{\frac{dn}{2}} + \frac{h^2}{r}\right)\right)\\
        &= \left(1 - \frac{2r}{dn}\right)^{t-h}\left(\frac{2r}{(dn)^2}\right)^h + O\left(h^2\left(\frac{2r}{(dn)^2}\right)^{h-1} \frac{2}{(dn)^2}\right)\\
        &\quad + O\left((t-h)^2 \left(1 - \frac{2r}{dn}\right)^{t-h-1}\left(\frac{2r}{(dn)^2}\right)^h \frac{2}{dn}\right)+ O\left((t+h)^2 \left(\frac{2r}{(dn)^2}\right)^h \frac{2}{dn}\right)\\
        &= \left(1 - \frac{2r}{dn}\right)^{t-h}\left(\frac{2r}{(dn)^2}\right)^h + O\left(\frac{t^2}{n^{h+1}}\right).
    \end{align*}
    Thus, \eqref{ineq:matching-prob-bound} holds when $r > 5t^2$, and $\frac{dn}{2} - r > 5t^2$ or $t = h$.

    \end{itemize}
    We have now established \eqref{ineq:matching-prob-bound} in all cases. For the fixed 
    \[w = ((\{(v_1, \vec{f}_1), (v_2,\vec{g}_1)\}, \tau_1), (\{(v_2, \vec{f}_2), (v_3,\vec{g}_2)\}, \tau_2), \dots, (\{(v_t, \vec{f}_t), (v_1,\vec{g}_t)\},\tau_t)) \in \sW_t,\]
    let $h = \rho(w)$ be the number of edges of type $\rematched$ among the $t$ edges of $w$, and $c^{\ret}(w)$ be the number of edges ot type $\originalretained$ among the $t$ edges of $w$. Note that $c^{\ret}(w) = t - \rho(w) = t - h$. Then, any $\alpha\in I_t$ such that $\cyc(\alpha) = w$ satisfies $\rho(\alpha) = h$, and $c^{\ret}(\alpha) = t-h$. Plugging \eqref{ineq:del-contribution} and \eqref{ineq:matching-prob-bound} into \eqref{ineq:w-prob}, we get
    \begin{align*}
        &\sum_{\substack{\alpha \in I_t:\\ \cyc(\alpha) = w,\\ c^{\del}(\alpha) = 2\rho(\alpha) }} p_{\alpha}\\
        &= \sum_{\substack{\alpha \in I_t:\\ \cyc(\alpha) = w,\\c^{\del}(\alpha) = 2\rho(\alpha) }} \left(\frac{1}{m}\right)^{t-h}\left(\frac{1}{m}\right)^{c^{\del}(\alpha)} \frac{\binom{\frac{dn}{2} - c(\alpha)}{r - c^{\del}(\alpha)}}{\binom{\frac{dn}{2}}{r}} \frac{(2r-2\rho(\alpha)-1)!!}{(2r-1)!!}\\
        &= \left(1 + O\left(\frac{D^2}{n}\right)\right)\left(\frac{1}{m}\right)^{t-h} \sum_{\substack{\alpha \in I_t:\\ \cyc(\alpha) = w,\\c^{\del}(\alpha) = 2h }}\left(\frac{1}{m}\right)^{c^{\del}(\alpha)}  \left(\left(1 - \frac{2r}{dn}\right)^{t-h}\left(\frac{2r}{(dn)^2}\right)^h + O\left(\frac{t^2}{n^{h+1}}\right)\right)
        \intertext{using \eqref{ineq:matching-prob-bound},}
        &= \left(1 + O\left(\frac{D^2}{n}\right)\right) \left(\frac{1}{m}\right)^{t-h}\left(\left(1 - \frac{2r}{dn}\right)^{t-h}\left(\frac{2r}{(dn)^2}\right)^h + O\left(\frac{t^2}{n^{h+1}}\right)\right), \numberthis \label{ineq:w-contribution}
    \end{align*}
    where the last step uses \eqref{ineq:del-contribution}.

    Now it remains to sum over $w \in \sW_t$ to obtain $\mu_t$.
    \begin{enumerate}
        \item First, we consider collection of 
        \[w = ((\{(v_1, \vec{f}_1), (v_2,\vec{g}_1)\}, \tau_1), (\{(v_2, \vec{f}_2), (v_3,\vec{g}_2)\}, \tau_2), \dots, (\{(v_t, \vec{f}_t), (v_1,\vec{g}_t)\},\tau_t)) \in \sW_t,\] where all the types $\tau_1 = \dots = \tau_t = \originalretained$. In this case, $h = \rho(w) = 0$, and
        \[\vec{g}_{\ell} = \iota(\vec{f}_{\ell}), \text{ for all } 1 \le \ell \le t.\]
        Applying \eqref{ineq:w-contribution}, we get
        \begin{align*}
            &\sum_{\substack{w\in \sW_t:\\ \tau_{\ell} = \originalretained,\\ \forall 1\le \ell \le t} } \sum_{\substack{\alpha \in I_t:\\ \cyc(\alpha) = w,\\ c^{\del}(\alpha) = 2\rho(\alpha) }} p_{\alpha}\\
            &= \sum_{\substack{w\in \sW_t:\\ \tau_{\ell} = \originalretained, \\ \forall 1\le \ell \le t} } \left(\frac{1}{m}\right)^t\left(1 + O\left(\frac{D^2}{n}\right)\right)\left(1 - \frac{2r}{dn}\right)^{t}\\
            &=  \sum_{(\vec{e}_{\ell})_{\ell=1}^t} \sum_{\substack{w\in \sW_t:\\ \tau_{\ell} = \originalretained,\\ \vec{e}_{\ell}(w) = \vec{e}_{\ell},\\ \forall 1\le \ell \le t } }\left(\frac{1}{m}\right)^t \left(1 + O\left(\frac{D^2}{n}\right)\right)\left(1 - \frac{2r}{dn}\right)^{t}\\
            &=  \frac{1}{2t}\sum_{(\vec{e}_{\ell})_{\ell=1}^t} \left(\prod_{\substack{1\le \ell \le t} } \left(m - O(D)\right)\left(\One\{\tail(\vec{e}_{\ell+1}) = \head(\vec{e}_{\ell}), \vec{e}_{\ell+1} \ne \iota(\vec{e}_{\ell}) \}\right)\right)\\
            &\quad \cdot \left(\frac{1}{m}\right)^t \left(1 + O\left(\frac{D^2}{n}\right)\right)\left(1 - \frac{2r}{dn}\right)^{t}
            \intertext{where we enumerate the number of $w \in \sW_t$ with the required stub colors $(\vec{e}_{\ell})_{\ell=1}^t$ by first enumerating the stubs $(v_1, \vec{f}_1), (v_2, \vec{f}_2), \dots, (v_t, \vec{f}_t)$, giving $m - O(D)$ choices for each stub, and then enumerating the stubs $(v_2, \vec{g}_1), (v_3, \vec{g}_2), \dots, (v_1, \vec{g}_t)$, which gives the factor \[\prod_{\ell = 1}^t \One\{\tail(\vec{f}_{\ell+1}) = \tail(\vec{g}_{\ell}), \vec{f}_{\ell+1} \ne \vec{g}_{\ell} \}  = \prod_{\ell = 1}^t \One\{\tail(\vec{f}_{\ell+1}) = \head(\vec{f}_{\ell}), \vec{f}_{\ell+1} \ne \iota(\vec{f}_{\ell}) \}\] as the stub $(v_{\ell+1}, \vec{g}_{\ell})$ needs to satisfy both consistency constraint $\One\{\tail(\vec{f}_{\ell+1}) = \tail(\vec{g}_{\ell})\}$ and non-backtracking constraint $\vec{f}_{\ell+1} \ne \vec{g}_{\ell}$. Now recall that in Definition~\ref{def:nb-matrix}, the non-backtracking matrix is defined to have entries $B_{\vec{e}_1,\vec{e}_2} = \One\{\tail(\vec{e}_2) = \head(\vec{e}_1), \vec{e}_2 \ne \iota(\vec{e}_1)\}$. Thus, we have}
            &= \frac{1}{2t}\sum_{(\vec{e}_{\ell})_{\ell=1}^t} \left(\prod_{\substack{1\le \ell \le t} } B_{\vec{e}_{\ell},\vec{e}_{\ell+1}}\right) \cdot  \left(1 + O\left(\frac{D^2}{n}\right)\right)\left(1 - \frac{2r}{dn}\right)^{t}\\
            &= \left(1 - \frac{2r}{dn}\right)^{t}\frac{\tr(B^t)}{2t} + O\left(\frac{D^2}{n}\right) \frac{\tr(B^t)}{2t}, \numberthis \label{eq:ret-mean}
        \end{align*}
        where the factor $\frac{1}{2t}$ comes from the fact that the trace overcounts a length-$t$ $w \in \sW_t$ by $2t$ times, by choosing each of the $t$ edges as the starting edge, and choosing $1$ of the $2$ orientations of $w$. Note that the same overcounting factor holds even when $t = 1$ or $2$.
        \item Next, we consider the collection of \[w = ((\{(v_1, \vec{f}_1), (v_2,\vec{g}_1)\}, \tau_1), (\{(v_2, \vec{f}_2), (v_3,\vec{g}_2)\}, \tau_2), \dots, (\{(v_t, \vec{f}_t), (v_1,\vec{g}_t)\},\tau_t)) \in \sW_t\] with at least one type $\tau_i = \rematched$. In this case, $h = \rho(w) \ge 1$. Applying \eqref{ineq:w-contribution}, we get
        \begin{align*}
            &\sum_{\substack{w\in \sW_t:\\ h = \rho(w) \ge 1 } } \sum_{\substack{\alpha \in I_t:\\ \cyc(\alpha) = w,\\ c^{\del}(\alpha) = 2\rho(\alpha) }} p_{\alpha}\\
            &= \sum_{\substack{w\in \sW_t:\\ h = \rho(w) \ge 1 } } \left(1 + O\left(\frac{D^2}{n}\right)\right) \left(\frac{1}{m}\right)^{t-h}\left(\left(1 - \frac{2r}{dn}\right)^{t-h}\left(\frac{2r}{(dn)^2}\right)^h + O\left(\frac{t^2}{n^{h+1}}\right)\right)\\
            &=  \sum_{\substack{ (\tau_{\ell})_{\ell=1}^t:\\ h \colonequals \#\{\ell: \tau_{\ell} = \rematched\},\\ h \ge 1 } } \sum_{(\vec{e}_{\ell})_{\ell=1}^t  } \left(\frac{1}{m}\right)^{t-h}  \sum_{\substack{w\in \sW_t:\\ \tau_{\ell}(w) = \tau_{\ell}, \forall\,\ell \in [t],\\ \vec{e}_{\ell}(w) = \vec{e}_{\ell}, \forall\,\ell \in [t] } } \left(1 + O\left(\frac{D^2}{n}\right)\right)\\
            &\quad \cdot \left(\left(1 - \frac{2r}{dn}\right)^{t-h}\left(\frac{2r}{(dn)^2}\right)^h + O\left(\frac{t^2}{n^{h+1}}\right)\right)\\
            &= \frac{1}{2t}\sum_{\substack{ (\tau_{\ell})_{\ell=1}^t:\\ h \colonequals \#\{\ell: \tau_{\ell} = \rematched\},\\ h \ge 1 } } \left(1 + O\left(\frac{D^2}{n}\right)\right)\left(\frac{1}{m}\right)^{t-h}  \sum_{(\vec{e}_{\ell})_{\ell=1}^t } (m-O(D))^t \\
            &\quad \cdot \left(\prod_{\substack{1\le \ell \le t:\\ \tau_{\ell} = \originalretained } }  B_{\vec{e}_{\ell}, \vec{e}_{\ell+1}}\right)\left(\prod_{\substack{1\le \ell \le t:\\ \tau_{\ell} = \rematched } }  (d-1)\right)   \left(\left(1 - \frac{2r}{dn}\right)^{t-h}\left(\frac{2r}{(dn)^2}\right)^h + O\left(\frac{t^2}{n^{h+1}}\right)\right)
            \intertext{where we enumerate the number of $w \in \sW_t$ with the required stub colors $(\vec{e}_{\ell})_{\ell=1}^t$, by first enumerating the stubs $(v_1, \vec{f}_1), (v_2, \vec{f}_2), \dots, (v_t, \vec{f}_t)$, giving $m - O(D)$ choices for each stub, and then enumerating the stubs $(v_2, \vec{g}_1), (v_3, \vec{g}_2), \dots, (v_1, \vec{g}_t)$, which in the case when $\tau_{\ell} = \originalretained$ gives a factor of $B_{\vec{e}_{\ell}, \vec{e}_{\ell+1}}$ due to the non-backtracking constraint and that $\vec{g}_{\ell} = \iota(\vec{f}_{\ell})$, and in the case when $\tau_{\ell} = \rematched$ gives a factor of $d-1$ since the only constraint is $\vec{g}_{\ell} \ne \vec{f}_{\ell+1}$. Thus, }
            &= \frac{1}{2t}\sum_{\substack{ (\tau_{\ell})_{\ell=1}^t:\\ h \colonequals \#\{\ell: \tau_{\ell} = \rematched\},\\ h \ge 1 } }  \left(1 + O\left(\frac{D^2}{n}\right)\right)(m(d-1))^{h} \sum_{(\vec{e}_{\ell})_{\ell=1}^t } \left(\prod_{\substack{1\le \ell \le t:\\ \tau_{\ell} = \originalretained } }  B_{\vec{e}_{\ell}, \vec{e}_{\ell+1}}\right)\\
            &\quad \cdot\left(\left(1 - \frac{2r}{dn}\right)^{t-h}\left(\frac{2r}{(dn)^2}\right)^h + O\left(\frac{t^2}{n^{h+1}}\right)\right)\\
            &= \frac{1}{2t}\sum_{\substack{ (\tau_{\ell})_{\ell=1}^t:\\ h \colonequals \#\{\ell: \tau_{\ell} = \rematched\},\\ h \ge 1 } }   \left(1 + O\left(\frac{D^2}{n}\right)\right)(m(d-1))^{h} (kd)^{h}(d-1)^{t-h} \\
            &\quad \cdot \left(\left(1 - \frac{2r}{dn}\right)^{t-h}\left(\frac{2r}{(dn)^2}\right)^h + O\left(\frac{t^2}{n^{h+1}}\right)\right)
            \intertext{where we first enumerate the choices of $\vec{e}_{\ell}$ for $\tau_{\ell} = \rematched$, each of which gives $kd$ choices as $|\vec{E}(H)| = kd$, and then enumerate the choices of $\vec{e}_{\ell}$ for $\tau_{\ell} = \originalretained$, with each contiguous segment corresponding to counting an open non-backtracking walk on $H$, giving a total factor of $(d-1)^{t-h}$. Therefore, using $m = \frac{n}{k}$, we get}
            &= \frac{1}{2t}\sum_{\substack{ (\tau_{\ell})_{\ell=1}^t:\\ h \colonequals \#\{\ell: \tau_{\ell} = \rematched\},\\ h \ge 1 }} \left(1 + O\left(\frac{D^2}{n}\right)\right)(d-1)^t (nd)^h \left(\left(1 - \frac{2r}{dn}\right)^{t-h}\left(\frac{2r}{(dn)^2}\right)^h + O\left(\frac{t^2}{n^{h+1}}\right)\right)\\
            &= \left(1 + O\left(\frac{D^2}{n}\right)\right)\cdot \frac{(d-1)^t}{2t} \sum_{\substack{ (\tau_{\ell})_{\ell=1}^t:\\ h \colonequals \#\{\ell: \tau_{\ell} = \rematched\},\\ h \ge 1 } } \left(\left(1 - \frac{2r}{dn}\right)^{t-h}\left(\frac{2r}{dn}\right)^h  + O\left(\frac{t^2(d-1)^t}{n}\right)\right)\\
            &= \left(1 + O\left(\frac{D^2}{n}\right)\right)\cdot \frac{(d-1)^t}{2t} \left(1 - \left(1 - \frac{2r}{dn}\right)^t\right) + O\left(\frac{t^2(2(d-1))^t}{n}\right)\\
            &= \frac{(d-1)^t}{2t} \left(1 - \left(1 - \frac{2r}{dn}\right)^t\right) + O\left(\frac{D^2(2(d-1))^t}{n}\right), \numberthis \label{eq:noise-mean}
        \end{align*}
        where $\frac{1}{2t}$ again accounts for the overcounting factor.
    \end{enumerate}

    Combining \eqref{eq:ret-mean} and \eqref{eq:noise-mean} and using \eqref{ineq:small-f-bound}, we get
    \begin{align*}
        \mu_t &= \sum_{\alpha \in I_t} p_{\alpha}\\
        &= \sum_{\substack{0 \le h \le t,\\ f = 2h}} \sum_{\substack{\alpha \in I_t:\\ f = c^{\del}(\alpha),\\
        h = \rho(\alpha)}} p_{\alpha} + \sum_{\substack{0 \le h \le t,\\ h \le f < 2h}} \sum_{\substack{\alpha \in I_t:\\ f = c^{\del}(\alpha),\\
        h = \rho(\alpha)}} p_{\alpha}\\
        &= \sum_{\substack{w\in \sW_t:\\ \tau_{\ell} = \originalretained,\\ \forall\, 1\le \ell \le t} } \sum_{\substack{\alpha \in I_t:\\ \cyc(\alpha) = w,\\ c^{\del}(\alpha) = 2\rho(\alpha) }} p_{\alpha} + \sum_{\substack{w\in \sW_t:\\ h = \rho(w) \ge 1 } } \sum_{\substack{\alpha \in I_t:\\ \cyc(\alpha) = w,\\ c^{\del}(\alpha) = 2\rho(\alpha) }} p_{\alpha}  + O\left(\frac{2D^2K_2^D}{n}\right)\\
        &= \left(1 - \frac{2r}{dn}\right)^{t}\frac{\tr(B^t)}{2t}  + \frac{(d-1)^t}{2t} \left(1 - \left(1 - \frac{2r}{dn}\right)^t\right)  \\
        &\quad + O\left(\frac{D^2}{n}\right) \frac{\tr(B^t)}{2t} + O\left(\frac{D^2(2(d-1))^t}{n}\right) + O\left(\frac{D^2K_2^D}{n}\right)\\
        &= \left(1 - \frac{2r}{dn}\right)^{t}\left(\frac{\tr(B^t)}{2t} - \frac{(d-1)^t}{2t}\right)  + \frac{(d-1)^t}{2t} + O\left(\frac{D^2K_4^D}{n}\right), \numberthis\label{eq:mu-mean}
    \end{align*}
    for a constant $K_4 = K_4(H) > 0$. This established the mean $\mu_t$ of the number of length-$t$ cycles.

    \subsection{Concluding the Proof}

    Taking $c_4 = c_4(H) > 0$ to be the minimum among $c_1, c_2$, and $c_3$ established in the previous sections, we get that for all $D \le c_4\log(n)$,
    \begin{align*}
        d_{\TV}\left(\sL(Z_1, Z_2, Z_3, \dots, Z_D),  \bigotimes_{t=1}^D \Pois(\mu_t)\right) \le 2(2b_1 + 2b_2 + b_3) = o(1), \quad \text{ as } n \to \infty,
    \end{align*}
    by Theorem~\ref{thm:chen-stein}, where the means are given in \eqref{eq:mu-mean} by
    \begin{align*}
        \mu_t \colonequals \left(1 - \frac{2r}{dn}\right)^{t}\left(\frac{\tr(B^t)}{2t} - \frac{(d-1)^t}{2t}\right)  + \frac{(d-1)^t}{2t} + O\left(\frac{D^2K_4^D}{n}\right), \quad \text{ for } 1 \le t \le D.
    \end{align*}
    Denoting $\delta \colonequals \frac{2r}{nd}$ and $\nu_t \colonequals \frac{\tr(B^t)}{2t}$, by triangle inequality, for $D \le c_4 \log(n)$ we get
    \begin{align*}
        &d_{\TV}\left(\sL(Z_1, Z_2, \dots, Z_D), \bigotimes_{t=1}^D \Pois\left((1-\delta)^t\left(\nu_t - \frac{(d-1)^t}{2t}\right) + \frac{(d-1)^t}{2t}\right)\right)\\
        &\le d_{\TV}\left(\sL(Z_1, Z_2, \dots, Z_D), \bigotimes_{t=1}^D \Pois\left(\mu_t\right)\right)\\
        &\quad + d_{\TV}\left(\bigotimes_{t=1}^D \Pois\left(\mu_t\right), \bigotimes_{t=1}^D \Pois\left((1-\delta)^t\left(\nu_t - \frac{(d-1)^t}{2t}\right) + \frac{(d-1)^t}{2t}\right)\right)\\
        &= o(1) + d_{\TV}\left(\bigotimes_{t=1}^D \Pois\left(\mu_t\right), \bigotimes_{t=1}^D \Pois\left((1-\delta)^t\left(\nu_t - \frac{(d-1)^t}{2t}\right) + \frac{(d-1)^t}{2t}\right)\right). \numberthis\label{ineq:Pois-TV-bound}
    \end{align*}

    Using the standard fact about KL divergence between two Poissons \[d_{\KL}(\Pois(a)\|\Pois(b)) = a\log\left(\frac{a}{b}\right) + b - a,\] and the tensorization of KL divergence for product distributions, we have
    \begin{align*}
        &d_{\KL}\left(\bigotimes_{t=1}^D \Pois\left(\mu_t\right)\bigg\| \bigotimes_{t=1}^D \Pois\left((1-\delta)^t\left(\nu_t - \frac{(d-1)^t}{2t}\right) + \frac{(d-1)^t}{2t}\right)\right)\\
        &= \sum_{t=1}^D d_{\KL}\left(\Pois(\mu_t) \bigg\| \Pois\left((1-\delta)^t\left(\nu_t - \frac{(d-1)^t}{2t}\right) + \frac{(d-1)^t}{2t}\right)\right)\\
        &= \sum_{t=1}^D \mu_t \log\left(\frac{\mu_t}{(1-\delta)^t\left(\nu_t - \frac{(d-1)^t}{2t}\right) + \frac{(d-1)^t}{2t}}\right) + (1-\delta)^t\left(\nu_t - \frac{(d-1)^t}{2t}\right) + \frac{(d-1)^t}{2t} -\mu_t\\
        &\le \mu_t\left(\frac{\mu_t - \left((1-\delta)^t\left(\nu_t - \frac{(d-1)^t}{2t}\right) + \frac{(d-1)^t}{2t}\right)}{(1-\delta)^t\left(\nu_t - \frac{(d-1)^t}{2t}\right) + \frac{(d-1)^t}{2t}}\right) + (1-\delta)^t\left(\nu_t - \frac{(d-1)^t}{2t}\right) + \frac{(d-1)^t}{2t} - \mu_t 
        \intertext{Now we use $\log(1+r) \le r$ to get}
        &= \sum_{t=1}^D \frac{\left(\mu_t - \left((1-\delta)^t\left(\nu_t - \frac{(d-1)^t}{2t}\right) + \frac{(d-1)^t}{2t}\right)\right)^2}{(1-\delta)^t\left(\nu_t - \frac{(d-1)^t}{2t}\right) + \frac{(d-1)^t}{2t}}
        \intertext{Plugging in $\mu_t = (1-\delta)^t\left(\nu_t - \frac{(d-1)^t}{2t}\right) + \frac{(d-1)^t}{2t} + O\left(\frac{D^2K_4^D}{n}\right)$, we get}
        &\le O\left(\frac{D^4 K_4^{2D}}{n^2}\right)\sum_{t=1}^D \frac{1}{(1-\delta)^t\left(\nu_t - \frac{(d-1)^t}{2t}\right) + \frac{(d-1)^t}{2t}}\\
        &\le O\left(\frac{D^4 K_4^{2D}}{n^2}\right)\sum_{t=1}^D \frac{2t}{(d-1)^t},
    \end{align*}
    which is $o(1)$ provided that $D \le c_5\log(n)$ for a constant $c_5 = c_5(H) > 0$. By Pinsker's inequality, for $D \le c_5\log(n)$,
    \begin{align*}
        &d_{\TV}\left(\bigotimes_{t=1}^D \Pois(\mu_t), \bigotimes_{t=1}^D \Pois\left((1-\delta)^t\left(\nu_t - \frac{(d-1)^t}{2t}\right) + \frac{(d-1)^t}{2t}\right)\right)\\
        &\le \sqrt{\frac{1}{2}d_{\KL}\left(\bigotimes_{t=1}^D \Pois(\mu_t)\bigg\| \bigotimes_{t=1}^D \Pois\left((1-\delta)^t\left(\nu_t - \frac{(d-1)^t}{2t}\right) + \frac{(d-1)^t}{2t}\right)\right)}\\
        &\le o(1). \numberthis \label{ineq:TV-Pois-mean-shift}
    \end{align*}

    Taking $c = c(H) > 0$ to be the minimum among $c_4$ and $c_5$ and plugging \eqref{ineq:TV-Pois-mean-shift} back to \eqref{ineq:Pois-TV-bound}, we get that for all $D \le c\log(n)$,
    \begin{align*}
        &d_{\TV}\left(\sL(Z_1, Z_2, \dots, Z_D), \bigotimes_{t=1}^D \Pois\left((1-\delta)^t\left(\nu_t - \frac{(d-1)^t}{2t}\right) + \frac{(d-1)^t}{2t}\right)\right)\\
        &\le o(1) + d_{\TV}\left(\bigotimes_{t=1}^D \Pois\left(\mu_t\right), \bigotimes_{t=1}^D \Pois\left((1-\delta)^t\left(\nu_t - \frac{(d-1)^t}{2t}\right) + \frac{(d-1)^t}{2t}\right)\right)\\
        &\le o(1).
    \end{align*}
    This finishes the proof of Theorem~\ref{thm:cycle-counts} in the non-bipartite case.

    The bipartite case follows the same proof essentially, with minor adjustments in the mean computation.

\section{Proof of Detection for Non-Ramanujan Noisy Lift}\label{sec:non-ramanujan}

In this section, we prove Theorem~\ref{thm:non-ramanujan}, which amounts to understanding $\tr(B^t)$ when $H$ is non-Ramanujan in light of Theorem~\ref{thm:cycle-counts}.

\begin{proof}
    First, we consider the non-bipartite case. Let $M$ be the adjacency matrix of $H$, and $B$ be the non-backtracking matrix of $H$. The Ihara-Bass formula in the presence of half-loops \cite[equation (1.1)]{friedman2014relativized} gives
    \begin{align*}
        \det(\lambda I - B) &=  (\lambda - 1)^{|E_1(H)|} (\lambda^2 - 1)^{|E_2(H)| - |V(H)|} \cdot \det(\lambda^2 I - \lambda M + (d-1)I),
    \end{align*}
    Thus, for any eigenvalue $\rho$ of $M$, the solutions $\lambda$ to
    \begin{align*}
        \lambda^2 - \lambda \rho + d - 1 = 0
    \end{align*}
    are eigenvalues of the non-backtracking matrix $B$.

    Since $H$ is non-Ramanujan, $H$ has an eigenvalue $\rho \ne d$ such that $|\rho| > 2\sqrt{d-1}$. Then, one corresponding solution $\lambda$ to 
    \begin{align*}
        \lambda^2 - \lambda \rho + d-1= 0
    \end{align*}
    then satisfies
    \begin{align*}
        |\lambda| = \max\left\{\frac{\left|\rho \pm \sqrt{\rho^2 - 4(d-1)}\right|}{2}\right\} > \sqrt{d-1}.
    \end{align*}
    Also note that $d-1$ is an eigenvalue of $B$. Consequently, $B$ has one eigenvalue $d-1$ and another eigenvalue $\lambda$ such that $|\lambda| > \sqrt{d-1}$, and we get that for any even $t \in \NN$,
    \begin{align*}
        \tr(B^t) \ge (d-1)^t + \lambda^{t}. 
    \end{align*}

    By Theorem~\ref{thm:cycle-counts}, there exists $c = c(H) > 0$ such that for any $D \le c\log(n)$, under $\sP^S$,
    \begin{align}
        d_{\TV}\left(\sL_{\sP^S}(Z_D), \Pois\left(b_D\right)\right) = o(1), \label{ineq:P-TV-bound}
    \end{align}
    while under $\sQ^S$,
    \begin{align}
        d_{\TV}\left(\sL_{\sQ^S}(Z_D), \Pois\left(a_D\right)\right) = o(1), \label{ineq:Q-TV-bound}
    \end{align}
    where $a_D \colonequals \frac{(d-1)^D}{2D}$ and $b_D \colonequals (1-\delta)^D \left(\frac{\tr(B^D)}{2D} - \frac{(d-1)^D}{2D}\right) + \frac{(d-1)^D}{2D}$. By Chebyshev's inequality,
    \begin{align*}
        &\Pr\left(\Pois(b_D) > \frac{a_D + b_D}{2}\right) + \Pr\left(\Pois(a_D) < \frac{a_D + b_D}{2}\right)\\
        &\le \frac{b_D}{\left(\frac{b_D - a_D}{2}\right)^2} + \frac{a_D}{\left(\frac{b_D - a_D}{2}\right)^2}\\
        &= \frac{4(b_D + a_D)}{(b_D - a_D)^2}\\
        &= \frac{4\left(\frac{(d-1)^D}{D} + (1-\delta)^D \left(\frac{\tr(B^D)}{2D} - \frac{(d-1)^D}{2D}\right)\right)}{(1-\delta)^{2D} \left(\frac{\tr(B^D)}{2D} - \frac{(d-1)^D}{2D}\right)^2}\\
        &\le \frac{4\left(\frac{(d-1)^D}{D} + (1-\delta)^D \frac{\lambda^D}{2D}\right)}{(1-\delta)^{2D} \frac{\lambda^{2D}}{4D^2}}\\
        &= 16D\left(\frac{d-1}{(1-\delta)^2\lambda^2}\right)^D + 8D\left(\frac{1}{(1-\delta) \lambda}\right)^D,
    \end{align*}
    for any even $D \in \NN$. Thus, by choosing $\delta = \delta(H) > 0$ small enough such that $(1-\delta)^2\lambda^2 > d-1$, we see that for any even $D = \omega(1)$,  
    \begin{align*}
        \Pr\left(\Pois(b_D) > \frac{a_D + b_D}{2}\right) + \Pr\left(\Pois(a_D) < \frac{a_D + b_D}{2}\right) = o(1),
    \end{align*}
    and thus
    \begin{align*}
        d_{\TV}\left(\Pois\left(b_D\right), \Pois\left(a_D\right)\right) = 1 - o(1). \numberthis\label{ineq:TV-Pois-D}
    \end{align*}

    Combining \eqref{ineq:TV-P}, \eqref{ineq:TV-Q}, and \eqref{ineq:TV-Pois-D}, we conclude that there exists $\Delta = \Delta(H) > 0$ such that for any even $D = \omega(1)$, the PTF that thresholds the number of length-$D$ cycle at an appropriate value achieves strong detection between $\sP^S$ and $\sQ^S$ for any $\delta \le \Delta$.

    The proof is essentially identical for the bipartite case, which we omit here. 
\end{proof}

\begin{remark}\label{rem:necessity-of-nonconstant-deg}
    It is easy to see that non-constant degree is necessary for symmetric PTF to achieve strong detection, as otherwise the products of Poisson distributions that the cycle counts converge to under both $\sP$ and $\sQ$ are contiguous.
\end{remark}

\section{Proof of Statistical Distinguishability}\label{sec:statistical}

Finally, we prove Theorem~\ref{thm:statistical} via a counting argument.

\begin{proof}
    We first consider the non-bipartite case, and work with a base $d$-regular graph $H$ that is not a single vertex with $d$ half-loops.

    Under $\sQ$, the number of potential stub-pair realization is $(dn - 1)!!$. Moreover, by Corollary~\ref{cor:simple-prob}, $\Theta((dn-1)!!)$ number of the stub-pair realizations correspond to simple graphs. Since each simple graph can be realized by $(d!)^n$ stub-pair realization (by permuting the $d$ stubs at each vertex), the total number of potential simple graphs under $\sQ^S$ is then
    \begin{align*}
        \Theta\left(\frac{(dn-1)!!}{(d!)^n}\right) = \Theta\left(\frac{\left(\frac{dn}{e}\right)^{\frac{dn}{2}}}{(d!)^n}\right). \numberthis \label{eq:num-Q}
    \end{align*}

    Under $\sP = \sP(H, \delta)$ and $r = \lceil \delta\frac{dn}{2} \rceil$, the number of potential stub-pair realizations of noisy random lift can be bounded by
    \begin{align*}
        &\binom{n}{\frac{n}{k}, \dots, \frac{n}{k}} \prod_{i=1}^k \binom{d}{M_{i,1}, M_{i,2}, \dots, M_{i,k}}^{n/k} \cdot \binom{dn}{2r}(2r-1)!!\\
        &\cdot \prod_{i=1}^k \left(\frac{n}{k}M_{i,i}-1\right)!! \prod_{1\le i < j \le k}\left(\frac{n}{k}M_{i,j}\right)! \left(\frac{2k}{n}\right)^r,
    \end{align*}
    where $\binom{n}{\frac{n}{k}, \dots, \frac{n}{k}}$ is the number of ways to assign the fiber partitions, $\binom{d}{M_{i,1}, M_{i,2}, \dots, M_{i,k}}$ is the number of ways to assign the stub types for a vertex in fiber $V_i$, $\binom{dn}{2r}$ is an upper bound on the number of ways to choose $2r$ stubs incident to $r$ deleted stub-pairs, $(2r-1)!!$ is the number of ways to form edges during the rematching process, and finally the last factor 
    \begin{align*}
        \prod_{i=1}^k \left(\frac{n}{k}M_{i,i}-1\right)!! \prod_{1\le i < j \le k}\left(\frac{n}{k}M_{i,j}\right)! \left(\frac{2k}{n}\right)^r
    \end{align*}
    is an upper bound on the number of possible realizations of the stub-pairs of the noiseless random lift that are not deleted among the $r$ chosen deleted stub-pairs. 
    
    To see this last factor is indeed an upper bound on the number of realization of the retained stub-pairs of the noiseless random lift, note that
    \begin{align*}
        \prod_{i=1}^k \left(\frac{n}{k}M_{i,i}-1\right)!! \prod_{1\le i < j \le k}\left(\frac{n}{k}M_{i,j}\right)!
    \end{align*}
    upper bounds the number of stub-pair realizations of the noiseless random lift before the deletion process. For each of the $r$ deleted edges, it removes a factor $\frac{n}{k}M_{i,j} - O(r) \ge \frac{n}{2k}$ from the product as long as $r = \lceil \delta\frac{dn}{2} \rceil$ is small compared to $\frac{n}{k}$.

    Consequently, for small enough $\delta > 0$, suppose $r = \delta \frac{dn}{2}$ is an integer, then the number of potential stub-pair realizations of noisy random lift is at most
    \begin{align*}
        &\binom{n}{\frac{n}{k}, \dots, \frac{n}{k}} \prod_{i=1}^k \binom{d}{M_{i,1}, M_{i,2}, \dots, M_{i,k}}^{n/k} \cdot \binom{dn}{2r}(2r-1)!!\\
        &\cdot \prod_{i=1}^k \left(\frac{n}{k}M_{i,i}-1\right)!! \prod_{1\le i < j \le k}\left(\frac{n}{k}M_{i,j}\right)! \left(\frac{2k}{n}\right)^r\\
        &\le k^n (d!)^n \prod_{i=1}^k \frac{\left(\frac{n}{k}M_{i,i}-1\right)!!}{(M_{i,i}!)^{\frac{n}{k}}} \prod_{1\le i < j \le k}\frac{\left(\frac{n}{k}M_{i,j}\right)!}{(M_{i,j}!)^{\frac{2n}{k}}} \cdot \binom{dn}{2r}(2r-1)!!\left(\frac{2k}{n}\right)^r\\
        &\le k^n (d!)^n n^{O(1)}\left(\frac{n}{ek}\right)^{\frac{dn}{2}} \left(\prod_{i=1}^k \prod_{j=1}^k\frac{M_{i,j}^{\frac{M_{i,j}}{2}}}{M_{i,j}!}\right)^{\frac{n}{k}} \left(\frac{edn}{2r}\right)^{2r} (2r)^r \left(\frac{2k}{n}\right)^r\\
        &= k^n (d!)^n n^{O(1)}\left(\frac{n}{ek}\right)^{\frac{dn}{2}} \left(\prod_{i=1}^k \prod_{j=1}^k\frac{M_{i,j}^{\frac{M_{i,j}}{2}}}{M_{i,j}!}\right)^{\frac{n}{k}} \left(\frac{2c^2d k}{\delta}\right)^{\delta \frac{dn}{2}}\\
        &\le \left(\frac{dn}{e}\right)^{\frac{dn}{2}} n^{O(1)} \left(\frac{2c^2d k}{\delta}\right)^{\delta \frac{dn}{2}} \left(k^{k\left(1 - \frac{d}{2}\right)} \frac{(d!)^k}{d^{\frac{dk}{2}}}\prod_{i=1}^k \prod_{j=1}^k\frac{M_{i,j}^{\frac{M_{i,j}}{2}}}{M_{i,j}!}\right)^{\frac{n}{k}}.
    \end{align*}

    First, we note that
    \begin{align*}
        \left(n^{O(1)} \left(\frac{2c^2d k}{\delta}\right)^{\delta \frac{dn}{2}}\right)^{\frac{k}{n}} = n^{O(\frac{k}{n})} \left(\frac{2c^2d k}{\delta}\right)^{\delta \frac{dk}{2}} \to 1, \text{ as } \delta \to 0.
    \end{align*}
    On the other hand, we may bound
    \begin{align*}
        (d!) \prod_{j=1}^k\frac{1}{M_{i,j}!} = \binom{d}{M_{i,1}, \dots, M_{i,k}} \le 2k^{d-2},\numberthis \label{ineq:bound-1}
    \end{align*}
    which comes from enumerating the first $d-2$ items $\in [k]$ in the multinomial theorem, leaving the last two items with only $2$ choices, and
    \begin{align*}
        \frac{d!}{d^{d}} \prod_{j=1}^k\frac{M_{i,j}^{M_{i,j}}}{M_{i,j}!} &= \binom{d}{M_{i,1}, \dots, M_{i,k}} \prod_{j=1}^k \left(\frac{M_{i,j}}{d}\right)^{M_{i,j}} < \frac{1}{2},\numberthis \label{ineq:bound-2}
    \end{align*}
    since we get the probability that the multinomial distribution $\textsf{Multi}(d; M_{i,1}/d, \dots, M_{i,k}/d)$ equals $(M_{i,1}, \dots, M_{i,k})$, which is at most the probability that a binomial distribution $\Bin(d, a/d)$ equals $a$ for some $1\le a \le d-1$, which is strictly less than $\frac{1}{2}$. Combining \eqref{ineq:bound-1} and \eqref{ineq:bound-2}, we get
    \begin{align*}
        &k^{k\left(1 - \frac{d}{2}\right)} \frac{(d!)^k}{d^{\frac{dk}{2}}}\prod_{i=1}^k \prod_{j=1}^k\frac{M_{i,j}^{\frac{M_{i,j}}{2}}}{M_{i,j}!}\\
        &< k^{k\left(1 - \frac{d}{2}\right)} \prod_{i=1}^k \sqrt{2k^{d-2} \frac{1}{2}}\\
        &= 1.
    \end{align*}

    As a result, then the number of potential stub-pair realizations of noisy random lift is at most
    \begin{align*}
        &\left(\frac{dn}{e}\right)^{\frac{dn}{2}} n^{O(1)} \left(\frac{2c^2d k}{\delta}\right)^{\delta \frac{dn}{2}} \left(k^{k\left(1 - \frac{d}{2}\right)} \frac{(d!)^k}{d^{\frac{dk}{2}}}\prod_{i=1}^k \prod_{j=1}^k\frac{M_{i,j}^{\frac{M_{i,j}}{2}}}{M_{i,j}!}\right)^{\frac{n}{k}}\\
        &\le \left(\frac{dn}{e}\right)^{\frac{dn}{2}} n^{O(1)} \left(n^{O(\frac{k}{n})} \left(\frac{2c^2d k}{\delta}\right)^{\delta \frac{dk}{2}} \cdot k^{k\left(1 - \frac{d}{2}\right)} \frac{(d!)^k}{d^{\frac{dk}{2}}}\prod_{i=1}^k \prod_{j=1}^k\frac{M_{i,j}^{\frac{M_{i,j}}{2}}}{M_{i,j}!} \right)^{\frac{n}{k}}\\
        &\le \left(\frac{dn}{e}\right)^{\frac{dn}{2}} n^{O(1)} \left(1 - \eps\right)^{\frac{n}{k}}
    \end{align*}
    for some $\eps = \eps(H) > 0$ when $\delta = \delta(H) > 0$ is small enough. 
    
    Dividing by $(d!)^n$, the number of possible simple graph realization of noisy random lift is then at most
    \begin{align*}
        \frac{\left(\frac{dn}{e}\right)^{\frac{dn}{2}}}{(d!)^n} n^{O(1)} \left(1 - \eps\right)^{\frac{n}{k}},
    \end{align*}
    which is exponentially smaller than the total number \eqref{eq:num-Q} of potential simple graphs under $\sQ^S$. Thus, most of the support of $\sQ^S$ is not supported on $\sP^S$, and we conclude for small enough constants $\delta = \delta(H) > 0$,
    \begin{align*}
        d_{\TV}(\sQ^S, \sP^S) = 1 - o(1), \quad \text{ as } n \to \infty.
    \end{align*}

    We omit the proof for the bipartite case, which follows a similar counting argument.
\end{proof}

\bibliographystyle{alpha}
\bibliography{main}

}
\end{document}